\documentclass[12pt]{amsart}
\usepackage{amsmath}
\usepackage{amsfonts}
\usepackage{amsthm}
\usepackage{amssymb}
\usepackage{enumerate}

\newtheorem{thm}{Theorem}[section]
\newtheorem{prop}[thm]{Proposition}
\newtheorem{cor}[thm]{Corollary}

\newtheorem{lem}[thm]{Lemma}

\theoremstyle{definition}
\numberwithin{equation}{section}

\newtheorem{rmk}[thm]{Remark}

\newcommand{\Iso}{{\mathrm{Iso}}}

\newcommand{\Zb}{\mathbb{Z}}

\newcommand{\reg}{{\mathrm{reg}}}

\begin{document}

\title{Shintani Lifts of Nearly Holomorphic Modular Forms}


\author{Yingkun Li, Shaul Zemel}

\begin{abstract}
In this paper, we compute the Fourier expansion of the Shintani lift of nearly holomorphic modular forms. As an application, we deduce modularity properties of generating series of cycle integrals of nearly holomorphic modular forms.
\end{abstract}

\maketitle


\section{Introduction}

For a discriminant $d\in\mathbb{Z}$, let $\mathcal{Q}_{d}$ be the set of binary quadratic forms of discriminant $d$, which is acted on by the group $\Gamma:=\operatorname{SL}_{2}(\mathbb{Z})$ with finitely many orbits. When $d<0$, each $\lambda\in\mathcal{Q}_{d}$ gives rise to a CM points $z_{\lambda}$ in the upper-half plane $\mathcal{H}$. The values of the $j$-function \[j(z):=\frac{1}{q}+744+196884q+\dots,\qquad q:=\mathbf{e}(z):=e^{2\pi iz}\] at such CM points are called \emph{singular moduli}, and they are algebraic numbers generating abelian extensions of the imaginary quadratic field $\mathbb{Q}(\sqrt{d})$ by the theory of complex multiplication.
The paper \cite{[Za]} proved the surprising result that the $d$\textsuperscript{th} trace of the normalized function $J(z):=j(z)-744$ is the $|d|$\textsuperscript{th} Fourier coefficient of a weakly holomorphic modular form $g$ of weight $\frac{3}{2}$.

When $d>0$, each $\lambda=[A,B,C]\in\mathcal{Q}_{d}$ gives rise to a geodesic \[c_{\lambda}:=\{z\in\mathcal{H}:A|z|^{2}+B\Re(z)+C=0\}\] on $\mathcal{H}$. If $d$ is not a perfect square, then the stabilizer $\Gamma_{\lambda}$ of $\lambda$ in $\Gamma$ is infinite and $c(\lambda):=\Gamma_{\lambda} \backslash c_{\lambda}$ is a closed cycle on the modular curve $Y=\Gamma\backslash\mathcal{H}$. Instead of values, one can consider integrals of modular forms along these cycles. The non-holomorphic Eisenstein series of weight 2, defined by \[E_{2}^{*}(z):=1-24\sum_{n\geq1}\sigma_{1}(n)q^{n}-\frac{3}{\pi y},\qquad y:=\Im z,\] offers an elegant example. For a fixed fundamental discriminant $\Delta<0$, let $\chi_{\Delta}$ be the genus character from, e.g., Section 1.2 of \cite{[GKZ]} (with $N=1$), which takes $\lambda\in\mathbb{Z}^{3}$ to
\begin{equation} \label{chiDelta}
\chi_{\Delta}(\lambda):= \begin{cases} \big(\frac{\Delta}{n}\big), & \text{if }\lambda\in\mathcal{Q}_{d}\text{ with }\Delta|d,\  (n,\Delta)=1,\text{ and }\lambda\text{ represents }n, \\ 0, & \text{otherwise}. \end{cases}
\end{equation}
Then for any fundamental discriminant $D<0$ co-prime to $\Delta$, we have the formula
\begin{equation} \label{E2cycle}
\sum_{\lambda\in\Gamma\backslash\mathcal{Q}_{\Delta D}}\chi_{\Delta}(\lambda)\int_{c(\lambda)}E^{*}_{2}(z)dz=-12H(-\Delta)H(-D),
\end{equation}
where $H(n)$ is the Hurwitz class number considered in \cite{[HZ]}. This twisted cycle integral is also the $|D|$\textsuperscript{th} Fourier coefficient of $12H(-\Delta)$ times the weight $\frac{3}{2}$ mock modular form studied loc.\ cit. In fact, this equality holds for any discriminant $D<0$ after suitably regularizing the left hand side (see Corollary 1.12 of \cite{[ANS]}\footnote{The different sign comes from the opposite orientation that they use---compare the formula on page 14 with our Equation \eqref{rlambda} below.}). Note that for a fundamental discriminant $D>0$, the twisted trace of singular moduli
\begin{equation} \label{Asingmod}
A(D,-\Delta):=\frac{1}{\sqrt{D}}\sum_{\lambda\in\Gamma\backslash\mathcal{Q}_{\Delta D},\ \lambda\gg0}\frac{\chi_{\Delta}(\lambda)}{|\Gamma_{\lambda}|}J(z_{\lambda})
\end{equation}
is the $D$\textsuperscript{th} Fourier coefficient of the weakly holomorphic modular form $f_{-\Delta}=q^{\Delta}+O(q)$ of weight $\frac{1}{2}$ from \cite{[Za]}. This coefficient is the same with $J(z)$ replaced by $j(z)$ when $D\Delta$ is not a square.

While searching for analogues of the algebraicity results above, Duke, Imamo\u{g}lu and T\'{o}th studied the generating series of untwisted cycle integrals of the $j$-function in \cite{[DIT]}, and showed that it is a mock modular form of weight $\frac{1}{2}$ whose shadow is the weight $\frac{3}{2}$ form $g$ from \cite{[Za]}. Furthermore, it is the first member of a family of mock modular forms with weakly holomoprhic shadows of weight $\frac{3}{2}$. In addition, the family also has a 0\textsuperscript{th} member $\mathbf{Z}_{+}$, whose Fourier coefficients are special values of Dirichlet $L$-functions of real quadratic fields (see Theorem 4 of \cite{[DIT]}). This function has a completion $\hat{\mathbf{Z}}_{+}$, which essentially first appeared in \cite{[DI]}. Its image under the weight $\frac{1}{2}$ Laplacian is $-\frac{1}{8\pi}$ times the Jacobi theta function $\theta(z):=\sum_{n\in\mathbb{Z}}q^{n^{2}}$. From these results, it is natural to ask what can be said about the twisted cycle integrals of the modular form $J(z)E_{2}^{*}(z)$. Our first result answers this question.
\begin{thm} \label{cycjE2}
Let $\Delta<0$ be a fixed fundamental discriminant. For any discriminant $D<0$, the twisted regularized cycle integral
\[\operatorname{Tr}_{\Delta,D}\big(J \cdot E_{2}^{*}\big):=\sum_{\lambda\in\Gamma\backslash\mathcal{Q}_{\Delta D}}\chi_{\Delta}(\lambda)\int_{c(\lambda)}^{\mathrm{reg}}J(z)E^{*}_{2}(z)dz,\] with the regularization defined as in Equation \eqref{Triota1}, is the $|D|$\textsuperscript{th} Fourier coefficient of the unique mock modular form $\tilde{f}_{-\Delta}$ of weight $\frac{3}{2}$ and level 4 in Kohnen's plus space which expands as $48|\Delta|H(-\Delta)+O(q^{3})$ with shadow $\frac{3}{2\pi}f_{-\Delta}$.
\end{thm}

The mock modular forms $\tilde{f}_{-\Delta}$ have previously been constructed by Jeon, Kang and Kim in \cite{[JKK1]} using Maass--Poincar\'{e} series. The sequel \cite{[JKK2]} expressed its Fourier coefficients, using the same approach as in \cite{[DIT]}, as cycle integrals of sesqui-harmonic modular forms of weight zero. Using a theta lift, Bruinier, Funke and Imamo\u{g}lu obtained another proof of the main result of \cite{[DIT]} in \cite{[BFI]}, which also gave a geometric interpretation of the Fourier coefficients with square indices. This idea was used by Alfes-Neumann and Schwagenscheidt in \cite{[ANS]} to construct $\tilde{f}_{-\Delta}$ as the holomoprhic part of the Shintani theta lift of a harmonic Maass form $\tilde{J}$ of weight 2, which expresses the Fourier coefficients of $\tilde{f}_{-\Delta}$ as the twisted cycle integrals of $\tilde{J}$. As the form $J \cdot E_{2}^{*}$ is not harmonic but just sesqui-harmonic of weight 2, it is not a priori clear its twisted cycle integrals coincide with the ones in \cite{[JKK2]} and \cite{[ANS]}.

To prove Theorem \ref{cycjE2}, we will follow the theta lift approach as in \cite{[BF1]}, \cite{[BFI]}, and \cite{[ANS]}, and use the theta kernel with the same archimedean Schwartz function as in \cite{[Sh]}. We shall apply it to nearly holomorphic modular forms, and compute the resulting Fourier expansions. Recall that a real-analytic modular form $f$ on $\mathcal{H}$ with at most linear exponential growth near the cusps is called \emph{nearly holomorphic} if it can be presented as
\begin{equation} \label{nearholdef}
f(z)=\sum_{l=0}^{p}\frac{f_{l}(z)}{y^{l}}\qquad\text{with}\qquad f_{l}:\mathcal{H}\to\mathbb{C}\text{ holomorphic for }0 \leq l \leq p
\end{equation}
for some $p\in\mathbb{N}$, which is called the \emph{depth} of $f$ if $f_{p}$ is not identically zero. In other words, it is annihilated by the operator $L_{z}^{p+1}$, where
\begin{equation} \label{defops}
\begin{split} R_{\kappa}&:=2i\partial_{z}+\tfrac{\kappa}{y},\qquad L=L_{z}:=-2iy^{2}\partial_{\overline{z}},\qquad\xi_{\kappa}=2iy^{\kappa}\overline{\partial_{\overline{z}}},\qquad\text{and} \\
\Delta_{\kappa}&:=-R_{\kappa-2}L_{z}=-\xi_{2-\kappa}\xi_{\kappa}=-4y^{2}\partial_{z}+\partial_{\overline{z}}+2i\kappa y\partial_{\overline{z}}=-y^{2}(\partial_{x}^{2}+\partial_{y}^{2})-\kappa y(\partial_{y}-i\partial_{x}) \end{split}
\end{equation}
are the weight raising operator of weight $\kappa$, the weight lowering operator, the $\xi$-operator of weight $\kappa$ from \cite{[BF1]}, and the Laplacian of weight $\kappa$. We denote the space of such modular forms of weight $\kappa$ with respect to $\Gamma$ by $\widetilde{M}_{\kappa}^{!}$, and use the superscript $\leq p$ to mean the subspace of forms with depth at most $p$. Since these differential operators commute with the slash operators, the condition of being nearly holomorphic is purely archimedean, and can be defined for any weight, Fuchsian group, character, representation, or multiplier system.

Nearly holomorphic modular forms of depth 0 are simply weakly holomorphic, and the Fourier expansions of their Shintani lifts have been computed in \cite{[Sh]}, \cite{[BFI]}, \cite{[BGK]}, and \cite{[ANS]}. For arbitrary $p\in\mathbb{N}$, we give the complete Fourier expansion of their Shintani lift in Theorem \ref{Shintot} below. In the Introduction we will just state a special case for scalar-valued modular forms of level 1, for which it is helpful to introduce a few notations. In the paper itself we work with vector-valued theta kernels and Shintani lifts.

Given $\lambda=[A,B,C]\in\mathcal{Q}_{d}$, denote $\lambda(z):=Az^{2}+Bz+C$. Suppose $f\in\widetilde{M}_{2k}^{!,\leq p}$ expands as
\begin{equation} \label{FourInt}
f(z)=\sum_{l=0}^{p}\sum_{n\in\mathbb{Z}}c(n,l)q^{n}y^{-l}.
\end{equation}
Given $d\in\mathbb{Z}$ that is not a square, we define, for $k$ even, the trace
\begin{equation} \label{TrdSL2Z}
\operatorname{Tr}_{d}(f):=\sum_{\lambda\in\Gamma\backslash\mathcal{Q}_{d}}\begin{cases} \frac{2}{|\Gamma_{\lambda}|}f(z_{\lambda}),& k=0,\ d<0, \\ \int_{c(\lambda)}f(z)\lambda(z)^{k-1} dz, & d>0,\ \sqrt{d}\not\in\mathbb{Z} \end{cases}
\end{equation}
(we shall not use the negative $d$ case for $k\neq0$). If $d=r^{2}>0$ with $r\in\mathbb{N}$, then we have $\Gamma\backslash\mathcal{Q}_{d}=\{\pm [0,r,j]|0 \leq j<r-1\}$ and we define the trace as \[\operatorname{Tr}_{d}(f):=2\lim_{T\to\infty}\sum_{j=0}^{r-1}\int_{\frac{(j,r)^{2}}{r^{2}}T^{-1}}^{T}f\big(-\tfrac{j}{r}+iy\big)(riy)^{k-1}idy+(2ir)^{k}\sum_{\substack{0 \leq l \leq p \\ n\leq0,\ r|n}}c(n,l)\phi_{n}(k-l,T;2\pi),\] where the function $\phi_{n}$ is defined in Equation \eqref{Singdef}. Finally, for $d=0$ we set
\begin{equation} \label{Trd0}
\operatorname{Tr}_{0}(f):=c(0,0)\operatorname{CT}_{s=1-k}\zeta(s),
\end{equation}
where $\zeta(s)$ is the Riemann zeta function.

We can now state the Fourier expansion of the Shintani lift of an even power of $E_{2}^{*}$.
\begin{thm} \label{E2klift}
Suppose $0<k\in\mathbb{N}$ is even. Consider the function taking $z\in\mathcal{H}$ to \[\begin{split} &\sum_{\substack{d\in\mathbb{N} \\ d\equiv0,1(\mathrm{mod\ }4)}}\sum_{b=0}^{k/2}\frac{\operatorname{Tr}_{d}\big((\tfrac{\pi}{3}E_{2}^{*})^{k-2b}\big)}{(16\pi y)^{b}(k-2b)!b!}q^{d}+2\sqrt{2\pi}\sum_{\substack{0>d\in\mathbb{Z} \\ d\equiv0,1(\mathrm{mod\ }4)}}
\frac{h_{k}\big(2\sqrt{2\pi|d|y}\big)}{\big(2\sqrt{2\pi|d|y}\big)^{k}}\cdot\frac{H(-d)q^{d}}{|d|^{\frac{1-k}{2}}}+ \\ -&\frac{\frac{\log(8\pi y)}{2}+C_{k}}{\big(\frac{k}{2}\big)!(16\pi y)^{k/2}}+\frac{\sqrt{2\pi}Q_{k-1}(0)}{6(8\pi y)^{(k-1)/2}}-2i^{k}\sum_{0<a\in\mathbb{N}}\frac{I_{k}\big(2a\sqrt{2\pi y}\big)-\tilde{\Omega}_{k}\big(2a\sqrt{2\pi y}\big)}{\big(2a\sqrt{2\pi y}\big)^{k}}a^{k}q^{a^{2}},\end{split}\] where $P_{k}$, $Q_{k-1}$, and $\tilde{\Omega}_{k}$ are the even rational polynomials given in Corollary \ref{Hermite}, Equation \eqref{Qnueval} and Remark \ref{tildeOmega} respectively, and the special functions $h_{k}$ and $I_{k}$ and the constant $C_{k}$ are defined in Equations \eqref{hndef}, \eqref{Ikdef} and \eqref{Cldef} respectively. Then this function is a real-analytic modular form of weight $k+\frac{1}{2}$ and level 4 in Kohnen's plus space. Furthermore, its image under $(-16\pi L_{z})^{k/2}$ is $2\pi\hat{\mathbf{Z}}_{+}(z)+\big(\frac{\gamma-\log\pi}{2}-2\log2+H_{k}\big)\theta(z)$, where $\gamma$ is the Euler--Mascheroni constant and $H_{k}$ is the $k$\textsuperscript{th} harmonic number.
\end{thm}

\begin{rmk} \label{E2kodd}
For odd $k\in\mathbb{N}$, one can obtain a similar result for twisted cycle integrals of $(E_{2}^{*})^{k}$ as in Equation \eqref{E2cycle}.
\end{rmk}

In general, when $f$ has singularity at the cusp, its Shintani lift has the following Fourier expansion, in which we recall that the power $R_{\kappa}^{n}$ of the weight raising operator from Equation \eqref{defops} means $R_{\kappa+2n-2} \circ R_{\kappa+2n-4} \circ \dots \circ R_{\kappa+2} \circ R_{\kappa}$.
\begin{thm} \label{liftnoc0k}
Let $f\in\widetilde{M}_{2k}^{!,\leq p}$ have the expansion from Equation \eqref{FourInt}, and suppose that $0<k\in\mathbb{N}$ is even and $c(0,k)=0$. Consider the expansion
\begin{align*} & \sum_{\substack{d\in\mathbb{N} \\ d\equiv0,1(\mathrm{mod\ }4)}}\sum_{b=0}^{\lfloor p/2 \rfloor}\frac{\operatorname{Tr}_{d}(L_{z}^{2b}f)}{(16\pi y)^{b}b!}q^{d}+\sqrt{2\pi}\sum_{\substack{0>d\in\mathbb{Z} \\ d\equiv0,1(\mathrm{mod\ }4)}}\sum_{l=k}^{p}\frac{h_{l}\big(2\sqrt{2\pi|d|y}\big)}{\big(2\sqrt{2\pi|d|y}\big)^{l}}\cdot\frac{\operatorname{Tr}_{d}(R_{2k-2l}^{l-k}L_{z}^{l}f)}{2^{l-k}(l-k)!}\cdot\frac{q^{d}}{|d|^{\frac{1-k}{2}}}+ \\ & +\!\sum_{\substack{l=k-1 \\ l\mathrm{\ odd}}}^{p}\!\frac{\big(\frac{l-1}{2}\big)!(-2)^{\frac{l-1}{2}}B_{l+1-k}(-1)^{\frac{k}{2}}c(0,l)}{(8\pi y)^{l/2}(2\pi)^{k-l-\frac{1}{2}}(l+1-k)!}-\sqrt{8\pi}\!\sum_{\substack{0<r\in\mathbb{N} \\ k \leq l \leq p \\ n<0,\ r|n}}\!\frac{J_{l}(2r\sqrt{2\pi y})}{\big(2r\sqrt{2\pi y}\big)^{l}}\cdot\frac{(2\pi n)^{l-k}}{(-1)^{\frac{k}{2}}}\cdot\frac{l!c(n,l)}{(l-k)!}r^{k}q^{r^{2}},
\end{align*}
where $B_{\kappa}$ is the $\kappa$\textsuperscript{th} Bernoulli number, and the special function $J_{l}$ is defined in Equation \eqref{Jnudef}. Then this expansion defines a real-analytic modular form of weight $k+\frac{1}{2}$ and level 4 in Kohnen's plus space, which is nearly holomorphic when $p<k$.
\end{thm}

\begin{rmk} \label{constk0}
Theorem \ref{liftnoc0k} holds also for $k=0$, once one adds to the expansion $2\sqrt{y}$ times the constant \[\int_{Y}f(z)d\mu(z):=\lim_{T\to\infty}\int_{Y_{T}}f(z)d\mu(z).\]
\end{rmk}

\begin{rmk} \label{QmodTr}
In the setting of Theorem \ref{liftnoc0k}, the generating series $\sum_{d}\operatorname{Tr}_{d}(f)q^{d}$ defines, when $p<k$, a quasi-modular form of weight $k+\frac{1}{2}$ and depth $\big\lfloor\frac{p}{2}\big\rfloor$. For $p \geq k$ this series can be completed to a such a modular form using these special functions---see Remark \ref{Qmod}.
\end{rmk}

The key ingredient to the calculation of the Fourier expansion is to find rapidly decaying higher order anti-derivatives of the Schwartz function used to construct the theta kernel. Such singular Schwartz functions are important also in evaluating singular theta lifts and constructing Green currents for special divisors on orthogonal and unitary Shimura varieties (see e.g.\ \cite{[FH]}). In our case, the first anti-derivative can be built from  the error function (see Equation \eqref{edef}). Surprisingly, the other higher order derivatives turn out to be constructed from linear combinations of the Gaussian and the error function with polynomial coefficients. These polynomials, which are defined in Equation \eqref{polsdef}, are closely related to the Hermite polynomials, and are of independent interest.

The paper is organized as follows. After recalling some basic notions in Section \ref{LatMF}, we devote Section \ref{SpecFunc} to study properties of these polynomials and related special functions, including their Fourier transforms, asymptotic behaviors and certain lattice sum evaluations. Then in Section \ref{NWHMF}, we complete the computations of the orbital integrals and the proof of the main theorem \ref{Shintot}, as well as its implications for Theorems \ref{cycjE2}, \ref{E2klift} and \ref{liftnoc0k}.

\textbf{Acknowledgement:} The authors are grateful to D. Zagier for suggesting the explicit value for the generating function $\Upsilon(\xi,t)$ in Equation \eqref{genser} below.

\section{Isotropic Lattices and Modular Forms \label{LatMF}}

This section introduces the notions and notation that are required for the rest of the paper. We follow the setup of \cite{[BF1]}, \cite{[BFI]}, \cite{[BFIL]}, \cite{[ANS]}, and others.

\subsection{Differential Operators on Modular Forms}

For $\gamma=\big(\begin{smallmatrix}a & b \\ c & d\end{smallmatrix}\big)\in\operatorname{SL}_{2}(\mathbb{R})$ and $z\in\mathcal{H}$, denote $j(\gamma,z):=cz+d$. Let $\operatorname{Mp}_{2}(\mathbb{R})$ denote the metaplectic double cover of $\operatorname{SL}_{2}(\mathbb{R})$, and let $\operatorname{Mp}_{2}(\mathbb{Z})$ be the inverse image of $\operatorname{SL}_{2}(\mathbb{Z})$ in $\operatorname{Mp}_{2}(\mathbb{R})$.

Given a representation $\rho$ of a finite index subgroup $\Gamma\subseteq\operatorname{Mp}_{2}(\mathbb{Z})$ on a finite dimensional complex vector space $V$, a function $f:\mathcal{H} \to V$ is called \emph{modular of weight $\kappa\in\frac{1}{2}\mathbb{Z}$ and representation $\rho$} if the functional equation \[f\mid_{\kappa}(M,\phi)(z):=\phi(z)^{-2\kappa}f(Mz)=\rho(M,\phi)f(z)\] holds for every element $(M,\phi)\in\Gamma$. Let $\mathcal{A}_{\kappa}^{!}(\Gamma,\rho)$ (resp. $\mathcal{A}_{\kappa}(\Gamma,\rho)$) denote the space of such functions that are real-analytic with at most exponential (resp. polynomial) growth near the cusps. It contains the subspaces $\widetilde{M}_{\kappa}^{!}(\Gamma,\rho), M_{\kappa}^{!}(\Gamma,\rho)$, $M_{\kappa}(\Gamma,\rho)$, and $S_{\kappa}(\Gamma,\rho)$ of nearly holomorphic, weakly holomorphic, holomorphic, and cusp forms respectively. We shall omit $\rho$ from the notation when it is trivial, as well as write $\mathcal{A}_{\kappa,\rho}^{!}$, and similarly for its subspaces, when $\Gamma=\operatorname{Mp}_{2}(\mathbb{Z})$.

The differential operators from Equation \eqref{defops} preserves modularity, in the sense that \[R_{\kappa}\mathcal{A}_{\kappa}^{!}(\Gamma,\rho)\subseteq\mathcal{A}_{\kappa+2}^{!}(\Gamma,\rho),\ L_{z}\mathcal{A}_{\kappa}^{!}(\Gamma,\rho)\subseteq\mathcal{A}_{\kappa-2}^{!}(\Gamma,\rho),\text{ and }\Delta_{\kappa}\mathcal{A}_{\kappa}^{!}(\Gamma,\rho)\subseteq\mathcal{A}_{\kappa}^{!}(\Gamma,\rho).\] It is known that $L_{z}$, $R_{\kappa}$, and $\Delta_{\kappa}$ preserve near holomorphicity, with $L_{z}$ decreasing the depth by 1, and $R_{\kappa}$ and $\Delta_{\kappa}$ usually increasing it by 1. For more on these modular forms, including their relations with quasi-modular forms and Shimura's vector valued modular forms, see \cite{[MR]}, \cite{[Ze3]}, and \cite{[Ze7]}.

Around a given point $w=s+it\in\mathcal{H}$, the natural local coordinate is
\begin{equation} \label{Awz}
\zeta=A_{w}(z):=\tfrac{z-w}{z-\overline{w}}\in\big\{\zeta\in\mathbb{C}\big||\zeta|<1\big\},\qquad\mathrm{with}\qquad1-A_{w}(z)=\tfrac{2it}{z-\overline{w}}
\end{equation}
for any $z\in\mathcal{H}$, which also satisfies \[|A_{\gamma w}(\gamma z)|=\Bigg|\frac{\overline{j(\gamma,w)}A_{w}(z)}{j(\gamma,w)}\Bigg|=|A_{w}(z)|\] for every $z$ and $w$ in $\mathcal{H}$ and $\gamma\in\operatorname{SL}_{2}(\mathbb{R})$. The expansion of a holomorphic modular form $f$ of weight $\kappa\in\mathbb{Z}$ is given by Proposition 17 of \cite{[BGHZ]}\footnote{Note that there is a small typo there, where the expansion in $4\pi yw$ should be in its additive inverse $-4\pi yw$.} as
\begin{equation} \label{ellexp}
f(z)=\bigg(\frac{2it}{z-\overline{w}}\bigg)^{\kappa}\sum_{n=0}^{\infty}R_{\kappa}^{n}f(w)\frac{t^{n}A_{w}(z)^{n}}{n!}=\big(1-A_{w}(z)\big)^{\kappa}\sum_{n=0}^{\infty}R_{\kappa}^{n}f(w)\frac{t^{n}A_{w}(z)^{n}}{n!}. \end{equation}
We shall need a formula extending Equation \eqref{ellexp} to nearly holomorphic modular forms.
\begin{lem} \label{ellnh}
For $f\in\widetilde{M}^{!}_{\kappa}(\Gamma,\rho)$ and a point $w=s+it\in\mathcal{H}$, we have the expansion \[f(z)=\big(1-A_{w}(z)\big)^{\kappa}\sum_{l=0}^{p}\frac{\big(1-\overline{A_{w}(z)}\big)^{l}}{t^{l}\big(1-\big|A_{w}(z)\big|^{2}\big)^{l}}\sum_{n=0}^{\infty} R_{\kappa-l}^{n}f_{l}(w)\frac{t^{n}A_{w}(z)^{n}}{n!}.\]
\end{lem}

\begin{proof}
We write $f(z)$ as in Equation \eqref{nearholdef}, and express each $f_{l}$ via Equation \eqref{ellexp}, but with $\kappa$ replaced by $\kappa-l$. Recalling from Lemma 5.1 of \cite{[Ze4]} that $y$ equals $\frac{t(1-|A_{w}(z)|^{2})}{|1-A_{w}(z)|^{2}}$, we get \[f(z)=\sum_{l=0}^{p}\frac{\big(1-A_{w}(z)\big)^{\kappa-l}\sum_{n=0}^{\infty}R_{\kappa-l}^{n}f_{l}(w)\frac{t^{n}A_{w}(z)^{n}}{n!}}{t^{l}\big(1-\big|A_{w}(z)\big|^{2}\big)^{l}}\big|1-A_{w}(z)\big|^{2l}.\]
Expanding $\big|1-A_{w}(z)\big|^{2}$ yields the desired result. This proves the lemma.
\end{proof}

For any $\epsilon>0$ we will denote the pre-image of the ball of radius $\epsilon$ in $\mathbb{C}$ under $A_{w}$ by $B_{\epsilon}(w)$ with the natural orientation on its boundary. We shall later need the limit value of the following integral, which is determined as follows.
\begin{cor} \label{intofexp}
Let $f$ and $w$ be as in Lemma \ref{ellnh}, and take an integer $\mu$. Then \[\lim_{\epsilon\to0}\int_{\partial B_{\epsilon}(w)}\frac{f(z)}{(1-A_{w}(z))^{\kappa-2}}A_{w}(z)^{\mu}dz=\begin{cases} -\frac{4\pi t^{|\mu|}}{(|\mu|-1)!}R_{\kappa}^{|\mu|-1}f(w), & \mu<0, \\  0, & \mu\geq0. \end{cases}\]
\end{cor}

\begin{proof}
The result follows from substituting in $\zeta=A_{w}(z)$, and thus $dz=\frac{2it}{(1-\zeta)^{2}}d\zeta$, inside Lemma \ref{ellnh}. This proves the corollary.
\end{proof}

We will carry out some integrations of modular forms on $\mathcal{H}$, with respect to the invariant measure $d\mu(z):=\frac{dz \wedge d\overline{z}}{-2iy^{2}}=\frac{dxdy}{y^{2}}$. The following standard consequence of Stokes' theorem will be useful for evaluating some of these integrals (see, e.g., Proposition 4.1.1 of \cite{[L]}).
\begin{lem} \label{Stokes}
Let $\mathcal{R}$ be a connected domain in $\mathcal{H}$ whose boundary $\partial\mathcal{R}$ is a piecewise smooth path in $\mathcal{H}$ (positively oriented), and assume that $f$, $g$, and $G$ are real-analytic functions on $\mathcal{R}$ such that $g=-L_{z}G$. Then we have the equality \[\int_{\mathcal{R}}f(z)g(z)d\mu(z)=\oint_{\partial\mathcal{R}}f(z)G(z)dz+\int_{\mathcal{R}}L_{z}f(z)G(z)d\mu(z).\]
\end{lem}

\subsection{Lattices Producing Modular Curves}

Let $V:=M_{2}(\mathbb{Q})^{0}$ be the signature  $(2,1)$ quadratic space of trace zero matrices over $\mathbb{Q}$ with quadratic form $Q(\lambda):=-N\det\lambda$ for some $0<N\in\mathbb{Q}$. Then $G:=\operatorname{Spin}(V)\cong\operatorname{SL}_{2}$, and the connected component of the symmetric space of $G$ containing the line spanned by $\big(\begin{smallmatrix}0 & -1 \\ 1 & 0\end{smallmatrix}\big)$ is identified with $\mathcal{H}$ via $z\mapsto\mathbb{R}Z^{\perp}(z)$ with
\[Z^{\perp}(z):=\frac{1}{\sqrt{N}y}\bigg(\begin{array}{cc}x & -|z|^{2} \\ 1 & -x\end{array}\bigg),\quad\mathrm{and\ we\ set}\quad Z(z):=\frac{1}{\sqrt{N}}\bigg(\begin{array}{cc}z & -z^{2} \\ 1 & -z\end{array}\bigg)=\frac{1}{\sqrt{N}}\binom{z}{1}(1\ \ -z).\] It is easy to check that $\gamma \cdot Z^{\perp}(z)=Z^{\perp}(\gamma z)$ and $\gamma \cdot Z(z)=j(\gamma,z)^{2}Z(\gamma z)$ for every $\gamma\in\operatorname{SL}_{2}(\mathbb{R})$.

Given $\lambda \in V_{\mathbb{R}}$ with $Q(\lambda)=-\xi^{2}<0$, we know that $\lambda=\xi Z^{\perp}(z_{\lambda})$ for some
$z_{\lambda}=x_{\lambda}+iy_{\lambda}\in\mathcal{H}$, with $\operatorname{sgn}(\xi)=-\operatorname{sgn}\big(\lambda,Z^{\perp}(z_{\lambda})\big)$.
In this case Lemma 4.2 of \cite{[Ze4]} proves the equalities
\begin{equation} \label{Zzw}
\begin{split} \big(\lambda,Z^{\perp}(z)\big)&=-2\xi\cosh d(z,z_{\lambda})=-2\xi\bigg(\frac{|z-{z_{\lambda}}|^{2}}{2yy_{\lambda}}+1\bigg)=-2\xi\frac{1+|A_{{z_{\lambda}}}(z)|^{2}}{1-|A_{{z_{\lambda}}}(z)|^{2}}, \\ \big(\lambda,Z(z)\big)&=-2\xi\frac{(z-{z_{\lambda}})(z-\overline{{z_{\lambda}}})}{2y_{\lambda}}=\frac{4\xi y_{\lambda}A_{{z_{\lambda}}}(z)}{\big(1-A_{{z_{\lambda}}}(z)\big)^{2}}, \end{split}
\end{equation}
where $d(z,z_{\lambda})$ is the hyperbolic distance between $z$ and $z_{\lambda}$.

Fix an even, integral lattice $L \subseteq V$, with its dual $L^{*}:=\operatorname{Hom}(L,\mathbb{Z})$ viewed as a subgroup of $V$ containing $L$, and $D_{L}:=L^{*}/L$ the associated finite quadratic module. We denote $\Gamma=\Gamma_{L} \subseteq G(\mathbb{Q})=\operatorname{SL}_{2}(\mathbb{Q})$ the inverse image of the discriminant kernel\footnote{For convenience we shall henceforth assume that $\Gamma\subseteq\operatorname{SL}_{2}(\mathbb{Z})$---see Remark \ref{GammaSL2Z} below.} of $L$, and set $Y:=Y_{L}:=\Gamma\backslash\mathcal{H}$ to be the associated (open) modular curve, with the projection map $\pi:\mathcal{H} \to Y$. For every $h \in D_{L}$ and $m\in\mathbb{Z}+Q(h)$, we denote
\begin{equation} \label{Lmhdef}
L_{m,h}:=\{\lambda \in L+h|Q(\lambda)=m\}.
\end{equation}
Typical examples can be found in \cite{[BO]}, \cite{[AE]}, \cite{[LZ]}, or \cite{[Ze6]}, e.g.,
\begin{equation} \label{LatSL2Z}
L=\big\{\big(\begin{smallmatrix}-B & C \\ -A & B\end{smallmatrix}\big)\big|A, B, C\in\mathbb{Z}\big\}\quad\mathrm{with}\quad Q=-\det\quad\mathrm{and}\quad\Gamma=\operatorname{SL}_{2}(\mathbb{Z}).
\end{equation}
Also let $\rho_{L}$ be the Weil representation associated with $L$, in which $\operatorname{Mp}_{2}(\mathbb{Z})$ operates on the vector space $\mathbb{C}[D_{L}]$, with the canonical basis $\{\mathfrak{e}_{h}\}_{h \in D_{L}}$ (see \cite{[Bo]}, \cite{[Sch]}, \cite{[Str]}, \cite{[Ze1]}, and others).

\subsection{Cusps and Geodesics}

The Baily--Borel completion $\mathcal{H}^{*}$ of $\mathcal{H}$ is obtained by adding the set $\operatorname{Iso}(V)\cong\mathbb{P}^{1}(\mathbb{Q})$ of isotropic lines in $V$. Let $\ell_{\infty}\in\operatorname{Iso}(V)$ be the line spanned by $u_{\infty}:=\big(\begin{smallmatrix} 0 & 1 \\ 0 & 0\end{smallmatrix}\big)$, and given $\ell\in\operatorname{Iso}(V)$, we take an element $\sigma_{\ell}\in\operatorname{SL}_{2}(\mathbb{Z})$ such that $\ell=\sigma_{\ell}\ell_{\infty}$, and set $u_{\ell}:=\sigma_{\ell}u_{\infty}$. If $\Gamma_{\ell}\subseteq\Gamma\subseteq\operatorname{SL}_{2}(\mathbb{Z})$ is the stabilizer of $\ell$, then there exists $\alpha_{\ell}\in\mathbb{N}$, called the \emph{width} of the cusp $\ell$, such that
\begin{equation} \label{widthdef}
\sigma_{\ell}^{-1}\Gamma_{\ell}\sigma_{\ell}=\big\{\pm\big(\begin{smallmatrix} 1 & n\alpha_{\ell} \\ 0 & 1\end{smallmatrix}\big)\big|n\in\mathbb{Z}\big\}.
\end{equation}
Let $0<\beta_{\ell}\in\mathbb{Q}$ be such that
\begin{equation} \label{betaelldef}
L\cap\ell=\mathbb{Z}\beta_{\ell}u_{\ell},\qquad\text{and set}\qquad\varepsilon_{\ell}:=\tfrac{\alpha_{\ell}}{\beta_{\ell}}.
\end{equation}
When $(L+h)\cap\ell\neq\emptyset$ we define $0 \leq k_{\ell,h}<\beta_{\ell}$ to be the unique number such that
\begin{equation} \label{komegaellh}
(L+h)\cap\ell(\mathbb{Z}\beta_{\ell}+k_{\ell,h})u_{\ell},\qquad\mathrm{and\ set\qquad} \omega_{\ell,h}:=\tfrac{k_{\ell,h}}{\beta_{\ell}}+\mathbb{Z}\in\mathbb{R}/\mathbb{Z}.
\end{equation}
All these parameters are constant on $\Gamma$-orbits.

Near the cusp associated with $\ell\in\operatorname{Iso}(V)$ we work with the coordinates
\begin{equation} \label{coorell}
z_{\ell}=x_{\ell}+iy_{\ell}:=\sigma_{\ell}^{-1}z\qquad\text{and}\qquad q_{\ell}(z_{\ell}):=\mathbf{e}\big(\tfrac{z_{\ell}}{\alpha_{\ell}}\big).
\end{equation}
For $\epsilon>0$ we define the neighborhood $B_{\epsilon}(\ell):=\big\{z\in\mathcal{H}\big||q_{\ell}(z_{\ell})<\epsilon\}$ of the cusp $\ell$. The set
\begin{equation} \label{trundom}
\mathcal{H}_{T}:=\mathcal{H}\setminus\bigcup_{\ell\in\operatorname{Iso}(V)}B_{e^{-2\pi T}}(\ell),\qquad\mathrm{for}\quad T>1,
\end{equation}
is $\Gamma$-invariant, and $Y_{T}:={\Gamma}\backslash\mathcal{H}_{T}$ is a truncated modular curve, with a fundamental domain\footnote{The fundamental domain actually depends on a choice of representatives for $\Gamma\backslash\operatorname{Iso}(V)$, but we suppress it from the notation since this choice does not affect the results later.}
\begin{equation} \label{funddomT}
\mathcal{F}_{T}(L):=\bigcup_{\ell\in\Gamma\backslash\operatorname{Iso}(V)}\sigma_{\ell}\mathcal{F}_{T}^{\alpha_{\ell}},\qquad\mathrm{where}\qquad \mathcal{F}_{T}^{\alpha}:=\bigcup_{j=0}^{\alpha-1}\big(\begin{smallmatrix} 1 & j \\ 0 & 1\end{smallmatrix}\big)\mathcal{F}_{T}
\end{equation}
is composed, for $\alpha\in\mathbb{N}$, of $\alpha$ translations of \[\mathcal{F}_{T}:=\big\{z=x+iy\in\mathcal{H}\big||x|\leq\tfrac{1}{2},\ |z|\geq1,\ y \leq T\}.\]

\begin{rmk} \label{GammaSL2Z}
The assumption $\Gamma=\Gamma_{L}\subseteq\operatorname{SL}_{2}(\mathbb{Z})$ is satisfied for the large family of lattices from \cite{[Ze6]}, but not for every lattice $L$ in $V$. However, the only places where we use the assumption that $\Gamma\subseteq\operatorname{SL}_{2}(\mathbb{Z})$ is in the form of the fundamental domain from Equation \eqref{funddomT}, with $T>1$ being a sufficient bound, and in the integrality of the parameter $\alpha_{\ell}$ from Equation \eqref{widthdef}. Since none of these facts are used in any proof below, our results hold equally well for more general lattices.
\end{rmk}

An element $\lambda \in V_{\mathbb{R}}$ with $Q(\lambda)>0$ defines a geodesic
\begin{equation} \label{geodef}
c_{\lambda}:=\big\{z\in\mathcal{H}\big|\big(\lambda,Z^{\perp}(z)\big)=0\big\}\subseteq\mathcal{H},\qquad\text{as well as}\qquad c(\lambda):=\Gamma_{\lambda} \backslash c_{\lambda} \subseteq Y,
\end{equation}
where $\Gamma_{\lambda}$ is the stabilizer of $\lambda$ in $\Gamma$. For $\lambda_{0}=\big(\begin{smallmatrix} 1 & 0 \\ 0 & -1\end{smallmatrix}\big) \in V$, we orient the geodesic $c_{\lambda_{0}}=(0,i\infty)$ to go up, and transfer this to an orientation on $c_{\lambda}$ and $c(\lambda)$ for each such $\lambda \in V_{\mathbb{R}}$ via the action of $\operatorname{SL}_{2}(\mathbb{R})$. We have the following well-known dichotomy.
\begin{lem} \label{splithyper}
Let $\lambda \in V$ be such that $m=Q(\lambda)>0$. If $m \in N\cdot(\mathbb{Q}^{\times})^{2}$ then $\Gamma_{\lambda}$ is the trivial subgroup $\{\pm I\}$, and the geodesic $c_{\lambda}$ connects two cusps in $\mathbb{P}^{1}(\mathbb{Q})$. Otherwise the image of $\Gamma_{\lambda}$ in $\operatorname{SO}^{+}(V)\cong\operatorname{PSL}_{2}(\mathbb{Q})$ is infinite cyclic.
\end{lem}
In the first case in Lemma \ref{splithyper} we call $\lambda$ \emph{split-hyperbolic}. For $m\in\mathbb{Q}$ and for $\lambda \in V$ we then set, by a slight abuse of notation,
\begin{equation} \label{iotash}
\iota(m):=\begin{cases} 1, & \text{if } \sqrt{m/N}\in\mathbb{Q}^{\times}, \\ 0, & \text{otherwise},\end{cases}\qquad\text{and}\qquad
\iota(\lambda):=\begin{cases} 1, & \text{if }\lambda\text{ is split-hyperbolic}, \\ 0, & \text{otherwise}. \end{cases}
\end{equation}
If $\iota(\lambda)=1$, then $\lambda^{\perp}$ is spanned by $\ell_{\lambda}$ and $\ell_{-\lambda}$ in $\Iso(V)$, which correspond to where $c_{\lambda}$ ends and begins respectively. If $Q(\lambda)=m$, then we have
\begin{equation}
\sigma_{\ell_{\lambda}}^{-1}\lambda=\sqrt{\tfrac{m}{N}}\big(\begin{smallmatrix} 1 & -2r_{\lambda} \\ 0 & -1\end{smallmatrix}\big)\quad\mathrm{for\ some}\quad r_{\lambda}\in\mathbb{Q},\quad\mathrm{with}\quad r_{\lambda}+\alpha_{\ell_{\lambda}}\mathbb{Z}\in\mathbb{Q}/\alpha_{\ell_{\lambda}}\mathbb{Z}\quad\mathrm{canonical}. \label{rlambda}
\end{equation}
The canonical image $r_{\lambda}+\alpha_{\ell_{\lambda}}\mathbb{Z}\in\mathbb{Q}/\alpha_{\ell_{\lambda}}\mathbb{Z}$ from Equation \eqref{rlambda}, which we shall henceforth still denote by just $r_{\lambda}$, is called the \emph{real part} of $c_{\lambda}$, and it is constant on $\Gamma$-orbits.

For $\ell\in\operatorname{Iso}(V)$, $m\geq0$, and $h \in D_{L}$ we set, for later use, the symbol
\begin{equation} \label{iotaellmh}
\iota_{\ell}(m,h):=\begin{cases} 1, & \text{there exists }\lambda \in L_{m,h}\cap\ell^{\perp},\ \text{positively oriented if } m>0, \\ 0, & \text{otherwise}. \end{cases}
\end{equation}
In other words, for $m>0$ we have $\iota_{\ell}(m,h)=1$ if and only if $\ell=\ell_{\lambda}$ for some $\lambda \in L_{m,h}$. The additive subgroup $L\cap\ell$ acts on $L_{m,h} \cap \ell^{\perp}$, and we have the following standard result.
\begin{lem}[Lemma 3.1 of \cite{[BFIL]}] \label{parLmhiota1}
For $\ell\in\operatorname{Iso}(V)$, $0<m\in\mathbb{Q}$, and $h \in D_{L}$ such that $\iota_{\ell}(m,h)=1$, the natural map $L_{m,h}\cap\ell^{\perp} \to (L_{m,h}\cap\ell^{\perp})/(L\cap\ell)$ factors through $\Gamma_{\ell}\backslash(L_{m,h}\cap\ell^{\perp})$. For every $\lambda \in L_{m,h}\cap\ell^{\perp}$ there are $2\sqrt{\frac{m}{N}}\varepsilon_{\ell}$ pre-images of $\lambda+(L\cap\ell)$ in $\Gamma_{\ell}\backslash(L_{m,h}\cap\ell^{\perp})$, namely the images of $\{\lambda+j\beta_{\ell}u_{\ell}|0 \leq j\leq2\sqrt{\frac{m}{N}}\varepsilon_{\ell}-1\}$ modulo $\Gamma_{\ell}$.
\end{lem}

\begin{rmk} \label{rporbs}
The number $2\sqrt{\frac{m}{N}}\varepsilon_{\ell}$ from Lemma \ref{parLmhiota1} is therefore integral. Moreover, for $\ell$, $m$, and $h$ as in Lemma \ref{parLmhiota1}, take some positively oriented $\lambda \in L_{m,h}\cap\ell^{\perp}$, and let $r_{\lambda}$ be as in Equation \eqref{rlambda}. Then we have \[\big\{r_{\mu}\big|\mu \in L_{m,h}\cap\ell^{\perp}\text{ positively oriented}\big\}=r_{\lambda}+\tfrac{\beta_{\ell}}{2}\sqrt{\tfrac{N}{m}}\mathbb{Z}\Big/\alpha_{\ell_{\lambda}}\mathbb{Z}\subseteq\mathbb{Q}/\alpha_{\ell_{\lambda}}\mathbb{Z},\] a set of $2\sqrt{\frac{m}{N}}\varepsilon_{\ell}$ evenly spaced elements of $\mathbb{Q}/\alpha_{\ell_{\lambda}}\mathbb{Z}$.
\end{rmk}

\subsection{Schwartz Forms, Theta Functions, and Shintani Lifts}

Given $k\in\mathbb{N}$, we can define the Schwartz function
\begin{equation} \label{Schwarzdef}
\tilde{\varphi}_{k}(\lambda;\tau,z):=\big(\lambda,Z(z)\big)^{k}\mathbf{e}\big[Q(\lambda)\tau+\big(\lambda,Z^{\perp}(z)\big)^{2}\tfrac{iv}{2}\big] \end{equation}
for $\lambda \in V_{\mathbb{R}}$ and $\tau=u+iv\in\mathcal{H}$, and construct the vector-valued theta function
\begin{equation} \label{Thetadef}
\Theta_{k,L}(\tau,z):=\sum_{h \in D_{L}}\Theta_{k,L,h}(\tau,z)\mathfrak{e}_{h},\quad\Theta_{k,L,h}(\tau,z):=\sqrt{v}\sum_{\lambda \in L+h}\tilde{\varphi}_{k}(\lambda;\tau,z).
\end{equation}
Theorem 4.1 in \cite{[Bo]} implies that for fixed $z\in\mathcal{H}$ we have $\Theta_{k,L}(\tau,z)\in\mathcal{A}_{k+\frac{1}{2},\rho_{L}}$, whereas for fixed $\tau\in\mathcal{H}$ it is easy to verify that $\Theta_{k,L}(\tau,z)\in\mathcal{A}_{-2k}(\Gamma)$. After collecting terms, we can use Equation \eqref{Lmhdef} to rewrite
\begin{equation} \label{expTheta}
\Theta_{k,L,h}(\tau,z)=\sqrt{v}\sum_{m\in\mathbb{Z}+Q(h)}\bigg[\sum_{\lambda \in L_{m,h}}\big(\lambda,Z(z)\big)^{k}e^{-\pi v(\lambda,Z^{\perp}(z))^{2}}\bigg]q_{\tau}^{m},\quad q_{\tau}:=\mathbf{e}(\tau).
\end{equation}

Recall that the (probabilists') Hermite polynomials are defined by \[\operatorname{He}_{n}(\xi):=(-1)^{n}e^{\xi^{2}/2}\big(\tfrac{d}{d\xi}\big)^{n}e^{-\xi^{2}/2}=\big(\xi-\tfrac{d}{d\xi}\big)^{n}\cdot1=\sum_{b=0}^{\lfloor n/2 \rfloor}\frac{(-1)^{b}n!}{b!(n-2b)!2^{b}}\xi^{n-2b}.\] Then for $\ell\in\operatorname{Iso}(V)$ and $k\in\mathbb{N}$ one defines the unary theta function
\begin{equation} \label{expwithiota}
\begin{split} \Theta_{k,\ell}(\tau)&:=\sum_{\lambda \in (L^{*}\cap\ell^{\perp})/(L^{*}\cap\ell)}\frac{\operatorname{He}_{k}\big(\sqrt{2\pi v}(\sigma_{\ell}^{-1}\lambda,\Im(Z(i)))\big)}{(2\pi v)^{k/2}}q_{\tau}^{Q(\lambda)}\sum_{\substack{h \in D_{L}\\ h+(L^{*} \cap \ell)/(L \cap \ell)=\lambda}}\mathfrak{e}_{h} \\ &=\sum_{h \in D_{L}}\sum_{\substack{0 \leq m\in\mathbb{Z}+Q(h) \\ \iota(m)=1}}a(\Theta_{k,\ell},m,h,v)q_{\tau}^{m}\mathfrak{e}_{h}\in\mathcal{A}_{k+\frac{1}{2}}(\rho_{L}),\qquad\text{with} \\ a&(\Theta_{k,\ell},m,h,v):=\frac{\operatorname{He}_{k}\big(2\sqrt{2\pi mv}\big)}{(2\pi v)^{k/2}}\begin{cases} \big(\iota_{\ell}(m,h)+(-1)^{k}\iota_{\ell}(m,-h)\big), & m>0, \\ \iota_{\ell}(0,h),& m=0. \end{cases} \end{split}
\end{equation}

\begin{rmk} \label{thetadiff}
The theta functions $\Theta_{Sh}(\tau,z)$ and $\Theta_{\ell,k}(\tau)$ from Equations (4.1) and (4.2) of \cite{[ANS]} correspond to $(-\sqrt{N}/y^{2})^{k+1}\overline{\Theta_{k+1,L}(\tau,z)}$, $(-i\sqrt{N})^{k}\overline{\Theta_{k,\ell}(\tau)}$, and Equation \eqref{expwithiota} respectively in our setting.
\end{rmk}

The function $\Theta_{k,\ell}$ from Equation \eqref{expwithiota} appears in the asymptotic expansion of the theta kernel $\Theta_{k,L}$ from \eqref{Thetadef}, as is given in the following result.
It is essentially part (2) of Proposition 4.2 of \cite{[ANS]}, which refers to Theorem 5.2 of \cite{[Bo]} for the proof, and can also be proved using properties of appropriate variants of the lattice sums from Equation \eqref{Gkldef}.
\begin{lem} \label{Thetanearell}
Given $\ell \in \operatorname{Iso}(V)$, there exists a constant $C_{\ell}>0$ such that \[(\Theta_{k,L}\mid_{2k,z}\sigma_{\ell})(\tau,z_{\ell})=\frac{i^{k}y_{\ell}^{k+1}}{\sqrt{N}\beta_{\ell}}\Theta_{k,\ell}(\tau)+O(e^{-C_{\ell}y_{\ell}^{2}}) \quad\mathrm{as}\quad y_{\ell}\to\infty.\]
\end{lem}
For $f\in\mathcal{A}_{2k}^{!}(\Gamma)$ we follow \cite{[ANS]} and \cite{[BFI]} (among others) to define its \emph{regularized Shintani lift}, using the fundamental domain from Equation \eqref{funddomT}, to be the theta integral
\begin{equation} \label{Shindef}
\mathcal{I}_{k,L}(\tau,f):=\sum_{\ell\in\Gamma\backslash\operatorname{Iso}(V)}\operatorname{CT}_{s=0}\bigg[\lim_{T\to\infty}\int_{\mathcal{F}_{T}^{\alpha_{\ell}}}(f\mid_{2k}\sigma_{\ell})(z_{\ell}) (\Theta_{k,L}\mid_{-2k}\sigma_{\ell})(\tau,z_{\ell})y_{\ell}^{-s}d\mu(z_{\ell})\bigg],
\end{equation}
which is an element of $\mathcal{A}_{k+\frac{1}{2},\rho_{L}}$. When the constant term of $f$ at every cusp is zero, the integral converges absolutely and no regularization is necessary (see Proposition 4.1 of \cite{[BF2]}).

\section{Special Functions \label{SpecFunc}}

In this section we define and study some special functions, which will be useful in evaluating the Shintani lift of nearly holomorphic modular forms below.

\subsection{Familiar Functions}

Let $\mathrm{g}(\xi)$ denote the Gaussian $e^{-\xi^{2}/2}$. For $\xi>0$ it has the anti-derivative
\begin{equation} \label{Gaussprim}
-\frac{\sqrt{\pi}}{\sqrt{2}}\cdot\operatorname{erfc}\big(\tfrac{\xi}{\sqrt{2}}\big)=-\int_{\xi}^{\infty}e^{-w^{2}/2}dw=-\int_{\xi^{2}/2}^{\infty}e^{-s}\frac{ds}{\sqrt{2s}}= -\frac{1}{\sqrt{2}}\Gamma\big(\tfrac{1}{2},\tfrac{\xi^{2}}{2}\big),
\end{equation}
where $\operatorname{erfc}$ is the complementary error function and $\Gamma(\mu,t)$ is the \emph{incomplete Gamma function} defined by \[\Gamma(\mu,t):=\int_{t}^{\infty}e^{-s}s^{\mu}\frac{ds}{s}\] for $t>0$. If $0<\mu\in\mathbb{N}$ then this formula is well-defined for every $t\in\mathbb{R}$, and for $0\leq\mu\in\mathbb{Z}$ it is meaningful for $t<0$ as follows: If $\mu=0$ then the integral is using the Cauchy principal value, and for smaller $\mu$ we employ repeated integration by parts. The explicit formulae are given by
\begin{equation} \label{GammamuZt}
\Gamma(\mu,t)= \begin{cases} e^{-t}(\mu-1)!\sum_{a=0}^{\mu-1}\frac{t^{a}}{a!}=e^{-t}t^{\mu}\big(1-\tfrac{d}{dt}\big)^{\mu-1}\tfrac{1}{t}, & \text{ when }0<\mu\in\mathbb{N}\text{ and }t\in\mathbb{R}, \\ \frac{(-1)^{\mu}}{|\mu|!}\bigg(\Gamma(0,t)+e^{-t}\sum_{a=0}^{|\mu|-1}\frac{a!}{(-t)^{a+1}}\bigg), & \text{ when }-\mu\in\mathbb{N},\text{ and }t\neq0, \end{cases}
\end{equation}
$\Gamma(0,t)$ can also be written as $-\operatorname{Ei}(-t)$ using the \emph{exponential integral} $\operatorname{Ei}(t):=-\int_{-t}^{\infty}e^{-w}\frac{dw}{w}$, and the equality
\begin{equation} \label{derGamma}
\tfrac{d}{dt}\Gamma(\mu,t)=-e^{-t}t^{\mu-1}
\end{equation}
holds whenever $\Gamma(\mu,t)$ is defined.

Modifying the anti-derivative from Equation \eqref{Gaussprim}, we now define
\begin{equation} \label{edef}
\mathrm{e}(\xi):=-\tfrac{\operatorname{sgn}(\xi)}{\sqrt{2}}\Gamma\big(\tfrac{1}{2},\tfrac{\xi^{2}}{2}\big)=-\operatorname{sgn}(\xi)\int_{|\xi|}^{\infty}e^{-w^{2}/2}dw\quad\mathrm{for}\quad\xi\neq0. \end{equation}
It decays rapidly as $|\xi|\to\infty$, but it is discontinuous at $\xi=0$ with the jump \[\lim_{\xi\to0^{+}}\mathrm{e}(\xi)-\lim_{\xi\to0^{-}}\mathrm{e}(\xi)=-\sqrt{2}\Gamma\big(\tfrac{1}{2},0\big)=-\sqrt{2\pi}.\] We therefore have, as distributions on $\mathbb{R}$, the equality
\begin{equation} \label{egdist}
\tfrac{d}{d\xi}\mathrm{e}(\xi)=\mathrm{g}(\xi)-\sqrt{2\pi}\cdot\delta(\xi),
\end{equation}
where $\delta(\xi)$ is the Dirac delta distribution. Our goal next is to find higher order anti-derivatives of $\mathrm{g}(\xi)$.

\subsection{Two Families of Polynomials}

First we will consider two sequences of polynomials in $\mathbb{Q}[\xi]$, which we denote by $P_{\nu}$ and $Q_{\nu}$ with $\nu\in\mathbb{N}$ and defined recursively as follows. Set
\begin{equation} \label{polsdef}
\begin{split} P_{0}(\xi) &=1\quad\text{and}\quad Q_{0}(\xi)=0, \qquad\text{as well as} \\ P_{\nu}'(\xi) &=P_{\nu-1}(\xi)\quad\text{and}\quad P_{\nu}(\xi)+Q_{\nu}'(\xi)-\xi Q_{\nu}(\xi)=Q_{\nu-1}(\xi)\text{ for }\nu\geq1.
\end{split}
\end{equation}
The fact that Equation \eqref{polsdef} defines unique sequences of polynomials, and their parity properties, are established via the following lemma.
\begin{lem} \label{primfunc}
Let $p$ and $q$ be two polynomials in $\mathbb{Q}[\xi]$. Then there is a unique pair $(P,Q)$ of polynomials $P$ and $Q$ such that $P'=p$ and $P+Q'-\xi Q=q$. Moreover, if $p$ and $q$ have opposite parities then so do $P$ and $Q$, with that of $P$ (resp.\ $Q$) coinciding with that of $q$ (resp.\ $p$). In addition, if $p\neq0$ has the leading coefficient $r$ and $\deg q\leq\deg p$, then $\deg P=\deg Q+1=\deg p+1$ and the leading coefficients of both $P$ and $Q$ are $\frac{r}{\deg P}$.
\end{lem}

\begin{proof}
For each $\mu\geq1$, there exists a unique $p_{\mu}\in\mathbb{Q}[\xi]$ of degree $\mu$ and leading coefficient $-1$ such that $\int\mathrm{g}(\xi)p_{\mu}(\xi)d\xi=\xi^{\mu-1}\mathrm{g}(\xi)$, with $\mathrm{g}$ the Gaussian. Therefore there exist a unique $c\in\mathbb{Q}$ and $Q\in\mathbb{Q}[\xi]$ such that $\int\mathrm{g}(\xi)\big(q(\xi)-\tilde{P}(\xi)+c\big)d\xi=\mathrm{g}(\xi)Q(\xi)$, where $\tilde{P}\in\mathbb{Q}[\xi]$ satisfies $\tilde{P}(0)=0$ and $\tilde{P}'=p$. Setting $P:=\tilde{P}-c$ proves the first assertion. Integrating from $-\infty$ to $\infty$ shows that $c=0$ when $p$ is even and $q$ is odd. The last two assertions can now be checked from this construction. This proves the lemma.
\end{proof}

\begin{cor} \label{PnuQnu}
Equation \eqref{polsdef} defines unique sequences $\{P_{\nu}\}_{\nu=0}^{\infty}$ and $\{Q_{\nu}\}_{\nu=0}^{\infty}$ of polynomials in $\xi$. Moreover, $P_{\nu}$ is a polynomial of degree $n$ with leading coefficient $\frac{1}{\nu!}$ and parity $(-1)^{\nu}$ for any $n\geq0$, and $Q_{\nu}$ is a polynomial of degree $\nu-1$ with leading coefficient $\frac{1}{\nu!}$ and parity $(-1)^{\nu-1}$ for any $n\geq1$.
\end{cor}

\begin{proof}
All the statements hold for $\nu=0$ by Equation \eqref{polsdef}, and once they hold for $\nu-1$, taking $p=P_{\nu-1}$ and $q=Q_{\nu-1}$ in Lemma \ref{primfunc} determines $P_{\nu}$ and $Q_{\nu}$ as $P$ and $Q$ respectively, with the required properties. This proves the corollary.
\end{proof}

It will be useful for us to consider the (ordinary) generating series
\begin{equation} \label{genser}
\Psi(\xi,t):=\sum_{\nu=0}^{\infty}P_{\nu}(\xi)t^{\nu}\qquad\text{and}\qquad\Upsilon(\xi,t):=\sum_{\nu=0}^{\infty}Q_{\nu}(\xi)t^{\nu}.
\end{equation}
They converge for every $t$ and $\xi$ because of the degrees and leading coefficients of $P_{\nu}$ and $Q_{\nu}$ and can be characterized as follows.
\begin{prop} \label{PsiUpsBchar}
The functions $\Psi=\Psi(\xi,t)$ and $\Upsilon=\Upsilon(\xi,t)$ are the only functions whose Taylor expansions in $t$ is based on polynomials in $\xi$ and which satisfy the equalities $\Psi(\xi,0)=1$ and $\Upsilon(\xi,0)=0$ for every $\xi$ and the differential equations \[\partial_{\xi}\Psi(\xi,t)=t\cdot\Psi(\xi,t)\quad\mathrm{and}\quad \Psi(\xi,t)+\partial_{\xi}\Upsilon(\xi,t)-\xi\cdot\Upsilon(\xi,t)=1+t\cdot\Upsilon(\xi,t).\]
\end{prop}

\begin{proof}
Assume that $\Psi$ and $\Upsilon$ satisfies these properties, and write their Taylor series in $t$ as in Equation \eqref{genser}, with polynomial coefficients $\tilde{P}_{\nu}(\xi)$ and $\tilde{Q}_{\nu}(\xi)$ respectively. Now, the equalities with $t=0$ imply that $\tilde{P}_{\nu}$ and $\tilde{Q}_{\nu}$ satisfy the condition for $\nu=0$ in Equation \eqref{polsdef}, and also explain the existence of the term 1 in the second differential equation (set $t=0$ there). In addition, comparing the coefficient of $t^{\nu}$ with $\nu\geq1$ in the series resulting from substituting these expansions into the differential equations yields the other part of Equation \eqref{polsdef}. Hence $\{\tilde{P}_{\nu}\}_{\nu=0}^{\infty}$ and $\{\tilde{Q}_{\nu}\}_{\nu=0}^{\infty}$ are sequences satisfying that equation, so that $\tilde{P}_{\nu}=P_{\nu}$ and $\tilde{Q}_{\nu}=Q_{\nu}$ for every $\nu\in\mathbb{N}$ by Corollary \ref{PnuQnu}. This proves the proposition.
\end{proof}
Concerning the polynomiality of the coefficients, see Remark \ref{Taylorpol} below. Proposition \ref{PsiUpsBchar} allows us to determine the series $\Psi$ and $\Upsilon$ explicitly.
\begin{thm} \label{PsiUpsBexp}
The functions $\Psi$ and $\Upsilon$ from Equation \eqref{genser} are given by \[\Psi(\xi,t)=e^{\xi t+t^{2}/2}\qquad\mathrm{and}\qquad\Upsilon(\xi,t)=e^{(\xi+t)^{2}/2}\int_{\xi}^{\xi+t}e^{-w^{2}/2}dw=e^{\xi t+t^{2}/2}\int_{0}^{t}e^{-\xi w-w^{2}/2}dw.\] They also satisfy the differential equations \[(\partial_{t}-\xi)\Psi(\xi,t)=t\Psi(\xi,t)\qquad\mathrm{and}\qquad(\partial_{t}-\xi)\Upsilon(\xi,t)=t\Upsilon(\xi,t)+1.\]
\end{thm}

\begin{proof}
It suffices to show that the asserted series have the properties from Proposition \ref{PsiUpsBchar}. The equalities with $t=0$ are immediate, and it is easy to see that $\Psi(\xi,t)$ satisfies the first differential equation there and that its expansion in $t$ involves polynomials in $\xi$ as its Taylor coefficients. In addition, simple differentiation shows that \[\partial_{\xi}\Upsilon(\xi,t)=(\xi+t)\Upsilon(\xi,t)+e^{(\xi+t)^{2}/2}\big[e^{-(\xi+t)^{2}/2}-e^{-\xi^{2}/2}\big]=(\xi+t)\Upsilon(\xi,t)+1-\Psi(\xi,t),\] from which the second differential equation from that proposition quickly follows.

It remains to verify that the Taylor expansion of $\Upsilon(\xi,t)$ in $t$ consists of polynomials in $\xi$. This is equivalent to $\partial_{t}^{n}\Upsilon(\xi,t)\big|_{t=0}$ being a polynomial in $\xi$ for every $n$. To see this, we first evaluate $\partial_{t}\Upsilon(\xi,t)$ as $1+(\xi+t)\Upsilon(\xi,t)$, yielding the remaining differential equation as well. Simple induction now shows that for every $n$ the derivative $\partial_{t}^{n}\Upsilon(\xi,t)$ is a polynomial in $\xi+t$ plus another polynomial in $\xi+t$ times $\Upsilon(\xi,t)$. After substituting $t=0$, the fact that $\Upsilon(\xi,0)=0$ yields the desired assertion. This proves the theorem.
\end{proof}

\begin{rmk} \label{Taylorpol}
Solving the differential equation from Proposition \ref{PsiUpsBchar} implies that there exist functions $\psi(t)$ and $\phi(t)$ such that $\Psi(\xi,t)$ and $\Upsilon(\xi,t)$ are
\[\psi(t)e^{\xi t+t^{2}/2}\quad\mathrm{and}\quad e^{(\xi+t)^{2}/2}\Bigg[\phi(t)+\int_{\xi}^{\xi+t}e^{-w^{2}/2}dw+\big(\psi(t)-1\big)\int_{\xi}^{\infty}e^{-w^{2}/2}dw\Bigg]\] respectively, and that $\psi(0)=1$ and $\phi(0)=0$. Showing that the Taylor expansions involve only polynomials in $\xi$ if and only if $\psi$ and $\phi$ are the constant functions with the respective values, seems, however, not very straightforward.
\end{rmk}

\begin{rmk} \label{PQrecur}
The differential equations from Theorem \ref{PsiUpsBexp} also imply that \[\nu P_{\nu}(\xi)-\xi P_{\nu-1}(\xi)=P_{\nu-2}(\xi)\quad\mathrm{and}\quad\nu Q_{\nu}(\xi)- \xi Q_{\nu-1}(\xi)=Q_{\nu-2}(\xi)\quad\mathrm{for\ all}\quad\nu\geq2.\]
\end{rmk}

For the explicit expressions for the polynomials from Equation \eqref{polsdef}, recall that
\begin{equation} \label{Hermitegen}
e^{\xi t-t^{2}/2}=\sum_{\nu=0}^{\infty}\operatorname{He}_{\nu}(\xi)\frac{t^{\nu}}{\nu!}.
\end{equation}
Theorem \ref{PsiUpsBexp} then implies the following result.
\begin{cor} \label{Hermite}
For every $\nu\in\mathbb{N}$ we have
\[P_{\nu}(\xi)=\frac{(-i)^{\nu}}{\nu!}\operatorname{He}_{\nu}(i\xi):=\frac{e^{-\xi^{2}/2}}{\nu!}\tfrac{d^{\nu}}{d\xi^{\nu}}e^{\xi^{2}/2}=\big(\xi+\tfrac{d}{d\xi}\big)^{\nu}\cdot\frac{1}{\nu!}= \sum_{a=0}^{\lfloor\nu/2\rfloor}\frac{\xi^{\nu-2a}}{a!(\nu-2a)!2^{a}}.\] In particular, $P_{\nu}$ is a polynomial of degree $\nu$ and parity $(-1)^{\nu}$ such that $\nu!P_{\nu}\in\mathbb{Z}[\xi]$ and $P_{\nu}(0)$ is $\frac{1}{2^{\nu/2}(\nu/2)!}$ for even $\nu$ and 0 for odd $\nu$.
\end{cor}

Note that the proof of Theorem \ref{PsiUpsBexp} yields the equality $\partial_{t}\Upsilon(\xi,t)=\partial_{\xi}\Upsilon(\xi,t)+\Psi(\xi,t)$, from which we deduce that $Q_{\nu+1}(\xi)=\frac{Q_{\nu}'(\xi)+P_{\nu}(\xi)}{\nu+1}$, and using Equation \eqref{polsdef} we get
\begin{equation} \label{Qnueval}
Q_{\nu}(\xi)=\sum_{a=0}^{\nu-1} \frac{(\nu-1-a)!}{\nu!}\tfrac{d^{a}}{d\xi^{a}}P_{\nu-1-a}(\xi)=\sum_{a=0}^{\lfloor (\nu-1)/2 \rfloor}\frac{(\nu-1-a)!}{\nu!}P_{\nu-1-2a}(\xi),
\end{equation}
with $P_{\mu}(\xi)$ given in Corollary \ref{Hermite}. It is thus indeed a polynomial of degree $\nu-1$ and parity $(-1)^{\nu-1}$ such that $\nu!Q_{\nu}\in\mathbb{Z}[\xi]$. It follows that $Q_{\nu}(0)=0$ for even $\nu$, while if $\nu$ is odd then
\begin{equation} \label{Qnu0}
Q_{\nu}(0)=\frac{2^{\frac{\nu-1}{2}}\big(\frac{\nu-1}{2}\big)!}{\nu!}=\prod_{j=0}^{(\nu-1)/2}\frac{1}{2j+1}.
\end{equation}

\begin{rmk} \label{Qneg}
The recursion \eqref{polsdef} extends naturally to $\nu\in\mathbb{Z}$, in which case $P_{\nu}(\xi)=0$ and \[Q_{\nu}(\xi)=(-1)^{1-\nu}\operatorname{He}_{-1-\nu}(\xi)\] for $\nu\leq-1$. We take these as the definitions of $P_{\nu}$ and $Q_{\nu}$ in these cases. Then Remark \ref{PQrecur} extends to all $\nu\in\mathbb{Z}$.
\end{rmk}

The polynomials $\operatorname{He}_{\nu}$ form an \emph{Appell sequence}, and Corollary \ref{Hermite} shows that the same applies also for the polynomials $\nu!P_{\nu}$. This means explicitly that the equalities
\begin{equation} \label{Appell}
\operatorname{He}_{\nu}(\xi+\eta)=\sum_{j=0}^{\nu}\binom{\nu}{j}\eta^{j}\operatorname{He}_{\nu-j}(\xi)\quad\text{and}\quad P_{\nu}(\xi+\eta)=\sum_{j=0}^{\nu}\frac{\eta^{j}P_{\nu-j}(\xi)}{j!}
\end{equation}
hold. This either follows from the generating function $\Psi(\xi,t)$ (as well as the exponential generating function $e^{\xi t-t^{2}/2}$) being $e^{\xi t}$ times a function of $\xi$ alone (see, e.g., \cite{[Ze5]}, even though this was known much earlier), or by simple computations using the explicit formulae. A change of variable in Equation \eqref{Appell} produces, for every $l\in\mathbb{N}$, the equality
\begin{equation} \label{PlTaylor}
\sum_{\nu=0}^{l}\frac{(-1)^{\nu}}{(l-\nu)!}(\xi+\zeta)^{l-\nu}P_{\nu}(\xi)=P_{l}(\zeta)\in\mathbb{Q}[\xi,\zeta].
\end{equation}

\subsection{Auxiliary Polynomials}

We now define a few other families of polynomials, which will appear later in the Fourier expansion of the Shintani lift.
\begin{lem} \label{Pilprop}
For any $l\in\mathbb{N}$, the polynomial \[\Pi_{l}(\xi,\zeta):=\sum_{\nu=0}^{l}\frac{(-1)^{\nu}}{(l-\nu)!}(\xi+\zeta)^{l-\nu}Q_{\nu}(\xi)\in\mathbb{Q}[\xi,\zeta]\] has degree $l-1$ in $\xi$, and it satisfies the equality $\Pi_{l}(-\zeta,\zeta)=-Q_{l}(\zeta)$. We also have \[\partial_{\zeta}\Pi_{l}(\xi,\zeta)=\Pi_{l-1}(\xi,\zeta)\qquad\text{and}\qquad\partial_{\xi}\big(\Pi_{l}(\xi,\zeta)\mathrm{g}(\xi)\big)=\bigg(\frac{(\xi+\zeta)^{l}}{l!}-P_{l}(\zeta)\bigg)\mathrm{g}(\xi)\] for every $l\in\mathbb{N}$, and the generating series \[\sum_{l=0}^{\infty}\Pi_{l}(\xi,\zeta)t^{l}=-e^{t^{2}/2+\zeta t}\int_{0}^{t}e^{\xi w-w^{2}/2}dw.\]
\end{lem}

\begin{proof}
The degree in $\xi$ and the value of $\Pi_{l}(-\zeta,\zeta)$ are immediate from the definition. Substituting the definition of $\Pi_{l}(\xi,\zeta)$ inside the generating series and setting produces \[\sum_{l=0}^{\infty}\sum_{\nu=0}^{l}\frac{(-1)^{\nu}t^{l}}{(l-\nu)!}(\xi+\zeta)^{l-\nu}Q_{\nu}(\xi)=\sum_{\nu=0}^{\infty}\sum_{\mu=0}^{\infty}\frac{t^{\mu}}{\mu!}(\xi+\zeta)^{\mu}(-t)^{\nu}Q_{\nu}(\xi)=e^{t(\xi+\zeta)} \Upsilon(\xi,-t)\] (with $\mu=l-\nu\geq0$), where we have substituted Equation \eqref{genser}. The value of the series, which we denote by $\Phi(\xi,\zeta,t)$, now follows from Theorem \ref{PsiUpsBexp}. One checks directly that this series satisfies the differential equations \[\partial_{\zeta}\Phi(\xi,\zeta,t)=t\Phi(\xi,\zeta,t)\text{ and }(\partial_{\xi}-\xi)\Phi(\xi,\zeta,t)=e^{t^{2}/2+\zeta t}\!\int_{0}^{t}\!(\xi-w)e^{\xi w-w^{2}/2}dw=e^{(\xi+\zeta)t}-\Psi(\zeta,t)\] (using Theorem \ref{PsiUpsBexp} again), from which the two required equalities follow for every $l$ after expanding everything in $t$. This proves the lemma.
\end{proof}

\begin{rmk} \label{tildePil}
Write the sum $\Pi_{l}(\xi,\zeta)+Q_{l}(\zeta)$ as $\tilde{\Pi}_{l}(\xi+\zeta,\zeta)$ for some polynomial $\tilde{\Pi}_{l}$. The equality $\Pi_{l}(-\zeta,\zeta)=-Q_{l}(\zeta)$ from Lemma \ref{Pilprop} implies that $\tilde{\Pi}_{l}(\omega,\zeta)\in\omega\mathbb{Q}[\omega,\zeta]$.
\end{rmk}

Using the polynomial $\tilde{\Pi}_{l}$ from Remark \ref{tildePil}, we define
\begin{equation} \label{Eldef}
E_{l}(\zeta):=\frac{1}{\sqrt{2\pi}}\int_{-\infty}^{\infty}\frac{\tilde{\Pi}_{l}(\xi+\zeta,\zeta)}{\xi+\zeta}\mathrm{g}(\xi)d\xi\in\mathbb{R}[\zeta].
\end{equation}
We shall need their following properties.
\begin{lem} \label{Elconst}
We have $E_{0}=E_{1}=0$, $E_{l}(-\zeta)=(-1)^{l}E_{l}(\zeta)$, and $E_{l}(\zeta)\in\mathbb{Q}[\zeta]$ for the polynomials from Equation \eqref{Eldef}. Moreover, if $H_{n}:=\sum_{a=1}^{n}\frac{1}{a}$ denotes the $n$\textsuperscript{th} harmonic number, then we have \[E_{l}(0)=-P_{l}(0)\sum_{a=1,\ 2 \nmid a}^{l}\frac{1}{a}=-\frac{P_{l}(0)}{2}(2H_{l}-H_{\lfloor l/2 \rfloor})\qquad\mathrm{for}\qquad l\in\mathbb{N}.\]
\end{lem}

\begin{proof}
Theorem \ref{PsiUpsBexp} and Lemma \ref{Pilprop} evaluate the generating series
\[\sum_{l=0}^{\infty}E_{l}(\zeta)t^{l}=\sum_{l=0}^{\infty}\!\int_{-\infty}^{\infty}\!\!g(\xi)\frac{\Pi_{l}(\xi,\zeta)+Q_{l}(\zeta)}{\xi+\zeta}t^{l}d\xi=\!\int_{0}^{t}\!\!e^{\zeta(t-w)+(t^{2}-w^{2})/2} \int_{-\infty}^{\infty}g(\xi)\frac{1-e^{(\xi+\zeta)w}}{\xi+\zeta}d\xi dw.\] We now claim that for every real $\zeta$, $w$, and $h$ we have
\begin{equation} \label{intidexp}
\frac{1}{\sqrt{2\pi}}\int_{-\infty}^{\infty}\frac{e^{(\xi+\zeta)w}-1}{\xi+\zeta}\mathrm{g}(\xi+h)d\xi=\int_{0}^{w}e^{(\zeta-h)s+s^{2}/2}ds.
\end{equation}
Indeed, both sides vanish for $r=0$, and their derivative with respect to $r$ coincide because $\int_{-\infty}^{\infty}\mathrm{g}(\xi)d\xi=\sqrt{2\pi}$. Substituting Equation \eqref{intidexp} into our generating series yields \[\sum_{l=0}^{\infty}E_{l}(\zeta)t^{l}=-e^{\zeta t+t^{2}/2}\int_{0}^{t}e^{-\zeta w-w^{2}/2}\int_{0}^{w}e^{\zeta s+s^{2}/2}dsdw,\] yielding the vanishing of $E_{0}$ and $E_{1}$. We apply the operator $\zeta+t-\partial_{t}$ to both sides, and using Theorem \ref{PsiUpsBexp} again, with Equation \eqref{genser}, we get \[\sum_{l=1}^{\infty}\big(\zeta E_{l}(\zeta)+E_{l-1}(\zeta)-(l+1)E_{l+1}(\zeta)\big)t^{l}=\int_{0}^{t}e^{\zeta s+s^{2}/2}ds=\int_{0}^{t}\Psi(\zeta,s)ds=\sum_{l=1}^{\infty}\frac{P_{l-1}(\zeta)}{l}t^{l}.\] The resulting recurrence relation establishes the rationality and the parity, and as the asserted values for the constant terms satisfy the resulting relation for $\zeta=0$, these values follow as well. This proves the lemma.
\end{proof}

Given $l\in\mathbb{N}$, Equation \eqref{PlTaylor} and the definition of the polynomials $\Pi_{l}$ in Lemma \ref{Pilprop} show that $\frac{P_{l}(\xi)-(-1)^{l}P_{l}(\zeta)}{\xi+\zeta}$ and
$\frac{Q_{l}(\xi)-(-1)^{l}\Pi_{l}(\xi,\zeta)}{\xi+\zeta}$ are polynomials in $\mathbb{Q}[\xi,\zeta]$. We can then define
\begin{equation} \label{Omegaldef}
\Omega_{l}(\zeta):=\frac{1}{\sqrt{2\pi}}\int_{-\infty}^{\infty}\frac{\big(P_{l}(\xi)-(-1)^{l}P_{l}(\zeta)\big)\mathrm{e}(\xi)+\big(Q_{l}(\xi)-(-1)^{l}\Pi_{l}(\xi,\zeta)\big)\mathrm{g}(\xi)}{\xi+\zeta}d\xi,
\end{equation}
and deduce the following properties.
\begin{lem} \label{Omegaser}
The generating series of the polynomials from Equation \eqref{Omegaldef} is \[\sum_{l=0}^{\infty}\Omega_{l}(\zeta)t^{l}=\int_{0}^{t}e^{-\zeta s}\frac{e^{t^{2}/2}-e^{s^{2}/2}}{t-s}ds.\]
\end{lem}

\begin{proof}
By applying Theorem \ref{PsiUpsBexp} and the generating series from Lemma \ref{Pilprop}, we can write
\[\sum_{l=0}^{\infty}\Omega_{l}(\zeta)t^{l}=\frac{e^{-\zeta t+t^{2}/2}}{\sqrt{2\pi}}\int_{-\infty}^{\infty}\frac{e^{(\zeta+\xi)t}-1}{\zeta+\xi}\bigg(\mathrm{e}(\xi)+\mathrm{g}(\xi)\int_{0}^{t}e^{-\xi w-w^{2}/2}dw\bigg)d\xi,\] where the second term in the parentheses is $\int_{0}^{t}\mathrm{g}(\xi+w)dw$. We now claim that
\begin{equation} \label{intidden}
\frac{1}{\sqrt{2\pi}}\int_{-\infty}^{\infty}\frac{e^{(\zeta+\xi)t}-1}{\zeta+\xi}\mathrm{e}(\xi)d\xi=\int_{0}^{t}e^{\zeta s}\frac{1-e^{s^{2}/2}}{s}ds
\end{equation}
Indeed, as with Equation \eqref{intidexp}, both sides vanish for $t=0$, and for comparing their derivatives with respect to $t$, we employ integration by parts and use Equation \eqref{egdist} and the equality $\int_{-\infty}^{\infty}\mathrm{g}(\xi)d\xi=\sqrt{2\pi}$ again. Substituting Equations \eqref{intidexp} and \eqref{intidden} transforms our expression for the generating series into \[e^{-\zeta t+t^{2}/2}\bigg(\int_{0}^{t}e^{\zeta s}\frac{1-e^{s^{2}/2}}{s}ds+\int_{0}^{t}\int_{0}^{t}e^{(\zeta-w)s+s^{2}/2}dsdw\bigg)=e^{-\zeta t+t^{2}/2}\int_{0}^{t}e^{\zeta s}\frac{1-e^{s^{2}/2-ts}}{s}ds,\] from which the desired formula follows by taking $s \mapsto t-s$. This proves the lemma.
\end{proof}

\begin{rmk} \label{tildeOmega}
Expanding all the exponents in the generating series from Lemma \ref{Omegaser} and integrating yields the series $\sum_{m=0}^{\infty}\sum_{b=1}^{\infty}\sum_{r=0}^{2b-1}\frac{t^{2b+m}(-\zeta)^{m}}{2^{b}b!m!(2b+m-r)}$, where the internal sum over $r$ can be written as the difference $H_{2b+m}-H_{m}$ between two harmonic numbers. Writing $\frac{1}{2^{b}b!}$ as $P_{2b}(0)$ as well as $P_{2b-1}(0)=0$ using Corollary \ref{Hermite}, we deduce that \[\Omega_{l}(\zeta)=(-1)^{l}\sum_{\nu=1}^{l}P_{\nu}(0)(H_{l}-H_{l-\nu})\frac{\zeta^{l-\nu}}{(l-\nu)!}=(-1)^{l}(l-\tfrac{1}{2})\frac{\zeta^{l-2}}{l!}+O(\zeta^{l-4})\] is a rational polynomial of degree $l-2$ and parity $(-1)^{l}$ every $l\in\mathbb{N}$. We shall later also need, for $k\in\mathbb{N}$, the rational, $(-1)^{k}$-symmetric, degree $k-2$ polynomial \[\tilde{\Omega}_{k}(\eta):=(-i)^{k}\Omega_{k}(i\eta)=(-1)^{k}\sum_{\nu=1}^{k}\frac{\operatorname{He}_{\nu}(0)}{\nu!}(H_{k}-H_{k-\nu})\frac{\eta^{k-\nu}}{(k-\nu)!}.\]
\end{rmk}

\subsection{Singular Schwartz Functions}

For every $\nu\in\mathbb{Z}$ and $\xi\in\mathbb{R}$, we shall now define
\begin{equation} \label{hndef}
h_{\nu}(\xi):=P_{\nu}(\xi)\mathrm{e}(\xi)+Q_{\nu}(\xi)\mathrm{g}(\xi),
\end{equation}
which has the following property generalizing Equation \eqref{egdist}.
\begin{prop} \label{distderh}
For any $\nu\in\mathbb{Z}$ the function $h_{\nu}$ has parity $(-1)^{\nu-1}$, and we have \[\tfrac{d}{d\xi}h_{\nu}(\xi)=h_{\nu-1}(\xi)-\sqrt{2\pi} \cdot P_{\nu}(0)\cdot\delta(\xi).\]
\end{prop}

\begin{proof}
The parity follows from that of the polynomials $P_{\nu}$ and $Q_{\nu}$ from Corollary \ref{PnuQnu}. The second claim follows from Equation \eqref{egdist}, the equality $\mathrm{g}'(\xi)=-\xi\mathrm{g}(\xi)$, and the relation from Equation \eqref{polsdef}. This proves the proposition.
\end{proof}

\begin{rmk} \label{decay}
The decay of $\mathrm{g}$ and $\mathrm{e}$ from Equations \eqref{Gaussprim} and \eqref{edef} imply that \[|{h}_{\nu}(\xi)|=o_{\epsilon,\nu}\big(e^{-(1-\epsilon)\xi^{2}/2}\big)\quad\text{as}\quad|\xi|\to\infty\qquad\mathrm{for}\quad\nu\in\mathbb{Z}\quad\mathrm{and}\quad\epsilon>0.\]
\end{rmk}

\begin{lem} \label{diffhnu}
For any $\nu\in\mathbb{Z}$ and $\xi>0$ we have the equalities $\nu h_{\nu}(\xi)-\xi h_{\nu-1}(\xi)=h_{\nu-2}(\xi)$ and $\xi^{3}\frac{d}{d\xi}\frac{h_{\nu}(\xi)}{\xi^{\nu}}=-\frac{h_{\nu-2}(\xi)}{\xi^{\nu-2}}$.
\end{lem}

\begin{proof}
The first equality is a direct consequence of Remarks \ref{PQrecur} and \ref{Qneg}, and the second one follows the first via Proposition \ref{distderh}. This proves the lemma.
\end{proof}

Since $h_{\nu}$ are in $L^{1}$ with exponential decay, their Fourier transforms $\widehat{h_{\nu}}$ should be bounded and $C^{\infty}$. However, for $\nu\geq0$, $h_{\nu}$ is not $C^{\infty}$ hence not a Schwartz function, and thus $\widehat{h_{\nu}}$ need not be in $L^{1}$. The explicit formula is given in the following result.
\begin{prop} \label{Fourierhnu}
For every $\nu\geq-1$ and $t\in\mathbb{R}$ we have the equality \[\widehat{h_{\nu}}(t):=\int_{-\infty}^{\infty}h_{\nu}(\xi)\mathbf{e}(-\xi t)d\xi=\sqrt{2\pi}\bigg(\frac{\mathrm{g}(2\pi t)}{(2\pi it)^{\nu+1}}-\sum_{r=0}^{\nu}\frac{P_{r}(0)}{(2\pi it)^{\nu-r+1}}\bigg)=\sqrt{2\pi}\sum_{r=\nu+1}^{\infty}\frac{P_{r}(0)}{(2\pi it)^{\nu-r+1}}.\] In particular, $\widehat{h_{\nu}}$ is bounded and $C^{\infty}$.
\end{prop}

\begin{proof}
The case $\nu=-1$ is just the Fourier transform of the Gaussian, combining with Equation \eqref{Hermitegen} and Corollary \ref{Hermite}. Now, applying the Fourier transform to Proposition \ref{distderh} yields the equality $2\pi it\cdot\widehat{h_{\nu}}(t)=\widehat{h_{\nu-1}}(t)-\sqrt{2\pi} \cdot P_{\nu}(0)$ for every $\nu\in\mathbb{Z}$. The general formula follows by induction on $\nu$ (to both sides), and implies the boundedness and $C^{\infty}$ properties. This proves the proposition.
\end{proof}

For two indices $\kappa$ and $\nu$ in $\mathbb{Z}$, we define the function
\begin{equation} \label{phidef}
\varphi_{\kappa,\nu}(\lambda,z):=\frac{(\lambda,Z(z))^{\kappa}}{(2\pi)^{(\nu+1)/2}}h_{\nu}\big(\sqrt{2\pi}\big(\lambda,Z^{\perp}(z)\big)\big), \end{equation}
with $\lambda \in V_{\mathbb{R}}$ and $z\in\mathcal{H}$ such that $\big(\lambda,Z^{\perp}(z)\big)\neq0$. For $\kappa<0$ we also impose the condition $\big(\lambda,Z(z)\big)\neq0$. These functions decay like Schwartz functions by Remark \ref{decay}, But near points where $\big(\lambda,Z^{\perp}(z)\big)$ or $\big(\lambda,Z^{}(z)\big)$ vanishes, they may become discontinuous. They are therefore ``singular'' Schwartz functions.

The parity from Proposition \ref{distderh} implies that
\begin{equation} \label{phipar}
\varphi_{\kappa,\nu}(-\lambda,z)=(-1)^{\kappa+\nu+1}\varphi_{\kappa,\nu}(\lambda,z)\quad\text{for\ every}\quad\lambda \in V_{\mathbb{R}}\quad\text{and}\quad z\in\mathcal{H},
\end{equation}
and the behavior of $Z(z)$ and $Z^{\perp}(z)$ under the action of $\operatorname{SL}_{2}(\mathbb{R})$ implies that the function $\varphi_{\kappa,\nu}(\lambda,z)$ from Equation \eqref{phidef} has the modularity property
\begin{equation} \label{phimod}
\varphi_{\kappa,\nu}(\gamma\lambda,\gamma z)=j(\gamma,z)^{-2\kappa}\varphi_{\kappa,\nu}(\lambda,z)\quad\text{for}\quad\lambda \in V_{\mathbb{R}},\ z\in\mathcal{H},\text{ and }\gamma\in\operatorname{SL}_{2}(\mathbb{R}).
\end{equation}
Note that for all $k\in\mathbb{N}$ the functions from Equations \eqref{Schwarzdef} and \eqref{phidef}, and the expansions from Equations \eqref{Thetadef} and \eqref{expTheta}, are related by
\begin{equation} \label{Schwarzcomp}
\tilde{\varphi}_{k}(\lambda;\tau,z)=v^{-k/2}q_{\tau}^{Q(\lambda)}\varphi_{k,-1}\big(\sqrt{v}\lambda,z\big)\mathrm{\ and\ }\Theta_{k,L,h}(\tau,z)=v^{\frac{1-k}{2}}\!\!\sum_{m\in\mathbb{Z}+Q(h)}\!\sum_{\lambda \in L_{m,h}}\!\varphi_{k,-1}\big(\sqrt{v}\lambda,z\big)q_{\tau}^{m},
\end{equation}
since ${h}_{-1}$ is just the Gaussian $\mathrm{g}$. Moreover, since $Z(z)-yZ^{\perp}(z)=\frac{iy}{\sqrt{N}}\big(\begin{smallmatrix}1 & -2z \\ 0 & 1\end{smallmatrix}\big)$ we can write
\begin{equation} \label{phixi+ieta}
\varphi_{\kappa,\nu}(\lambda,z)=\frac{y^{\kappa}(\xi+i\eta)^{\kappa}}{(2\pi)^{(\kappa+\nu+1)/2}}h_{\nu}(\xi)\quad\text{with}\quad\xi=\sqrt{2\pi}\big(\lambda,Z^{\perp}(z)\big)\quad\text{and}\quad \eta=\sqrt{\tfrac{2\pi}{N}}\Big(\lambda,\big(\begin{smallmatrix}1 & -2z \\ 0 & 1\end{smallmatrix}\big)\Big),
\end{equation}
and prove the following result.
\begin{prop} \label{Lzphi}
Take $\kappa$ and $\nu$ in $\mathbb{Z}$ as well as an element $0\neq\lambda \in V_{\mathbb{R}}$. Then we have \[-L_{z}\varphi_{\kappa,\nu}(\lambda,z)=\varphi_{\kappa+1,\nu-1}(\lambda,z)\]
at every point $z\in\mathcal{H}$ such that $\big(\lambda,Z^{\perp}(z)\big)\neq0$, where if $\kappa<0$ then we assume that $z$ must also satisfy $\big(\lambda,Z(z)\big)\neq0$.
\end{prop}

\begin{proof}
Write $\varphi_{\kappa,\nu}(\lambda,z)$ as in Equation \eqref{phixi+ieta}, and then simple calculations give that $L_{z}y=y^{2}$, $L_{z}\xi=-y(\xi+i\eta)$, and $L_{z}\eta=0$ in these parameters. Combining this with Proposition \ref{distderh} produces the desired result. This proves the proposition.
\end{proof}
Note that the assumption $\big(\lambda,Z(z)\big)\neq0$ is always satisfied when $Q(\lambda)\geq0$ (and $\lambda\neq0$), and if $Q(\lambda)<0$ then it holds for every $z\in\mathcal{H}$ except a single point, in which $Z^{\perp}(z)$ is a scalar multiple of $\lambda$.

\subsection{Fourier Transforms}

We shall also need the Fourier expansion of the following generalization of the functions $h_{\nu}$. Take $\kappa\in\mathbb{Z}$, and then for $\xi\in\mathbb{R}$ and $\eta\in\mathbb{R}$ we define
\begin{equation} \label{egkdef}
\begin{split} \mathrm{e}_{\kappa}(\xi;\eta)&:=(\xi+i\eta)^{\kappa-1}h_{0}(\xi)=(\xi+i\eta)^{\kappa-1}\mathrm{e}(\xi)\qquad\mathrm{and} \\ \mathrm{g}_{\kappa}(\xi;\eta)&:=(\xi+i\eta)^{\kappa-1}h_{-1}(\xi)=(\xi+i\eta)^{\kappa-1}\mathrm{g}(\xi). \end{split}
\end{equation}
For any $l\in\mathbb{N}$ we use Equation \eqref{PlTaylor} and the polynomials from Lemma \ref{Pilprop} and define
\begin{equation} \label{gkldef}
\mathbf{g}_{\kappa,l}(\xi;\eta):=P_{l}(i\eta)\mathrm{e}_{\kappa}(\xi;\eta)+\Pi_{l}(\xi,i\eta)\mathrm{g}_{\kappa}(\xi;\eta)=\sum_{\nu=0}^{l}\frac{(-1)^{\nu}}{(l-\nu)!}(\xi+i\eta)^{\kappa+l-\nu-1}h_{\nu}(\xi). \end{equation}
By applying Equation \eqref{egdist} and the derivative formula from that lemma (or Proposition \ref{distderh} with Equation \eqref{PlTaylor}), we get
\begin{equation} \label{gklder}
\partial_{\xi}\mathbf{g}_{\kappa+1,l}(\xi;\eta)=\kappa\mathbf{g}_{\kappa,l}(\xi;\eta)+\frac{\mathrm{g}_{\kappa+l+1}(\xi;\eta)}{l!}-\sqrt{2\pi}(i\eta)^{\kappa}P_{l}(i\eta)\delta(\xi). \end{equation}
We assume that $\eta\neq0$ in case $\kappa\leq0$, and then the expression from Equations \eqref{egkdef} and \eqref{gkldef} are $L^{1}$ as functions of $\xi$. Let thus $\widehat{\mathrm{e}_{\kappa}}$, $\widehat{\mathrm{g}_{\kappa}}$ and $\widehat{\mathbf{g}_{\kappa,l}}$ denote the  Fourier transform of $\mathrm{e}_{\kappa}, \mathrm{g}_{\kappa}$ and $\mathbf{g}_{\kappa,l}$ in $\xi$ respectively. Result for $\eta<0$ can be obtained from those with $\eta>0$ because we have \[\mathbf{g}_{\kappa,l}(\xi;-\eta)=(-1)^{\kappa+l}\mathbf{g}_{\kappa,l}(-\xi;\eta)\quad\mathrm{and\ hence}\quad\widehat{\mathbf{g}_{\kappa,l}}(t;-\eta)=(-1)^{\kappa+l}\widehat{\mathbf{g}_{\kappa,l}}(-t;\eta).\]

For analyzing the behavior of $\widehat{\mathbf{g}_{\kappa,l}}(t;\eta)$, we shall need additional special functions. For every $\nu\in\mathbb{Z}$, $j\in\mathbb{N}$, $t\in\mathbb{R}$, and $\eta>0$ we define
\begin{equation} \label{Ikdef}
I_{\nu,j}(\eta,t):=\int_{-t}^{\infty}\frac{(w+t)^{j}}{j!}{e^{-\eta w}}\bigg(e^{-w^{2}/2}-\sum_{\mu=0}^{\nu}\frac{\operatorname{He}_{\mu}(0)}{\mu!}w^{\mu}\bigg)\frac{dw}{w^{\nu+1}},\quad
I_{\nu}(\eta):=I_{\nu,0}(\eta,0).
\end{equation}
For $\nu\leq-1$ this is defined for all $\eta\in\mathbb{R}$. It is easy to check that
\begin{equation} \label{Ikjder}
\partial_{\eta}I_{\nu,j}(\eta,t)=\frac{\operatorname{He}_{\nu}(0)}{\nu!\eta^{j+1}}e^{\eta t}-I_{\nu-1,j}(\eta,t)\qquad\mathrm{and\ thus}\qquad I_{\nu}'(\eta)=\frac{\operatorname{He}_{\nu}(0)}{\nu!\eta}-I_{\nu-1}(\eta)
\end{equation}
for all $\nu$ and $j$, while Equation \eqref{Gaussprim} and simple differentiation give
\begin{equation} \label{derofIkj}
I_{-1}(\eta)=-e^{\eta^{2}/2}\mathrm{e}(\eta)\qquad\mathrm{and}\qquad\partial_{t}I_{\nu,j}(\eta,t)=\begin{cases} I_{\nu,j-1}(\eta,t), & \text{ if }j\geq1, \\ e^{\eta t}\mathrm{g}(t),& \text{ if }j=0\text{ and }\nu=-1. \end{cases}
\end{equation}
Using the functions from Equation \eqref{Ikdef}, we can now evaluate the Fourier transforms $\widehat{\mathrm{e}_{\kappa}}$ and $\widehat{\mathrm{g}_{\kappa}}$ of the functions from Equation \eqref{egkdef} as follows.
\begin{lem} \label{egFT}
For $\kappa\geq1$ we have the equalities \[\begin{split} \widehat{\mathrm{e}_{\kappa}}(t;\eta)&=\sqrt{2\pi}\sum_{\mu=0}^{\kappa-1}\binom{\kappa-1}{\mu}\frac{(i\eta)^{\kappa-1-\mu}}{(-2\pi i)^{\mu}}\cdot\bigg(\frac{d}{dt}\bigg)^{\mu}\bigg(\frac{\mathrm{g}(2\pi t)-1}{2\pi it}\bigg)\qquad\mathrm{and} \\ \widehat{\mathrm{g}_{\kappa}}(t;\eta)&=\sqrt{2\pi}(-i)^{\kappa-1}\operatorname{He}_{\kappa-1}(2\pi t-\eta)\mathrm{g}(2\pi t). \end{split}\] On the other hand, when $\kappa\leq0$ and $\eta>0$ we get
\[\widehat{\mathrm{e}_{\kappa}}(t;\eta)=\sqrt{2\pi} \cdot i^{\kappa}e^{-2\pi\eta t}I_{0,-\kappa}(\eta,2\pi t)\quad\mathrm{and}\quad\widehat{\mathrm{g}_{\kappa}}(t;\eta)=\sqrt{2\pi} \cdot i^{\kappa-1}e^{-2\pi\eta t}I_{-1,-\kappa}(\eta,2\pi t).\]
\end{lem}

\begin{proof}
The results for $\kappa\geq1$ follow from applying the differential operator $\big(i\eta-\frac{\partial_{t}}{2\pi i}\big)^{\kappa-1}$ to the cases $\nu=0$ or $\nu=-1$ of Proposition \ref{Fourierhnu} respectively, combined with Equation \eqref{Appell} for the latter. For the case $\kappa\leq0$ we recall that \[\int_{-\infty}^{\infty}(\xi+i\eta)^{\kappa-1} \mathbf{e}(-\xi t)d\xi= \begin{cases} 2\pi \cdot i^{\kappa-1}\frac{(2\pi t)^{|\kappa|}e^{-2\pi\eta t}}{|\kappa|!}, & \text{ if }t>0, \\ 0,& \text{ if } t<0, \end{cases}\] and we can expressing the Fourier transform of the product as the convolution of the Fourier transforms. This proves the lemma.
\end{proof}

\begin{rmk} \label{FourIk}
The proof of Lemma \ref{egFT} can evaluate, using the full Proposition \ref{Fourierhnu} and Corollary \ref{Hermite}, the Fourier transform the generalization $\xi\mapsto(\xi+i\eta)^{\kappa-1}h_{\nu}(\xi)$ of the functions from Equation \eqref{egkdef} to any $\nu$. When $\kappa\leq0$ this gives just $\sqrt{2\pi}i^{\kappa+\nu}I_{\nu,\kappa}(\eta,t)$. For $\kappa=0$ and $t=0$ this reduces to the equality $\int_{-\infty}^{\infty}\frac{h_{\nu}(\xi)}{\xi+i\eta}d\xi=\sqrt{2\pi}i^{\nu}I_{\nu}(\eta)$.
\end{rmk}
We can now deduce the following useful property.
\begin{lem} \label{Ikdiff}
For every $\nu\in\mathbb{Z}$ and $\eta>0$ we have the equality $\eta^{3}\frac{d}{d\eta}\frac{I_{\nu}(\eta)}{\eta^{\nu}}=\frac{I_{\nu-2}(\eta)}{\eta^{\nu-2}}$.
\end{lem}

\begin{proof}
Using Equation \eqref{Ikjder}, Remark \ref{FourIk}, and the first equality in Lemma \ref{diffhnu} we get
\begin{align*}
\eta^{1+\nu}\frac{d}{d\eta}\frac{I_{\nu}(\eta)}{\eta^{\nu}}&=\frac{\operatorname{He}_{\nu}(0)}{\nu!}-(\eta I_{\nu-1}(\eta)+\nu I_{\nu}(\eta))=\frac{\operatorname{He}_{\nu}(0)}{\nu!}-\frac{(-i)^{\nu}}{\sqrt{2\pi}}\int_{-\infty}^{\infty}\frac{\nu h_{\nu}(\xi)+i\eta h_{\nu-1}(\xi)}{\xi+i\eta}d\xi \\ &=\frac{\operatorname{He}_{\nu}(0)}{\nu!}-\frac{(-i)^{\nu}}{\sqrt{2\pi}}\int_{-\infty}^{\infty}\bigg(\frac{h_{\nu-2}(\xi)}{\xi+i\eta}+h_{\nu-1}(\xi)\bigg)d\xi.
\end{align*}
The first integral produces the desired result $I_{\nu-2}(\eta)$ by Remark \ref{FourIk} again, and integrating Proposition \ref{distderh} evaluates, via Remark \ref{decay}, the second one as $-(-1)^{\nu}P_{\nu}(0)$. As this cancels with the first term by Corollary \ref{Hermite}, this proves the lemma.
\end{proof}
One can also show that $I_{\nu}$, using the polynomials from Equation \eqref{polmod} below, that
\begin{equation} \label{expInu}
I_{\nu}=\widetilde{P}_{\nu}I_{0}-\widetilde{Q}_{\nu}I_{-1}+\hat{\Omega}_{\nu},\quad\text{with a polynomial }\hat{\Omega}_{\nu}\text{ of degree }|\nu|-2\text{ and parity }(-1)^{\nu}.
\end{equation}

We can also evaluate the Fourier transform of $\mathbf{g}_{\kappa,l}$ at $t=0$.
\begin{prop} \label{FTt0}
Given any $\kappa\in\mathbb{Z}$, $l\in\mathbb{N}$, and $\eta>0$ we have
\[\widehat{\mathbf{g}_{\kappa,l}}(0;\eta)=\begin{cases} -\frac{\sqrt{2\pi}i^{\kappa+l}}{\kappa \cdot l!}\bigg(\operatorname{He}_{\kappa+l}(\eta)-\eta^{\kappa}\operatorname{He}_{l}(\eta)\bigg),& \text{ if }\kappa\neq0\text{ and }\kappa+l\geq0, \\ \sqrt{2\pi}\cdot(-i)^{l}\big(I_{l}(\eta)-\tilde{\Omega}_{l}(\eta)\big),& \text{ if }\kappa=0, \end{cases}\] where $I_{\nu,j}$ and $\tilde{\Omega}_{k}$ are defined in Equation \eqref{Ikdef} and Remark \ref{tildeOmega} respectively.
\end{prop}

\begin{proof}
When $\kappa\neq0$ and $\eta>0$ we can integrate Equation \eqref{gklder} (using the fact that Remark \ref{decay} extends to the functions $\mathbf{g}_{\kappa,l}$) to obtain
\begin{align*}
\widehat{\mathbf{g}_{\kappa,l}}(0;\eta)&=\int_{-\infty}^{\infty}\mathbf{g}_{\kappa,l}(\xi;\eta)d\xi=\frac{1}{\kappa}\int_{-\infty}^{\infty}
\bigg(-\frac{\mathrm{g}_{\kappa+l+1}(\xi;\eta)}{l!}+\sqrt{2\pi}(i\eta)^{\kappa}P_{l}(i\eta)\delta(\xi)\bigg)d\xi \\ &=\frac{1}{\kappa}\bigg(-\frac{\widehat{\mathrm{g}_{\kappa+l+1}}(0;\eta)}{l!}+\sqrt{2\pi}P_{l}(i\eta)(i\eta)^{\kappa}\bigg),
\end{align*}
which gives the desired value by Lemma \ref{egFT} and Corollary \ref{Hermite}.

When $\kappa=0$ we evaluate directly, where the formula from Remark \ref{FourIk} gives \[\widehat{\mathbf{g}_{0,l}}(0;\eta)=\sqrt{2\pi}(-i)^{l}I_{l}(\eta)-(-1)^{l}\int_{-\infty}^{\infty}\frac{h_{l}(\xi)-(-1)^{l}(P_{l}(i\eta)\mathrm{e}(\xi)+\Pi_{l}(\xi,i\eta)\mathrm{g}(\xi))}{\xi+i\eta}d\xi.\] As Equation \eqref{hndef} transforms the latter integral into $(-1)^{l}\sqrt{2\pi}\Omega_{l}(i\eta)$ via Equation \eqref{Omegaldef}, we get the desired value by Remark \ref{tildeOmega}. This proves the proposition.
\end{proof}

\subsection{Asymptotic Estimates}

For determining the asymptotic behavior of the functions $I_{\nu,j}$ and the Fourier transforms $\widehat{\mathbf{g}_{\kappa,l}}$, we shall also need the following special functions. Like the definition of $\tilde{\Omega}_{k}$ in Remark \ref{tildeOmega}, the polynomials $P_{\nu}$ and $Q_{\nu}$ for $\nu\in\mathbb{Z}$ given in Equation \eqref{polsdef} and Remark \ref{Qneg} have the modifications
\begin{equation} \label{polmod}
\widetilde{P}_{\nu}(\eta):=i^{\nu}P_{\nu}(i\eta)\qquad\mathrm{and}\qquad\widetilde{Q}_{\nu}(\eta):=i^{\nu-1}Q_{\nu}(i\eta),
\end{equation}
of the same parities as $P_{\nu}$ and $Q_{\nu}$. Using these we define functions
\begin{equation} \label{Jnudef}
\begin{split} &J_{-1}(\eta):=\mathrm{g}(i\eta)=e^{\eta^{2}/2},\quad J_{0}(\eta):=-\int_{0}^{\eta}e^{r^{2}/2}dr=\frac{1}{\sqrt{2\pi}}\int_{0}^{\infty}\frac{e^{-\eta s}-e^{\eta s}}{s}\mathrm{g}(s)ds \\ &\mathrm{and}\qquad J_{\nu}(\eta):=\tilde{P}_{\nu}(\eta)J_{0}(\eta)-\tilde{Q}_{\nu}(\eta)J_{-1}(\eta)\qquad\mathrm{for}\qquad\nu\in\mathbb{Z} \end{split}
\end{equation}
(the two expressions for $J_{0}(\eta)$ are the same because they vanish at $\eta=0$ and have the same derivative, and since $\widetilde{P}_{0}=\widetilde{Q}_{-1}=1$ and $\widetilde{P}_{-1}=\widetilde{Q}_{0}=0$ the two definitions for $\nu=0$ and for $\nu=-1$ coincide). Since the polynomials from Equation \eqref{polmod} satisfy identities analogous to those from Equation \eqref{polsdef} and Remark \ref{PQrecur}, the latter of which yields the equality $\eta J_{\nu}(\eta)+\eta J_{\nu-1}(\eta)=-J_{\nu-2}(\eta)$ for every $\eta$ and $\nu$, evaluating the derivatives of $J_{0}$ and $J_{-1}$ implies, for all $\nu\in\mathbb{Z}$, the relations
\begin{equation} \label{Jdiff}
J_{\nu}'=-J_{\nu-1},\quad J_{\nu}(-\eta)=(-1)^{\nu-1}J_{\nu}(\eta),\quad\mathrm{and}\quad\eta^{3}\frac{d}{d\eta}\frac{J_{\nu}(\eta)}{\eta^{\nu}}=\frac{J_{\nu-2}(\eta)}{\eta^{\nu-2}}. \end{equation}

The following estimates will be helpful when we evaluate sums of Fourier transforms later.
\begin{lem} \label{asympI}
For any $\epsilon>0$ and fixed $\eta\in\mathbb{R}$ we have the asymptotic growth
\[I_{-1,0}(\eta,t)=\begin{cases} \sqrt{2\pi}J_{-1}(\eta)+o_{\eta,\epsilon}(e^{-(1-\epsilon)t^{2}/2}), & t\to\infty, \\ o_{\eta,\epsilon}(e^{-(1-\epsilon)t^{2}/2}),& t\to-\infty,\end{cases}\] with $J_{-1}(\eta)$ from Equation \eqref{Jnudef}. More generally, given $j\in\mathbb{N}$ we have \[I_{-1,j}(\eta,t)=P_{j}(t-\eta)I_{-1,0}(\eta,t)+o_{j,\eta,\epsilon}(e^{-(1-\epsilon)t^{2}/2})\qquad\mathrm{as}\qquad t\to\pm\infty.\] Moreover, given $\eta>0$ and such $\epsilon$, and with $J_{0}$ from Equation \eqref{Jnudef}, we get \[\frac{j!}{t^{j}}I_{0,j}(\eta,t)=-(-1)^{j}j!\Gamma(-j,-\eta t)+\begin{cases} \sqrt{2\pi}J_{0}(\eta)+O_{j,\eta}\big(\frac{1}{t}\big), & t\to\infty, \\ o_{j,\eta,\epsilon}(e^{-(1-\epsilon)t^{2}/2}), & t\to-\infty. \end{cases}\]
\end{lem}

\begin{proof}
Using Equation \eqref{edef} and the value $\sqrt{2\pi}$ of $\int_{-\infty}^{\infty}g(\xi)d\xi$, the first equation follows from \[I_{-1,0}(\eta,t)=\int_{-t}^{\infty}e^{-\eta w-w^{2}/2}dw=e^{\eta^{2}/2}\int_{\eta-t}^{\infty}\mathrm{g}(s)ds=\begin{cases}e^{\eta^{2}/2}\big(\sqrt{2\pi}-\mathrm{e}(t-\eta)\big), & t>\eta, \\ -e^{\eta^{2}/2}\mathrm{e}(\eta-t), & t<\eta,\end{cases}\] and Remark \ref{decay}. Next, using the definition in Equation \eqref{Ikdef} and simple algebra we get, for every $\eta\in\mathbb{R}$ and $j\in\mathbb{N}$, the equality
\[\delta_{j,0}e^{\eta t-t^{2}/2}=-\int_{-t}^{\infty}\frac{d}{dw}\big(\tfrac{(w+t)^{j}}{j!}e^{-\eta w-w^{2}/2}\big)dw=(j+1)I_{-1,j+1}-(t-\eta)I_{-1,j}-(1-\delta_{j,0})I_{-1,j-1}.\] Since the left hand side times any polynomial is $o_{\epsilon}(e^{-(1-\epsilon)t^{2}/2})$ as $t\to\pm\infty$, a simple induction on $j$ combines with Remark \ref{PQrecur} to prove the second relation.

Now suppose $t\neq0$ and $\eta>0$. For $j=0$ we evaluate \[I_{0,0}(\eta,t)=\int_{-t}^{\infty}e^{-\eta w}\frac{\mathrm{g}(w)-1}{w}dw=\int_{-t}^{|t|}e^{-\eta w}\frac{\mathrm{g}(w)}{w}dw+\int_{|t|}^{\infty}e^{-\eta w}\frac{\mathrm{g}(w)}{w}dw-\Gamma(0,-\eta t).\] As $t\to-\infty$, the first term above vanishes and the second term is $o_{\eta,\epsilon}(e^{-(1-\epsilon)t^{2}/2})$ for any fixed $\eta\in\mathbb{R}$. As $t\to\infty$, the second term behaves the same, whereas using Equation \eqref{Jnudef} we see that the first term contributes \[\int_{-t}^{t}\!\!e^{-\eta w}\frac{\mathrm{g}(w)}{w}dw\!=\!\!\int_{0}^{\infty}\!\frac{e^{-\eta w}\!-\!e^{\eta w}}{w}\mathrm{g}(w)dw-\!\int_{t}^{\infty}\!\frac{e^{-\eta w}\!-\!e^{\eta w}}{w}\mathrm{g}(w)dw\!=\!\sqrt{2\pi}J_{0}(\eta)+o_{\eta,\epsilon}(e^{-(1-\epsilon)t^{2}/2}).\] This proves the third equality for $j=0$. When $j\geq1$ we can write $\frac{j!}{t^{j}}I_{0,j}(\eta,t)$ as \[\frac{1}{t^{j}}\int_{-t}^{\infty}\frac{(w+t)^{j}-t^{j}}{w}e^{-\eta w-w^{2}/2}dw-\frac{e^{\eta t}}{t^{j}}\int_{0}^{\infty}\frac{s^{j}-t^{j}}{s-t}e^{-\eta s}ds+\int_{-t}^{\infty}e^{-\eta w}\frac{\mathrm{g}(w)-1}{w}dw.\] After expanding $(w+t)^{j}$ binomially, the first term is seen to be $o_{j,\eta,\epsilon}(e^{-(1-\epsilon)t^{2}/2})$ as $t\to-\infty$ and $O_{j,\eta}\big(\frac{1}{t}\big)$ as $t\to\infty$. Expanding $\frac{s^{j}-s^{j}}{s-t}$ and Equation \eqref{GammamuZt} evaluate the second term as $-(-1)^{j}j!\Gamma(-j,-\eta t)+\Gamma(0,-\eta t)$, and the third term is just the expression for $j=0$. Putting everything together proves the last desired equality. This proves the lemma.
\end{proof}
In fact, the first term in the last equation in the proof of Lemma \ref{asympI} can easily be evaluated explicitly up to $o_{j,\eta,\epsilon}(e^{-(1-\epsilon)t^{2}/2})$ also when $t\to\infty$, but the estimate $O_{j,\eta}\big(\frac{1}{t}\big)$ will be sufficient for our purposes.

\begin{rmk} \label{I-1j00}
The generating series $\sum_{j=0}^{\infty}I_{-1,j}(0,0)X^{j}$ of the constants $\big\{I_{-1,j}(0,0)\}_{j=0}^{\infty}$ is \[\int_{0}^{\infty}e^{wX-w^{2}/2}dw=e^{X^{2}/2}\int_{0}^{\infty}e^{-(w-X)^{2}/2}dw=e^{X^{2}/2}\bigg(\int_{0}^{\infty}e^{-w^{2}/2}dw+\int_{0}^{X}e^{-w^{2}/2}dw\bigg)\] (by symmetry), which equals $\Upsilon(0,X)+\sqrt{\pi/2}\Psi(0,X)$ by Theorem \ref{PsiUpsBexp}. It therefore follows from Equations \eqref{genser} and \eqref{hndef} and the parity from Proposition \ref{distderh} that \[I_{-1,j}(0,0)=Q_{j}(0)+\sqrt{\pi/2}P_{j}(0)=\lim_{\xi\to0^{-}}h_{j}(\xi)=(-1)^{j-1}\lim_{\xi\to0^{+}}h_{j}(\xi)\quad\mathrm{for}\quad j\in\mathbb{N}.\]
\end{rmk}
We can now state the asymptotic expansion of $\widehat{\mathbf{g}_{\kappa,l}}(t;\eta)$.
\begin{prop} \label{asympgkl}
Take $\kappa\in\mathbb{Z}$, $l\in\mathbb{N}$, and $\eta\in\mathbb{R}$, with $\eta>0$ in case $\kappa\leq0$. We then have \[\widehat{\mathbf{g}_{\kappa,l}}(t;\eta)=-\sqrt{2\pi}\frac{i^{\kappa+l}\operatorname{He}_{l}(\eta)}{l!(-2\pi t)^{\kappa}}e^{-2\pi\eta t}\Gamma(\kappa,-2\pi\eta t)+o_{\kappa,l,\eta,\epsilon}(e^{-2\pi^{2}(1-\epsilon)t^{2}}),\] an expression holding as $t\to-\infty$ as well as in the limit $t\to\infty$ when $\kappa\geq1$. On the other hand, if $\kappa\leq0$ and $\eta>0$ then the limit as $t\to\infty$ is given by the same formula plus the term \[\frac{2\pi(-i)^{\kappa+l}}{|\kappa|!(-2\pi t)^{\kappa}}e^{-2\pi\eta t}\big[J_{l}(\eta)+O_{\kappa,l,\eta}\big(\tfrac{1}{t}\big)\big],\] where $J_{l}(\eta)$ is the function defined in Equation \eqref{Jnudef}. \end{prop}

\begin{proof}
Taking the Fourier transform of Equation \eqref{gkldef} allows us to write \[\widehat{\mathbf{g}_{\kappa,l}}(t;\eta)=P_{l}(i\eta)\widehat{\mathrm{e}_{\kappa}}(t;\eta)+\Pi_{l}\big(\tfrac{\partial_{t}}{-2\pi i},i\eta\big)\widehat{\mathrm{g}_{\kappa}}(t;\eta).\] When $\kappa\geq1$ we apply Lemma \ref{egFT} and Equation \eqref{GammamuZt} to obtain that \[\widehat{\mathrm{e}_{\kappa}}(t;\eta)=-\sqrt{2\pi}\frac{e^{-2\pi\eta t}\Gamma(\kappa,-2\pi\eta t)}{(2\pi it)^{\kappa}}+o_{\kappa,l,\epsilon}(\eta^{\kappa-1}e^{-(1-\epsilon)t^{2}/2})\qquad\mathrm{as}\qquad t\to\pm\infty\] and that $\Pi_{l}\big(\frac{\partial_{t}}{-2\pi i},i\eta\big)\widehat{\mathrm{g}_{\kappa}}(t;\eta)$ is bounded by the same error term. This establishes, via Corollary \ref{Hermite}, the desired formula in this case. On the other hand, for $\kappa\leq0$ and $\eta>0$ we deduce from Lemma \ref{egFT} that \[\widehat{\mathbf{g}_{\kappa,l}}(t;\eta)=\sqrt{2\pi} \cdot i^{\kappa}e^{-2\pi\eta t}P_{l}(i\eta)I_{0,-\kappa}(\eta,2\pi t)+\sqrt{2\pi} \cdot i^{\kappa-1}\Pi_{l}\big(\tfrac{\partial_{t}}{-2\pi i},i\eta\big)e^{-2\pi\eta t}I_{-1,-\kappa}(\eta,2\pi t).\] Lemma \ref{asympI} now shows that in the limit $t\to-\infty$ the first summand is
\[-\sqrt{2\pi} \cdot i^{\kappa}P_{l}(i\eta)\frac{e^{-2 \pi\eta t}\Gamma(\kappa,-2\pi\eta t)}{(-2\pi t)^{\kappa}}+o_{\kappa,l,\epsilon}(e^{-2\pi^{2}(1-\epsilon)t^{2}})\] and the second one goes into the error term, which yields the required formula also here by another application of Corollary \ref{Hermite}. In the limit $t\to\infty$ we get the same contribution, but we need to consider the additional terms from Lemma \ref{asympI} in that limit. The term arising from the first summand is
\begin{equation} \label{term1as}
\frac{2\pi \cdot i^{\kappa}e^{-2\pi\eta t}}{|\kappa|!(2\pi t)^{\kappa}}\big[P_{l}(i\eta)J_{0}(\eta)+O_{\kappa,l,\eta}\big(\tfrac{1}{t}\big)\big]. \end{equation}
For the one from the second summand, Remark \ref{tildePil} allows us to write
\begin{align*}
\Pi_{l}\big(\tfrac{\partial_{t}}{-2\pi i},i\eta\big)&e^{-2\pi\eta t}I_{-1,-\kappa}(\eta,2\pi t)=\big[-Q_{l}(i\eta)+\tilde{\Pi}_{l}\big(\tfrac{\partial_{t}}{-2\pi i}+i\eta,i\eta\big)\big]e^{-2\pi\eta t}I_{-1,-\kappa}(\eta,2\pi t) \\ &=e^{-2\pi\eta t}\big[-Q_{l}(i\eta)I_{-1,-\kappa}(\eta,2\pi t)+\tilde{\Pi}_{l}\big(\tfrac{\partial_{t}}{-2\pi i},i\eta\big)I_{-1,-\kappa}(\eta,2\pi t)\big]
\end{align*}
(by the action on that exponent). Lemma \ref{asympI} and the fact that Corollary \ref{PnuQnu} gives the estimate $P_{|\kappa|}(2\pi t-\eta)=\frac{(2\pi t)^{|\kappa|}}{|\kappa|!}\big[1+O_{\kappa,\eta}\big(\frac{1}{t}\big)\big]$ show that $\sqrt{2\pi} \cdot i^{\kappa-1}$ times the first summand here is \[-\frac{2\pi \cdot i^{\kappa-1}e^{-2\pi\eta t}}{|\kappa|!(2\pi t)^{\kappa}}\big[Q_{l}(i\eta)J_{-1}(\eta)+O_{\kappa,l,\eta}\big(\tfrac{1}{t}\big)\big].\] Equations \eqref{polmod} and \eqref{Jnudef} now show that this expression combines with the one from Equation \eqref{term1as} to the asserted extra term, up to the required error term. Finally, Equation \eqref{derofIkj} and the property of $\tilde{\Pi}_{l}$ from Remark \ref{tildePil} imply, via Lemma \ref{asympI} again, that the expression involving that polynomial also goes into the error term. This proves the proposition.
\end{proof}

\subsection{Lattice Sums}

To evaluate the constant term of the Shintani lift, we need to calculate certain lattice sums involving the function $\mathbf{g}_{\kappa,l}$ defined in Equation \eqref{gkldef}. For $\kappa\in\mathbb{Z}$, $l\in\mathbb{N}$, $\eta\in\mathbb{R}$, a real number $\upsilon>0$, and an element $\omega\in\mathbb{R}/\mathbb{Z}$, we consider the sum
\begin{equation} \label{Gkldef}
\mathbf{G}_{\kappa,l}(\omega;\upsilon,\eta):=\sum_{0\neq\xi\in\mathbb{Z}+\omega}\mathbf{g}_{\kappa,l}(\upsilon\xi;\eta),
\end{equation}
which converges absolutely by Remark \ref{decay} and defines a continuous function in $\eta$. We will be interested in its value at $\eta=0$, denoted by
\begin{equation} \label{Gkleta0}
\mathbf{G}_{\kappa,l}(\omega;\upsilon):=\mathbf{G}_{\kappa,l}(\omega;\upsilon,0)=\lim_{\eta\to0^{+}}\mathbf{G}_{\kappa,l}(\omega;\upsilon,\eta). \end{equation}
and its asymptotic expansion as $\upsilon\to0^{+}$.

For analyzing it, we first need to recall a few familiar functions. Let $\{B_{\mu}(\omega)\}_{\mu=0}^{\infty}$ (for $\omega\in\mathbb{R}$) be the \emph{Bernoulli polynomials} defined by \[\frac{te^{\omega t}}{e^{t}-1}=\sum_{\mu=0}^{\infty}B_{\mu}(\omega)\frac{t^{\mu}}{\mu!},\quad\mathrm{so\ that\ in\ particular}\quad B_{1}(\omega)=\omega-\tfrac{1}{2},\] and then $B_{\mu}:=B_{\mu}(0)$ are the \emph{Bernoulli numbers}. Moreover, one defines $\mathbb{B}_{\mu}$ to be the 1-periodic function that coincides with $B_{\mu}$ on the interval $(0,1)$, and whose value on the integers is 0 in case $\mu=1$ and $B_{\mu}$ otherwise. Then $\mathbb{B}_{\mu}$ with $\mu\geq2$ is continuous on $\mathbb{R}$ (in particular, $\mathbb{B}_{0}$ is the constant function 1), and we have
\begin{equation} \label{Bernoulli}
\mathbb{B}_{1}(0)=0=\lim_{\omega\to0+}\mathbb{B}_{1}(\omega)+\tfrac{1}{2}=\lim_{\omega\to0-}\mathbb{B}_{1}(\omega)-\tfrac{1}{2}\qquad\mathrm{and}\qquad\mathbb{B}_{\mu}(\omega)=-\sum_{0 \neq m\in\mathbb{Z}}\frac{\mu!\mathbf{e}(m\omega)}{(2\pi im)^{\mu}},
\end{equation}
the latter Fourier expansion being valid for every $\omega\in\mathbb{R}/\mathbb{Z}$ and $\mu\in\mathbb{N}$ (this is essentially Equations (13), (14), and (15) in Section 1.13 of \cite{[EMOT]}).

We also recall from Section 1.11 of \cite{[EMOT]} the function
\begin{equation} \label{Fqsdef}
F(q,s)=\sum_{m=1}^{\infty}\frac{q^{m}}{m^{s}}\qquad\mathrm{for}\qquad s\in\mathbb{C} \quad\mathrm{and}\quad q\in\mathbb{C}\quad\mathrm{with}\quad|q|<1.
\end{equation}
Since we shall only use this function when $s=-j$ for $j\in\mathbb{N}$, where $F(q,-j)$ is a polynomial in $q$ divided by $(1-q)^{j+1}$, the analytic continuation to any $q\in\mathbb{C}\setminus[1,\infty)$, and even to any $1 \neq q\in\mathbb{C}$, is immediate. Writing $q=\mathbf{e}(\omega)$ for $\omega\in\mathbb{R}$, the fact that $F(q,0)=\frac{q}{1-q}$ combines with Equation (15) of Section 1.11 of \cite{[EMOT]} (for $j\geq1$) to gives, for all $j\in\mathbb{N}$, the expansion
\begin{equation} \label{Fexp}
F\big(\mathbf{e}(\omega),-j\big)=\frac{j!}{(-2\pi i\cdot\omega)^{j+1}}-\frac{B_{j+1}+\delta_{j,0}}{j+1}+O(\omega).
\end{equation}

Another function to recall is the \emph{polygamma function}, defined for $m\in\mathbb{N}$ and $z\in\mathbb{C}$ by
\begin{equation} \label{polygamma}
\psi^{(m)}(z):=\frac{d^{m+1}}{dz^{m+1}}\log\Gamma(z)=-\gamma\delta_{m,0}+(-1)^{m}m!\sum_{a=0}^{\infty}\bigg(\frac{\delta_{m,0}}{(a+1)^{m+1}}-\frac{1}{(a+z)^{m+1}}\bigg), \end{equation}
where $\delta_{m,0}$ is the Kronecker $\delta$-symbol again.

Let $\tilde{\psi}^{(m)}$ be the 1-periodic function that coincides with $\psi^{(m)}$ from Equation \eqref{polygamma} on $(0,1]$. Then for $\kappa\in\mathbb{Z}$ and $\omega\in\mathbb{R}/\mathbb{Z}$ we define
\begin{equation} \label{Phidef}
\Phi_{\kappa}(\omega):=\begin{cases} -\mathbb{B}_{\kappa}(\omega)/\kappa, & \kappa\geq1, \\ -\big[\tilde{\psi}^{(|\kappa|)}(-\omega)+(-1)^{\kappa}\tilde{\psi}^{(|\kappa|)}(\omega)\big]\big/2|\kappa|!, & \kappa\leq0, \end{cases} \end{equation}
as well as
\begin{equation} \label{Xidef}
\Xi_{\kappa}(\omega):=\frac{(-2\pi i)^{1-\kappa}}{\sqrt{2\pi}}\begin{cases} F(\mathbf{e}(\omega),\kappa)/|\kappa|!+\delta_{\kappa,0}/2, & \kappa\leq0\text{ and }\omega\not\in\mathbb{Z}, \\ -\mathbb{B}_{1-\kappa}(0)/(1-\kappa)!, & \kappa\leq1\text{ and }\omega\in\mathbb{Z}, \\ 0, & \text{otherwise.} \end{cases}
\end{equation}
Note that for all $\kappa\in2\mathbb{Z}$ we have the equality
\begin{equation} \label{Phizeta}
\Phi_{\kappa}(1)=\operatorname{CT}_{s=1-\kappa}\zeta(s).
\end{equation}
Recalling the notation $H_{n}$ for the $n$\textsuperscript{th} harmonic number, we shall also need the constant
\begin{equation} \label{Cldef}
C_{l}:=\frac{\gamma+\log2-2H_{l}+H_{\lfloor l/2 \rfloor}}{2}=\frac{\gamma+\log2}{2}-\sum_{a=1,\ 2 \nmid a}^{l}\frac{1}{a}.
\end{equation}
We remark that it is easy to see, via the asymptotic $H_{n}=\log n+\gamma+o(1)$ as $n\to\infty$, that $C_{l}$ from Equation \eqref{Cldef} grows as $-\frac{\log l}{2}+o(1)$ as $l\to\infty$.

The evaluation of the expression that we need is now carried out as follows.
\begin{prop} \label{Gkleval}
Take $\kappa\in\mathbb{Z}$, $l\in\mathbb{N}$, $\eta\in\mathbb{R}$, $\upsilon>0$, and $\omega\in\mathbb{R}/\mathbb{Z}$. Then the value of the expression $\mathbf{G}_{\kappa,l}(\omega;\upsilon)$ from Equation \eqref{Gkleta0} is \[-\sqrt{2\pi}\upsilon^{\kappa-1}\big[P_{l}(0)\Phi_{\kappa}(\omega)+Q_{l}(0)\Xi_{\kappa}(\omega)\big]+\begin{cases} -\frac{\sqrt{2\pi} \cdot i^{\kappa+l}\operatorname{He}_{\kappa+l}(0)}{\upsilon\kappa \cdot l!}+o_{\kappa,l,\epsilon}(e^{-2\pi^{2}(1-\epsilon)/\upsilon^{2}}), & \kappa\geq1, \\ \frac{\sqrt{2\pi}}{\upsilon}\delta_{\kappa,0}P_{l}(0)(\log\upsilon+C_{l})+O_{\kappa,l,\omega}(\upsilon^{\kappa}), & \kappa\leq0. \end{cases}\]
\end{prop}

\begin{proof}
Consider first the case where $\kappa\geq1$. For $\omega\neq0$ we can apply the Poisson Summation Formula, after which Propositions \ref{FTt0} and \ref{asympgkl} yield
\begin{align*}
\mathbf{G}_{\kappa,l}(\omega;\upsilon)&=\frac{1}{\upsilon}\bigg(\widehat{\mathbf{g}_{\kappa,l}}(0;0)+\lim_{\eta\to0}\sum_{0 \neq m\in\mathbb{Z}}\mathbf{e}(m\omega)\widehat{\mathbf{g}_{\kappa,l}}\big(\tfrac{m}{\upsilon};\eta\big)\bigg) \\
&=-\frac{\sqrt{2\pi}i^{\kappa+l}}{\upsilon\kappa}\bigg(\frac{\operatorname{He}_{\kappa+l}(0)}{l!}+\frac{\kappa\operatorname{He}_{l}(0)}{l!}\lim_{\eta\to0}\sum_{0 \neq m\in\mathbb{Z}}\frac{\mathbf{e}(m\omega)\upsilon^{\kappa}}{(-2\pi m)^{\kappa}}e^{-2\pi\eta m/\upsilon}\Gamma\big(\kappa,-\tfrac{2\pi\eta m}{\upsilon}\big)\bigg) \\ &=\frac{\sqrt{2\pi}}{\upsilon\kappa}\bigg(\frac{-i^{\kappa+l}\operatorname{He}_{\kappa+l}(0)}{l!}-\kappa!P_{l}(0)\sum_{0 \neq m\in\mathbb{Z}}\frac{\mathbf{e}(m\omega)\upsilon^{\kappa}}{(2\pi im)^{\kappa}}\bigg) \\ &=\frac{\sqrt{2\pi}}{\upsilon\kappa}\bigg(\frac{-i^{\kappa+l}\operatorname{He}_{\kappa+l}(0)}{l!}+P_{l}(0)\upsilon^{\kappa}\mathbb{B}_{\kappa}(\omega)\bigg),
\end{align*}
via Corollary \ref{Hermite} and Equation \eqref{Bernoulli}, up to an error term of $o_{\kappa,l,\epsilon}(e^{-2\pi^{2}(1-\epsilon)/\upsilon^{2}})$. From the definition, it is also clear that
\begin{equation} \label{omega0lim}
\mathbf{G}_{\kappa,l}(0;\upsilon)=\lim_{\omega\to0^{+}}\big(\mathbf{G}_{\kappa,l}(\omega+\mathbb{Z};\upsilon)-\mathbf{g}_{\kappa,l}(\upsilon\omega;0)\big). \end{equation}
Equation \eqref{gkldef} implies that $-\lim_{\xi\to0}\mathbf{g}_{\kappa,l}(\xi;0)$ vanishes for $\kappa>1$, and we have \[-\lim_{\xi\to0^{+}}\mathbf{g}_{1,l}(\xi;0)=\lim_{\xi\to0^{+}}(-1)^{l+1}h_{l}(\xi)=I_{-1,l}(0,0)=P_{l}(0)\sqrt{\pi/2}+Q_{l}(0)\] by Remark \ref{I-1j00}. Substituting these into Equation \eqref{omega0lim}, and applying Equation \eqref{Bernoulli} for $\kappa=1$, completes the proof for $\kappa\geq1$.

For $\kappa\leq0$, we will first consider the case where $l=1$, in which \[\mathbf{g}_{\kappa,1}(\xi;0)=\xi^{\kappa}h_{0}(\xi)-\xi^{\kappa-1}h_{1}(\xi)=-\xi^{\kappa-1}\mathrm{g}(\xi)\quad\mathrm{hence}\quad \mathbf{G}_{\kappa,1}(\omega;\upsilon)=-\lim_{\eta\to0^{+}}\sum_{0 \neq \xi\in\mathbb{Z}+\omega}\mathrm{g_{\kappa}}(\upsilon\xi;\eta).\] Assuming that $\omega\neq0$, we apply Equation \eqref{Gkleta0}, the Poisson summation formula again, Lemma \ref{egFT}, Lemma \ref{asympI}, Equation \eqref{Appell}, and Equation \eqref{Fqsdef} (with its analytic continuation), which compares $\mathbf{G}_{\kappa,1}(\omega;\upsilon)$ with
\begin{align} \label{Gk1exp}
&-\frac{\sqrt{2\pi}i^{\kappa-1}}{\upsilon}\bigg(\lim_{\eta\to0^{+}}I_{-1,-\kappa}(\eta,0)+\lim_{\eta\to0^{+}}\sum_{0 \neq m\in\mathbb{Z}}e^{-2\pi\eta m/\upsilon}I_{-1,-\kappa}\big(\eta,\tfrac{2\pi m}{\upsilon}\big)\mathbf{e}(m\omega)\bigg)\nonumber \\ \nonumber
&=-\frac{\sqrt{2\pi}i^{\kappa-1}}{\upsilon}\bigg(I_{-1,-\kappa}(0,0)+\sum_{m=1}^{\infty}\Big(P_{|\kappa|}\big(\tfrac{2\pi m}{\upsilon}\big)+o_{\kappa,\epsilon}(e^{-2\pi^{2}(1-\epsilon)m^{2}/\upsilon^{2}})\Big)\sqrt{2\pi}\mathbf{e}(m\omega)\bigg) \\ &=-\frac{\sqrt{2\pi}i^{\kappa-1}}{\upsilon}\bigg(I_{-1,-\kappa}(0,0)+\sqrt{2\pi}\sum_{j=0}^{|\kappa|}\big(\tfrac{2\pi}{\upsilon}\big)^{j}\frac{P_{|\kappa|-j}(0)}{j!}F(\mathbf{e}(\omega),-j)\bigg) +o_{\kappa,\epsilon}(e^{-2\pi^{2}(1-\epsilon)/\upsilon^{2}}).
\end{align}
Now, the summands with $j<|\kappa|$ give $O_{\kappa,\omega}(\upsilon^{\kappa})$, and the same applies for the first term when $\kappa\leq-1$. Since $I_{-1,0}(0,0)=\sqrt{\frac{\pi}{2}}$ by Remark \ref{I-1j00}, this is indeed the desired value, since $P_{1}(0)=0$. For $\omega=0$ we apply Equation \eqref{omega0lim}, where we have seen that the second term there is now $+\frac{\mathrm{g}(\upsilon\omega)}{(\upsilon\omega)^{|\kappa|+1}}$. We expand the term $F\big(\mathbf{e}(\omega),\kappa\big)$ from Equation \eqref{Gk1exp} as in Equation \eqref{Fexp}, and observe that the singularities in $\omega$ cancel with those of the Laurent expansion of $\frac{\mathrm{g}(\upsilon\omega)}{(\upsilon\omega)^{|\kappa|+1}}$, which is $\sum_{\nu=0}^{\infty}i^{\nu}P_{\nu}(0)(\upsilon\omega)^{\nu+\kappa-1}$ by Equation \eqref{Hermitegen} and Corollary \ref{Hermite}. Substituting into the limit from Equation \eqref{omega0lim} yields \[-\frac{\sqrt{2\pi}i^{\kappa-1}}{\upsilon}\bigg(I_{-1,-\kappa}(0,0)-\sqrt{2\pi}\sum_{j=0}^{|\kappa|}\big(\tfrac{2\pi}{\upsilon}\big)^{j}\frac{P_{|\kappa|-j}(0)}{(j+1)!}(B_{j+1}+\delta_{j,0})\bigg) +1+o_{\kappa,\epsilon}(e^{-2\pi^{2}(1-\epsilon)/\upsilon^{2}}),\] where again the same terms (and the 1) go into the error term. Since for $\kappa=0$ the two terms cancel, the result follows also in this case.

We now consider the case $l=0$, where Equations \eqref{Gkldef}, \eqref{Gkleta0}, and \eqref{gkldef} and the trick from the proof of Lemma 8.5 of \cite{[BFI]} evaluate $\mathbf{G}_{\kappa,0}(\omega;\upsilon)$ as
\begin{equation} \label{expsume}
\sum_{0 \neq \xi\in\mathbb{Z}+\omega}\frac{\mathrm{e}(\upsilon\xi)}{(\upsilon\xi)^{|\kappa|+1}}=-\!\!\sum_{0 \neq \xi\in\mathbb{Z}+\omega}\frac{\operatorname{sgn}(\xi)}{(\upsilon\xi)^{|\kappa|+1}}\operatorname{CT}_{s=0}(\upsilon|\xi|)^{-s}\bigg(\int_{0}^{\infty}e^{-w^{2}/2}w^{s}dw-\int_{0}^{\upsilon|\xi|}e^{-w^{2}/2}w^{s}dw\bigg). \end{equation}
Recalling the Hurwitz zeta function $\zeta(s,z):=\sum_{n=1}^{\infty}\frac{1}{(n+z)^{s}}$, the first term in Equation \eqref{expsume} is the constant term at $s=0$ of $-\frac{2^{(s-1)/2}}{\upsilon^{|\kappa|+1+s}}\Gamma\big(\frac{s+1}{2}\big)$ times \[\sum_{0<\xi\in\mathbb{Z}+\omega}\frac{1}{\xi^{|\kappa|+1+s}}+\sum_{0<\xi\in\mathbb{Z}-\omega}\frac{(-1)^{\kappa}}{\xi^{|\kappa|+1+s}}=\zeta(|\kappa|+1+s,\omega)+(-1)^{\kappa}\zeta(|\kappa|+1+s,-\omega),\] where $\pm\omega$ here means the corresponding representatives in $(0,1]$. Since $\zeta(m+1+s,z)$ with $m\in\mathbb{N}$ expands as $\frac{\delta_{m,0}}{s}-\frac{(-1)^{m}\psi^{(m)}(z)}{m!}+O(s)$ (see Equation (9) on Section 1.10 of \cite{[EMOT]} for $m=0$ and just Equation \eqref{polygamma} for $m>0$), the Taylor expansion of the remaining functions and the value $-\gamma-2\log2$ of $\psi\big(\frac{1}{2}\big)$ produce the constant term \[-\sqrt{2\pi}\upsilon^{\kappa-1}\bigg(-\frac{\tilde{\psi}^{(|\kappa|)}(-\omega)+(-1)^{\kappa}\tilde{\psi}^{(|\kappa|)}(\omega)}{2|\kappa|!}-\delta_{\kappa,0}\frac{\gamma+\log2+2\log\upsilon}{2}\bigg),\] which is the desired expression since $C_{0}=\frac{\gamma+\log2}{2}$ by Equation \eqref{Cldef}. The second term in Equation \eqref{expsume} becomes, after a simple substitution, \[\operatorname{CT}_{s=0}\sum_{0 \neq \xi\in\mathbb{Z}+\omega}\int_{0}^{1}\frac{\mathrm{g}(\upsilon\rho\xi)}{(\upsilon\xi)^{|\kappa|}}\rho^{s}d\rho=-\operatorname{CT}_{s=0}\int_{0}^{1}\mathbf{G}_{\kappa+1,1}(\omega;\upsilon\rho)\rho^{s-\kappa}d\rho.\] For $\kappa\leq-1$ our expression for $\mathbf{G}_{\kappa+1,1}(\omega;\upsilon\rho)$ is $O_{\kappa}\big((\upsilon\rho)^{\kappa}\big)$, hence the integral converges at $s=0$ and is $O_{\kappa}(\upsilon^{\kappa})$. When $\kappa=0$ we have $\mathbf{G}_{1,1}(\omega;\upsilon\rho)=-\frac{\sqrt{2\pi}}{\upsilon\rho}-\sqrt{2\pi}\Xi_{1}(\omega)$ up to rapidly decreasing functions, so that the integral is $-\frac{\sqrt{2\pi}}{\upsilon s}$ (with no constant term), again plus $O(1)=O_{\kappa}(\upsilon^{\kappa})$. This proves the result for $l=0$ as well.

For general $l\in\mathbb{N}$, Equation \eqref{gkldef} and Remark \ref{tildePil} allow us to write
\begin{align*}
\mathbf{G}_{\kappa,l}(\omega;\upsilon)&=\sum_{0 \neq \xi\in\mathbb{Z}+\omega}\mathbf{g}_{\kappa,l}(\upsilon\xi;0)=P_{l}(0)\sum_{0 \neq \xi\in\mathbb{Z}+\omega}\frac{\mathrm{e}(\upsilon\xi)}{(\upsilon\xi)^{|\kappa|+1}}+\sum_{0\neq\xi\in\mathbb{Z}+\omega}\Pi_{l}(\upsilon\xi,0)\frac{\mathrm{g}(\upsilon\xi)}{(\upsilon\xi)^{|\kappa|+1}} \\ &=P_{l}(0)\mathbf{G}_{\kappa,0}(\omega;\upsilon)+Q_{l}(0)\mathbf{G}_{\kappa,1}(\omega;\upsilon)+\sum_{0\neq\xi\in\mathbb{Z}+\omega}\frac{\tilde{\Pi}_{l}(\upsilon\xi,0)}{\upsilon\xi} \frac{\mathrm{g}(\upsilon\xi)}{(\upsilon\xi)^{|\kappa|}},
\end{align*}
where $\frac{\tilde{\Pi}_{l}(\upsilon\xi)}{\upsilon\xi}$ is a polynomial in $\upsilon\xi$. The first two terms now give the desired formula, up to $\frac{\sqrt{2\pi}}{\upsilon}P_{l}(0)(C_{l}-C_{0})$ in case $\kappa=0$. When $\kappa\leq-1$ it suffices to view the third term as a linear combination of $\mathbf{G}_{j,1}(\omega;\upsilon)$ with $j\geq\kappa+1$, all of which are $O(\upsilon^{j-1})$ when $j\leq0$ and $O\big(\frac{1}{\upsilon}\big)$ in case $j>0$, since these are all $O(\upsilon^{\kappa})$. For $\kappa=0$ we evaluate $\sum_{\xi\in\mathbb{Z}+\omega}\frac{\tilde{\Pi}_{l}(\upsilon\xi,0)}{\upsilon\xi}\mathrm{g}(\upsilon\xi)$ using the Poisson Summation Formula, where all the Fourier terms with $m\neq0$ give $o_{l,\epsilon}(e^{-2\pi^{2}(1-\epsilon)/\upsilon^{2}})$ once again. Finally, $\frac{1}{\upsilon}$ times the 0\textsuperscript{th} Fourier term is $\frac{\sqrt{2\pi}}{\upsilon}E_{l}(0)$ by Equation \eqref{Eldef}, which is precisely the required expression by Lemma \ref{Elconst} and Equation \eqref{Cldef}. This completes the proof of the proposition.
\end{proof}

\section{Nearly Holomorphic Modular Forms \label{NWHMF}}

In this section we shall prove our result in the most general case, evaluating the Shintani lift $\mathcal{I}_{k,L}(\tau,f)$ from Equation \eqref{Shindef} for nearly holomorphic modular forms $f\in\widetilde{M}_{2k}^{!}(\Gamma)$. For comparing our formulae with those from \cite{[ANS]}, consult Remark \ref{thetadiff}.

\subsection{Traces and Regularizations}

Recall that if $\lambda \in L^{*}$ satisfies $Q(\lambda)<0$ then the stabilizer $\Gamma_{\lambda}$ of $\lambda$ in $\Gamma$ is finite, and $\lambda$ is a multiple of $Z^{\perp}(z_{\lambda})$ for a unique $z_{\lambda}\in\mathcal{H}$. We then define, for every $k\in\mathbb{Z}$ and $f\in\mathcal{A}_{0}^{!}(\Gamma)$, the \emph{trace}
\begin{equation} \label{TrQneg}
\operatorname{Tr}_{\lambda}^{(k)}(f):=\frac{\big[-\operatorname{sgn}\big(\lambda,Z^{\perp}(z_{\lambda})\big)\big]^{k}}{|\Gamma_{\lambda}|}f(z_{\lambda}). \end{equation}
If $Q(\lambda)>0$ then we recall the geodesic $c_{\lambda}\subseteq\mathcal{H}$ and its image $c(\lambda) \subseteq X$ from Equation \eqref{geodef}, and that when $\lambda$ is not split-hyperbolic, i.e., when $\iota(\lambda)=0$ in the notation of Equation \eqref{iotash}, the latter is a closed geodesic inside the open modular curve $Y$. We can then define, for every $g\in\mathcal{A}_{2k}^{!}(\Gamma)$, the \emph{trace}
\begin{equation} \label{Triota0}
\operatorname{Tr}_{\lambda}(g):=\oint_{c(\lambda)}g(z)\big(\lambda,Z(z)\big)^{k-1}dz.
\end{equation}

On the other hand, when $\lambda$ is split-hyperbolic, i.e., when $\iota(\lambda)=1$, the image $c(\lambda)$ of $c_{\lambda}$ in $Y$ is not compact, and if $g$ grows towards the cusps then the integral corresponding to that from Equation \eqref{Triota0} does not converge. We shall regularize it only for nearly  holomorphic modular forms, i.e., $g\in\widetilde{M}_{2k}^{!}(\Gamma)$. Then its Fourier expansion near the cusp associated with some $\ell\in\operatorname{Iso}(V)$ is given, in the coordinates from Equation \eqref{coorell}, by
\begin{equation} \label{Fourg}
g_{\ell}(z_{\ell}):=(g\mid_{2k}\sigma_{\ell})(z_{\ell})=\sum_{l=0}^{p}\frac{g_{\ell,l}(z_{\ell})}{y_{\ell}^{l}}=\sum_{l=0}^{p}\sum_{n\in\mathbb{Z}}\frac{c_{\ell}(n,l)q_{\ell}^{n}}{y_{\ell}^{l}}= \sum_{l=0}^{p}\sum_{n\leq0}\frac{c_{\ell}(n,l)q_{\ell}^{n}}{y_{\ell}^{l}}+g_{\ell}^{0}(z_{\ell}),
\end{equation}
where $p$ is the depth of $g$, $c_{\ell}(n,l)=0$ for all $0 \leq l \leq p$ when $n\ll0$, and the latter decomposition is into the (finite) principal part and the cuspidal part $g_{\ell}^{0}$. Recalling our extension of the incomplete Gamma function in \eqref{GammamuZt}, and the truncated modular curve $Y_{T}$ from Equation \eqref{trundom} for any $T>1$. We also set for $n$ and $\kappa$ from $\mathbb{Z}$, positive reals $c$ and $T$, split-hyperbolic $\lambda \in V$, and $g\in\widetilde{M}_{2k}^{!}(\Gamma)$ with expansion as in Equation \eqref{Fourg} for $\ell=\ell_{\lambda}$, the quantities
\begin{equation} \label{Singdef}
\begin{split} \phi_{n}(\kappa,T;r)&:=\begin{cases} \Gamma(\kappa,rnT)/(rn)^{\kappa}, & n\neq0, \\ -T^{\kappa}/\kappa, & n=0\mathrm{\ and\ }\kappa\neq0, \\ -\log T, & n=0{\ and\ }\kappa=0 \end{cases}\qquad\mathrm{and} \\ \operatorname{Sing}_{\lambda}(g,T)&:=i^{k}(2\sqrt{Q(\lambda)})^{k-1}\sum_{l=0}^{p}\sum_{n\in\mathbb{Z}}c_{\ell_{\lambda}}(n,l)\mathbf{e}\bigg(\frac{nr_{\lambda}}{\alpha_{\ell_{\lambda}}}\bigg) \phi_{n}\bigg(k-l,T;\frac{2\pi}{\alpha_{\ell_{\lambda}}}\bigg). \end{split}
\end{equation}
Note that when $n=0$, $\phi_{0}$ is independent of $r$, and we can then omit it from the notation. In addition, assuming that $f\in\widetilde{M}_{2k}^{!}(\Gamma)$ expands as in Equation \eqref{Fourg}, the weight lowering property of the operator $L_{z}$ implies that for every $\nu\in\mathbb{N}$ we can write $(L_{z}^{\nu}f)_{\ell}(z_{\ell})$ as
\begin{equation} \label{FourLf}
(L_{z}^{\nu}f\mid_{2k-2\nu}\sigma_{\ell})(z_{\ell})=L_{z_{\ell}}^{\nu}f_{\ell}(z_{\ell})=(-1)^{\nu}\sum_{l=\nu}^{p}\frac{l!f_{\ell,l}(z_{\ell})}{(l-\nu)!y_{\ell}^{l}}= (-1)^{\nu}\sum_{n\in\mathbb{Z}}\sum_{l=\nu}^{p}\frac{l!c_{\ell}(n,l)q_{\ell}^{n}}{(l-\nu)!y_{\ell}^{l-\nu}}.
\end{equation}

We can now define the \emph{trace} to be
\begin{equation} \label{Triota1}
\operatorname{Tr}_{\lambda}(g):=\lim_{T\to\infty}\Bigg(\int_{c(\lambda) \cap Y_{T}}g(z)\big(\lambda,Z(z)\big)^{k-1}dz+\operatorname{Sing}_{\lambda}(g,T)+(-1)^{k}\operatorname{Sing}_{-\lambda}(g,T)\Bigg).
\end{equation}
We now prove that this is a regularization of the required trace.
\begin{prop} \label{regindep}
The quantity in the limit defining $\operatorname{Tr}_{\lambda}(g)$ in Equation \eqref{Triota1} exists and is independent of the choice of $T$ when it is sufficiently large.
\end{prop}

\begin{proof}
Since $c_{\lambda}$ only intersects the cusps $\ell_{\pm\lambda}$, there exists some $R>1$ such that for all $T>R$, we find that $c(\lambda) \cap Y_{T} \cong c_{\lambda}\cap\mathcal{H}_{T}$ is contained in $\mathcal{H}_{R}\cup\sigma_{\ell_{\lambda}}\mathcal{F}_{T}^{\alpha_{\ell_{\lambda}}}\cup\sigma_{\ell_{-\lambda}}\mathcal{F}_{T}^{\alpha_{\ell_{-\lambda}}}$. We thus obtain
\[\int_{c(\lambda) \cap Y_{T}}g(z)\big(\lambda,Z(z)\big)^{k-1}dz=\int_{c_{\lambda}\cap\mathcal{H}_{R}}\tilde{g}(z)dz+\int_{c_{\lambda}\cap\sigma_{\ell_{\lambda}}\mathcal{F}_{T}^{\alpha_{\ell_{\lambda}}}\setminus\mathcal{H}_{R}}\tilde{g}(z)dz+ \int_{c_{\lambda}\cap\sigma_{\ell_{-\lambda}}\mathcal{F}_{T}^{\alpha_{\ell_{-\lambda}}}\setminus\mathcal{H}_{R}}\tilde{g}(z)dz\] for every $T>R$, where we wrote $\tilde{g}(z)$ for $g(z)\big(\lambda,Z(z)\big)^{k-1}$. The first term is independent of $T$, and if we change, in the integral corresponding to $\pm\lambda$, the variable to $z_{\ell}$ for $\ell=\ell_{\pm\lambda}$ from Equation \eqref{coorell}, then it becomes \[(\pm1)^{k}\int_{r_{\pm\lambda}+iR}^{r_{\pm\lambda}+iT}g_{\ell}(z_{\ell})\big(\sigma_{\ell}^{-1}(\pm\lambda),Z(z_{\ell})\big)^{k-1}dz_{\ell}=(\pm i)^{k}(2\sqrt{Q(\lambda)})^{k-1}\int_{R}^{T}g_{\ell}(r_{\pm\lambda}+iy_{\ell})y_{\ell}^{k-1}dy_{\ell}\] via Equation \eqref{rlambda}. This expression is a differentiable function of $T$, and Equations \eqref{Fourg} and \eqref{derGamma} show that its derivative is minus that of $\operatorname{Sing}_{\lambda}(g,T)$ from Equation \eqref{Singdef}. Hence the expression from Equation \eqref{Triota1} is independent of $T$ as long as $T>R$, and in particular $\operatorname{Tr}_{\lambda}(g)$ exists. This proves the proposition.
\end{proof}
Note that $\lim_{T\to\infty}\phi_{n}(\kappa,T;r)=0$ for any $r>0$ when $n>0$, and the integral of the part $g_{\ell_{\pm\lambda}}^{0}$ from Equation \eqref{Fourg} converges as $T\to\infty$. Hence the regularization from Proposition \ref{regindep} is essentially only of the integral of the principal part.

Following \cite{[ANS]} and \cite{[BFI]}, we now give an equivalent expression for the regularized theta lift $\mathcal{I}_{k,L}(\tau,f)$ of $f$ from Equation \eqref{Shindef}.
\begin{prop} \label{regcomp}
For $f\in\widetilde{M}^{!}_{2k}(\Gamma)$ with asymptotic expansion at the cusp $\ell$ as in Equation \eqref{Fourg}, the regularized theta lift $\mathcal{I}_{k,L}(\tau,f)$ of $f$ from Equation \eqref{Shindef} can be written as
\[\lim_{T\to\infty}\Bigg(\int_{Y_{T}}f(z)\Theta_{k,L}(\tau,z)d\mu(z)+\sum_{\ell\in\Gamma\backslash\operatorname{Iso}(V)}i^{k}\frac{\varepsilon_{\ell}}{\sqrt{N}}\Theta_{k,\ell}(\tau)\sum_{l=0}^{p}c_{\ell}(0,l) \phi_{0}(k-l,T)\Bigg),\] where $\phi_{0}$ is defined in Equation \eqref{Singdef}\footnote{Note that $c_{\ell}(0,l)$ is well-defined for $\ell\in\Gamma\backslash\operatorname{Iso}(V)$ by the modularity of $f$.}.
\end{prop}

\begin{proof}
The proof uses Lemma \ref{Thetanearell} and proceeds in the same way as that of Proposition 5.2 of \cite{[ANS]}, with the only difference being that now we have
\[\int_{1}^{\infty}\int_{0}^{\alpha_{\ell}}(f\mid_{2k}\sigma_{\ell})(z_{\ell})y_{\ell}^{k-1-s}dx_{\ell}dy_{\ell}=\alpha_{\ell}\sum_{l=0}^{p}c_{\ell}(0,l)\int_{1}^{\infty}y_{\ell}^{k-s-1-l}dy_{\ell}=\alpha_{\ell}\sum_{l=0}^{p} \frac{c_{\ell}(0,l)}{k-l-s}.\]
\end{proof}

For $k$, $f$, and $g$ as above and $m\neq0$, we can define the combinations
\begin{equation} \label{Trmh}
\operatorname{Tr}_{m,h}^{(k)}(f):=\sum_{\lambda\in\Gamma \backslash L_{m,h}}\operatorname{Tr}_{\lambda}^{(k)}(f)\text{ for }m<0,\quad \operatorname{Tr}_{m,h}(g):=\sum_{\lambda\in\Gamma \backslash L_{m,h}}\operatorname{Tr}_{\lambda}(g)\text{ for }m>0.
\end{equation}
Note that $\Gamma \backslash L_{m,h}$ is finite when $m\neq0$, so that there is no question of convergence in Equation \eqref{Trmh}. On the other hand, when $Q(\lambda)=m=0$, we will define the trace to be
\begin{equation} \label{Tr0hg}
\operatorname{Tr}_{0,h}(g):=\sum_{\ell\in\Gamma\backslash\operatorname{Iso}(V)}\frac{\varepsilon_{\ell}}{\sqrt{N}}\iota_{\ell}(0,h)c_{\ell}(0,0)(\sqrt{N}\beta_{\ell})^{k}\Phi_{k}(\omega_{\ell,h}), \end{equation}
where $\iota_{\ell}(m,h)$, $\omega_{\ell,h}$, and $\Phi_{\kappa}$ are defined in Equations \eqref{iotaellmh}, \eqref{komegaellh}, and \eqref{Phidef} respectively.

The main term of the Shintani lift from Theorem \ref{Shintot} below will have the traces from Equations \eqref{Trmh} and \eqref{Tr0hg} as coefficients. However, for the terms with $\iota(m)=1$ we need to define some corrections. Recall that when $m>0$ and $\iota_{\ell}(m,h)=1$, Remark \ref{rporbs} implies that the numbers $r_{\lambda}$ for oriented $\lambda \in L_{m,h}\cap\ell^{\perp}$ are all the same modulo $\frac{\beta_{\ell}}{2}\sqrt{\frac{N}{m}}\mathbb{Z}$. We can thus define, for our element $f\in\widetilde{M}^{!}_{2k}(\Gamma)$ of depth $p$ expanded as in Equation \eqref{Fourg} near each cusp $\ell\in\Gamma\backslash \operatorname{Iso}(V)$, the \emph{complementary trace}
\begin{equation} \label{Trcdef}
\begin{split} &\operatorname{Tr}^{\mathrm{c}}_{m,h}(f,v):=-(-2i\sqrt{m})^{k}\cdot\sqrt{2\pi}\sum_{\ell\in\Gamma\backslash\operatorname{Iso}(V)}\big(\iota_{\ell}(m,h)+(-1)^{k}\iota_{\ell}(m,-h)\big)\frac{\varepsilon_{\ell}}{\sqrt{N}} \\ &\times\sum_{\substack{0>n\in\mathbb{Z} \\ n\equiv0\bmod 2\varepsilon_{\ell}\sqrt{m/N}}}\mathbf{e}\bigg(\frac{nr_{\lambda}}{\alpha_{\ell}}\bigg)\sum_{l=k}^{p}\bigg(\frac{2\pi n}{\alpha_{\ell}}\bigg)^{l-k}\frac{l!c_{\ell}(n,l)}{(l-k)!}\cdot\frac{J_{l}\big(2\sqrt{2\pi mv}\big)}{\big(2\sqrt{2\pi mv}\big)^{l}}, \end{split}
\end{equation}
with $J_{l}(\eta)$ from Equation \eqref{Jnudef}. Note that the sum over $n$ in Equation \eqref{Trcdef} is finite, and is empty for all but finitely many values of $m$.

We also define the \emph{complementary trace from constants}
\begin{equation} \label{Trccdef}
\begin{split} &\operatorname{Tr}^{\mathrm{cc}}_{m,h}(f,v):=-(-2i\sqrt{m})^{k}\sum_{\ell\in\Gamma\backslash\operatorname{Iso}(V)}\big(\iota_{\ell}(m,h)+(-1)^{k}\iota_{\ell}(m,-h)\big)\frac{\varepsilon_{\ell}}{\sqrt{N}} \\ &\times k!c_{\ell}(0,k)\frac{I_{k}\big(2\sqrt{2\pi mv}\big)-\tilde{\Omega}_{k}\big(2\sqrt{2\pi mv}\big)}{\big(2\sqrt{2\pi mv}\big)^{k}}, \end{split}
\end{equation}
where $I_{k}$ and $\tilde{\Omega}_{k}$ are defined in Equations \eqref{Ikdef} and Remark \ref{tildeOmega} respectively.

When $m=0$, there is only a complementary trace from constants, which is defined by
\begin{equation} \label{Trcc0h}
\begin{split} &\operatorname{Tr}^{\mathrm{cc}}_{0,h}(f,v):=\delta_{k,0}\delta_{h,0}\sqrt{v}\int_{Y}^{\reg}f(z)d\mu(z)+\sum_{\ell\in\Gamma\backslash\operatorname{Iso}(V)}\iota_{\ell}(0,h)\frac{\varepsilon_{\ell}}{\sqrt{N}}\times \\ &\Bigg(-c_{\ell}(0,k)k!P_{k}(0)\frac{\log\big(\sqrt{2\pi Nv}\beta_{\ell}\big)+C_{k}}{(2\pi v)^{k/2}}+\sum_{l=0}^{p}l!c_{\ell}(0,l)Q_{l}(0)\frac{(\sqrt{N}\beta_{\ell})^{k-l}\Xi_{k-l}(\omega_{\ell,h})}{(2\pi v)^{l/2}}\Bigg), \end{split} \end{equation}
where $C_{k}$ and $\Xi_{\kappa}$ are defined in Equations \eqref{Cldef} and \eqref{Xidef} respectively, and for $f\in\widetilde{M}^{!}_{2k}(\Gamma)$ the regularized integral $\int_{Y}^{\reg}f(z)d\mu(z)$ is the (convergent) limit $\lim_{T\to\infty}\int_{Y_{T}}f(z)d\mu(z)$.

\subsection{Main Theorem and Proof}

We can now state and prove our main theorem. Given $k\in\mathbb{N}$ and an element $f\in\widetilde{M}_{2k}^{!}(\Gamma)$, we gather the traces from Equations \eqref{Trmh} and \eqref{Tr0hg} and define
\begin{equation} \label{Inhdef}
\mathcal{I}_{k,L,h}^{\mathrm{nh}}(\tau,f):=\sum_{b=0}^{\lfloor p/2 \rfloor}\sum_{0 \leq m\in\mathbb{Z}+Q(h)}\frac{\operatorname{Tr}_{m,h}(L_{z}^{2b}f)}{(4\pi v)^{b}b!}q_{\tau}^{m},
\end{equation}
which is a nearly holomorphic function of depth $\big\lfloor\frac{p}{2}\big\rfloor$ on $\mathcal{H}$ that is bounded at $\infty$. Using the negative index case of Equation \eqref{Trmh}, we also define
\begin{equation} \label{Inegdef}
\mathcal{I}_{k,L,h}^{\mathrm{neg}}(\tau,f):=\sum_{0>m\in\mathbb{Z}+Q(h)}\sum_{l=k}^{p}\frac{4^{k}\sqrt{\pi}|m|^{\frac{k-1}{2}}h_{l}\big(2\sqrt{2\pi|m|v}\big)}{\sqrt{2}\big(4\sqrt{2\pi|m|v}\big)^{l}(l-k)!} \operatorname{Tr}_{m,h}^{(k)}(R_{2k-2l}^{l-k}L_{z}^{l}f)q_{\tau}^{m},
\end{equation}
which resembles the non-holomorphic part of a harmonic weak Maass form with holomorphic $\xi$-image. We also gather the traces from Equations \eqref{Trcdef}, \eqref{Trccdef}, and \eqref{Trcc0h}, and set
\begin{equation} \label{Icccdef}
\mathcal{I}_{k,L,h}^{\mathrm{c}}(\tau,f):=\sum_{\substack{0<m\in\mathbb{Z}+Q(h) \\ \iota(m)=1}}\operatorname{Tr}^{\mathrm{c}}_{m,h}(f,v)q_{\tau}^{m}\quad\mathrm{and}\quad\mathcal{I}_{k,L,h}^{\mathrm{cc}}(\tau,f):=\sum_{\substack{0 \leq m\in\mathbb{Z}+Q(h) \\ \iota(m)=1}}\operatorname{Tr}^{\mathrm{cc}}_{m,h}(f,v)q_{\tau}^{m},
\end{equation}
where the former is a finite sum of increasing terms, and the second one is infinite but converges.

The main result, which evaluates the regularized Shintani lift of $f$, now reads as follows.
\begin{thm} \label{Shintot}
Write the vector-valued Shintani lift $\mathcal{I}_{k,L}(\tau,f)$ of $f\in\widetilde{M}_{2k}^{!}(\Gamma)$, defined in Equation \eqref{Shindef}, as $\sum_{h \in D_{L}}\mathcal{I}_{k,L,h}(\tau,f)\mathfrak{e}_{h}$. Then the scalar-valued coefficient associated with $h \in D_{L}$ is given by \[\mathcal{I}_{k,L,h}(\tau,f)=\mathcal{I}_{k,L,h}^{\mathrm{nh}}(\tau,f)+\mathcal{I}_{k,L,h}^{\mathrm{neg}}(\tau,f)+\mathcal{I}_{k,L,h}^{\mathrm{c}}(\tau,f)+\mathcal{I}_{k,L,h}^{\mathrm{cc}}(\tau,f),\] where the terms are defined in Equations \eqref{Inhdef}, \eqref{Inegdef}, and \eqref{Icccdef}.
\end{thm}

\begin{proof}
We apply Proposition \ref{regcomp}, and expand $\Theta_{k,L}(\tau,z)$ and $\Theta_{k,\ell}(\tau)$ using Equations \eqref{Schwarzcomp} and \eqref{expwithiota} respectively, which gives
\begin{equation} \label{limTIkLh}
\mathcal{I}_{k,L,h}(\tau,f)=\lim_{T\to\infty}\left(\begin{array}{c}\displaystyle v^{\frac{1-k}{2}}\int_{Y_{T}}f(z)\sum_{m\in\mathbb{Z}+Q(h)}q_{\tau}^{m}\sum_{\lambda \in L_{m,h}}\varphi_{k,-1}\big(\sqrt{v}\lambda,z\big)d\mu(z)+ \\ \displaystyle\sum_{\substack{0 \leq m\in\mathbb{Z}+Q(h) \\ \iota(m)=1}}q_{\tau}^{m}\sum_{\ell\in\Gamma\backslash\operatorname{Iso}(V)}i^{k}\frac{\varepsilon_{\ell}}{\sqrt{N}}a(\Theta_{k,\ell},m,h,v)\sum_{l=0}^{p}c_{\ell}(0,l)\phi_{0}(k-l,T)\end{array}\right). \end{equation}
We may interchange the order of integration and summation as both are absolutely convergent for fixed $T$. Propositions \ref{intnegm}, \ref{intiota0}, \ref{intiota1}, and \ref{intm0} below now evaluate the coefficient of $q_{\tau}^{m}$ to be the one implied by the asserted sum. This proves the theorem.
\end{proof}

A much simpler but interesting special case is the one where the depth $p<k$.
\begin{cor} \label{nhIpk}
For $f\in\widetilde{M}_{2k}^{!,\leq p}(\Gamma)$ with $p<k$, we have \[\mathcal{I}_{k,L}(\tau,f)=\sum_{h \in D_{L}}\mathcal{I}^{\mathrm{nh}}_{k,L,h}(\tau,f)\mathfrak{e}_{h}\in\widetilde{M}_{k+\frac{1}{2}}^{\leq\lfloor p/2 \rfloor}(\rho_{L}).\] For $p=0$, we recover the part of Theorem 6.1 in \cite{[ANS]} with weakly holomorphic input.
\end{cor}

\begin{proof}
If $p<k$, then $\mathcal{I}^{\mathrm{neg}}_{k,L,h}$ and $\mathcal{I}^{\mathrm{c}}_{k,L,h}$ vanishes identically, and Equations \eqref{Trccdef} and \eqref{Trcc0h} imply that $\mathcal{I}^{\mathrm{cc}}_{k,L,h}$ consists only of the constant term \[\operatorname{Tr}^{\mathrm{cc}}_{0,h}(f,v)=\sum_{\ell\in\Gamma\backslash\operatorname{Iso}(V)}\iota_{\ell}(0,h)\frac{\varepsilon_{\ell}}{\sqrt{N}}\sum_{l=0}^{k-1}l!c_{\ell}(0,l)Q_{l}(0) \frac{(\sqrt{N}\beta_{\ell})^{k-l}\Xi_{k-l}(\omega_{\ell,h})}{(2\pi v)^{l/2}}.\] But using Equations \eqref{Xidef}, \eqref{betaelldef}, \eqref{iotaellmh}, and \eqref{coorell}, this becomes just a multiple of $\delta_{h,0}$ times the sum of the residues of $L_{z}^{k-1}fdz$ at the cusps, with $L_{z}^{k-1}f \in M_{2}^{!}(\Gamma)$, which therefore vanishes. This proves the corollary.
\end{proof}

We can now deduce the theorems mentioned in the Introduction.
\begin{proof}[Proof of Theorem \ref{cycjE2}]
We apply Theorem \ref{Shintot} to the scaled lattice $L_{\Delta}:=\Delta L$ for $L$ from Equation \eqref{LatSL2Z}, with $Q=-\frac{\det}{|\Delta|}$. Then $L_{\Delta}^{*}=L^{*}$ and it is well-known (see \cite{[GKZ]} or \cite{[AE]}) that if $g\in\mathcal{A}_{\frac{3}{2},\rho_{L_{\Delta}}}^{!}$ and $\Gamma:=\operatorname{PSL}_{2}(\mathbb{Z})$ then \[\Upsilon_{\Delta}(g)(z):=\frac{1}{[\Gamma:\Gamma_{L_{\Delta}}]}\sum_{\delta \in D_{L_{\Delta}}}\chi_{\Delta}(\delta)g_\delta(4z)\in\mathcal{A}^!_{\frac{3}{2}}(\tilde{\Gamma}_0(4)),\] is in the Kohnen plus space, where $\chi_{\Delta}$ is the character from Equation \eqref{chiDelta} and $\tilde{\Gamma}_0(4)\subset \mathrm{Mp}_2(\Zb)$ is the metaplectic cover of $\Gamma_0(4)$. We denote $\mathcal{I}^{*}_{\Delta}(z,f):=\Upsilon_{\Delta}\big(\mathcal{I}^{*}_{1,L_\Delta}(z,f)\big)$ for $*\in\{\mathrm{nh},\mathrm{neg},\mathrm{c},\mathrm{cc}\}$. When $f=J \cdot E_{2}^{*}$, we have $k=p=1$, and only non-zero coefficients of the principal part are $c_{\ell}(\Delta,0)=1$, $c_{\ell}(\Delta,1)=-\frac{3}{\pi}$, and $c_{\ell}(0,0)=-24$ for any $\ell\in\operatorname{Iso}(V)$. The fact that $c_{\ell}(0,1)=0$ and $Q_{0}=0$ implies that $\mathcal{I}^{\mathrm{cc}}_{\Delta}(z,f)=0$, and we have
\begin{equation} \label{quotSL2Z}
\Gamma\backslash\mathcal{Q}_{\Delta^{2}}=\Gamma\big\{[0,|\Delta|,C]\big|C\in\mathbb{Z}/\Delta\mathbb{Z}\big\},~ \Gamma_{L_{\Delta},\infty}\backslash\ell_{\infty}=\Gamma_{L_{\Delta},\infty}\big\{[0,0,C]\big|C\in\mathbb{Z}/\Delta\mathbb{Z}\big\},
\end{equation}
with the value $\chi_{\Delta}\big([0,|\Delta|,C]\big)=\chi_{\Delta}\big([0,0,C]\big)=\big(\frac{C}{\Delta}\big)$, the real part $r_{[0,|\Delta|,C]}=-\frac{C}{|\Delta|}$, and $\mathbb{B}_{1}(\omega_{\ell_{\infty},[0,0,C]})=\mathbb{B}_{1}\big(\frac{C}{|\Delta|}\big)=\frac{C}{|\Delta|}-\frac{1}{2}$ for $0<C<|\Delta|$. Since $L_{z}f=\frac{3}{\pi}J$, $\chi_{\Delta}$ is anti-symmetric, and the sign from Equation \eqref{TrQneg} is that of $-A$ from $\lambda$ as in Equation \eqref{LatSL2Z}, we have
\[\mathcal{I}^{\mathrm{neg}}_{\Delta}(z,f)=\frac{-12}{\sqrt{2\pi}}\sum_{0<D\in\mathbb{Z}}\Bigg(\sum_{0\ll\lambda\in\Gamma\backslash\mathcal{Q}_{\Delta D}}\frac{\chi_{\Delta}(\lambda)}{|\Gamma_{\lambda}|}J(z_{\lambda})\Bigg)\frac{h_{1}\big(2\sqrt{2\pi Dy}\big)}{2\sqrt{2\pi Dy}}q^{-D}\] as well as  \[\mathcal{I}^{\mathrm{nh}}_{\Delta}(z,f)=48|\Delta|H(-\Delta)+\sum_{0>D\in\mathbb{Z}}\Bigg(\sum_{\lambda\in\Gamma\backslash\mathcal{Q}_{\Delta D}}\chi_\Delta(\lambda)\operatorname{Tr}_{\lambda}(f)\Bigg)q^{-D},\] where for the latter constant we use Equation \eqref{quotSL2Z} and the equality \[\sum_{C=1}^{|\Delta|-1}\bigg(\frac{C}{\Delta}\bigg)\bigg(\frac{C}{|\Delta|}-\frac{1}{2}\bigg)=\frac{1}{|\Delta|}\sum_{C=1}^{|\Delta|-1}C\bigg(\frac{C}{\Delta}\bigg)=\frac{\sqrt{|\Delta|}}{\pi} L\Big(\big(\tfrac{\cdot}{\Delta}\big),1\Big)=2H(-\Delta).\] Moreover, the congruence from Equation \eqref{Trcdef} gives $a|\Delta$, the fundamentality of $\Delta$ leaves only $a=|\Delta|$ after $\Upsilon_{\Delta}$, and then Equation \eqref{quotSL2Z} and a standard Gauss sum evaluation yield $\mathcal{I}^{\mathrm{c}}_{\Delta}(z,f)=\frac{12}{\sqrt{2\pi}}\sqrt{|\Delta|}\cdot\frac{J_{1}\big(2\sqrt{2\pi|\Delta|y}\big)}{2\sqrt{2\pi|\Delta|y}}q^{|\Delta|}$.

Now, let $\tilde{f}_{-\Delta}(z)$ denote the holomorphic part $\mathcal{I}^{\mathrm{nh}}_{\Delta}(z,f)$, which has the required expansion. Since Lemma \ref{diffhnu} and Equation \eqref{Jdiff} imply that $\xi_{\frac{3}{2}}$ takes the functions $\frac{h_{1}(2\sqrt{2\pi Dy})}{(2\sqrt{2\pi Dy})}q^{-D}$ and $\frac{J_{1}\big(2\sqrt{2\pi|\Delta|y}\big)}{2\sqrt{2\pi|\Delta|y}}q^{|\Delta|}$ to $-\frac{q^{D}}{4\sqrt{2\pi D}}$ and $+\frac{q^{\Delta}}{4\sqrt{2\pi|\Delta|}}$ respectively, the complement $\mathcal{I}^{\mathrm{neg}}_{\Delta}(z,f)+\mathcal{I}^{\mathrm{c}}_{\Delta}(z,f)$ of $\tilde{f}_{-\Delta}(z)$ is indeed harmonic, with the asserted $\xi_{\frac{3}{2}}$-image $\frac{3}{2\pi}f_{-\Delta}$. This completes the proof of the theorem.
\end{proof}

\begin{proof}[Proof of Theorems \ref{E2klift} and \ref{liftnoc0k}]
We apply Theorem \ref{Shintot} to the lattice $L$ in Equation \eqref{LatSL2Z}, where $N=\alpha_{\ell}=\beta_{\ell}=\varepsilon_{\ell}=1$, combined with the isomorphism from \cite{[K]} to scalar-valued modular forms. For Theorem \ref{E2klift}, we let $f=\big(\frac{\pi}{3}E_{2}^{*}\big)^{k}$, $c_{n,l}=(-1)^{l}\binom{k}{l}\big(\frac{\pi}{3}\big)^{k-l}\delta_{n,0}$ for $n\leq0$. Then $p=k$, $L_{z}^{2b}f=\frac{k!}{(k-2b)!}(\frac{\pi}{3}E_{2}^{*})^{k-2b}$, $c(0,k)=1$, $\mathcal{I}^{\mathrm{c}}_{k,L,0}(4z,f)=\mathcal{I}^{\mathrm{c}}_{k,L,1}(4z,f)=0$, for $d<0\in\mathbb{Z}$ we have $\operatorname{Tr}_{m,h}(L_{z}^{k}f)=\operatorname{Tr}_{m,h}(k!)=2k!H(-d)$ for $m=\frac{d}{4}$ and $h=\frac{d}{2}+\mathbb{Z}$, for $m=\frac{d}{4}>0$ we have $\iota(m)=1$ if and only if $d=a^{2}$ is a square and $\iota_{\ell}(m,h)=1$ for $h=\mathbb{d}{2}+\mathbb{Z}$ if and only if $a\in\mathbb{N}$, $P_{k}(0)=\frac{1}{2^{k/2}(k/2)!}$ by Corollary \ref{Hermite}, and for $l \leq k$ the expression $Q_{l}(0)\Xi_{k-l}(\omega_{\ell,h})$ is $-\delta_{h,0}\delta_{l,k-1}\frac{Q_{k-1}(0)}{\sqrt{2\pi}}$. Since $\operatorname{Tr}_{d}$ for $d>0$ from Equation \eqref{TrdSL2Z} is $\operatorname{Tr}_{m,h}$ with these $m$ and $h$, and Equation \eqref{Phizeta} reduces the trace from Equation \eqref{Tr0hg} to that from Equation \eqref{Trd0}, we get
\begin{align*}
\mathcal{I}^{\mathrm{nh}}_{k,L,0}(4z,f)+\mathcal{I}^{\mathrm{nh}}_{k,L,1}(4z,f)&=\sum_{b=0}^{k/2}\sum_{d=0}^{\infty}\frac{\frac{k!}{(k-2b)!}\operatorname{Tr}_{d}\big((\frac{\pi}{3}E_{2}^{*})^{k-2b}\big)}{(16\pi y)^{b}b!}q^{d}, \\ \mathcal{I}^{\mathrm{neg}}_{k,L,0}(4z,f)+\mathcal{I}^{\mathrm{neg}}_{k,L,1}(4z,f)&=\sum_{0>d\in\mathbb{Z}}\frac{\sqrt{2\pi}|d|^{\frac{k-1}{2}}h_{k}\big(2\sqrt{2\pi|d|y}\big)2H(-d)k!}{\big(2\sqrt{2\pi|d|y}\big)^{k}}q^{d}, \\ \mathcal{I}^{\mathrm{cc}}_{k,L,0}(4z,f)+\mathcal{I}^{\mathrm{cc}}_{k,L,1}(4z,f)&=-2i^{k}k!\sum_{0<a\in\mathbb{N}}a^{k}\frac{I_{k}\big(2r\sqrt{2\pi y}\big)-\tilde{\Omega}_{k}\big(2a\sqrt{2\pi y}\big)}{\big(2a\sqrt{2\pi y}\big)^{k}}q^{a^{2}} \\ &\quad+\bigg(-k!\frac{\log(\sqrt{8\pi y})+C_{k}}{\big(\frac{k}{2}\big)!(16\pi y)^{k/2}}+\frac{k!\pi}{3}\frac{Q_{k-1}(0)}{\sqrt{2\pi}(8\pi y)^{(k-1)/2}}\bigg),
\end{align*}
where only discriminants $d\equiv0,1(\mathrm{mod\ }4)$ appear. Adding these together and dividing by $k!$ then gives the expression from Theorem \ref{E2klift}. Note that if $d=a^{2}$ for $a\in\mathbb{N}$, then $\operatorname{Tr}_{d}(L_{z}^{2b}f)$ is $i^{k-2b}$ times the regularized completed $L$-function of the image of $L_{z}^{2b}f$ under the Hecke operator $U_{a}$.

To prove the second claim, we can apply Lemmas \ref{diffhnu} and \ref{Ikdiff} and Equations \eqref{Qnu0} and \eqref{Cldef}, and recall the value $\frac{(-1)^{k/2}}{2^{k/2}(k/2)!}$ of $\tilde{\Omega}_{k}(0)$ from Remark \ref{tildeOmega}, to see that applying $(-16\pi L_{z})^{k/2}$ to $(\mathcal{I}_{k,L,0}(4z,f)+\mathcal{I}_{k,L,1}(4z,f))/k!$ gives the holomorphic series \[\sum_{d=0}^{\infty}\operatorname{Tr}_{d}(1)q^{d}+2\sqrt{2\pi}\sum_{0>d\in\mathbb{Z}}H(-d)h_{0}(2\sqrt{2\pi|d|y})\frac{q^{d}}{\sqrt{|d|}}\] from the $\mathcal{I}^{\mathrm{nh}}_{k,L,h}$'s and the $\mathcal{I}^{\mathrm{neg}}_{k,L,h}$'s, while the $\mathcal{I}^{\mathrm{nh}}_{k,L,h}$'s are taken to \[-\bigg(\log\sqrt{8\pi y}+\frac{\gamma+\log 2}{2}-H_{k}\bigg)+\frac{2\pi}{3}\sqrt{y}+\sum_{0<a\in\mathbb{N}}2q^{a^{2}}\Big(I_{0}\big(2a\sqrt{2\pi y}\big)+H_{k}\Big).\] The Laplacian $\Delta_{\frac{1}{2}}=-\xi_{\frac{3}{2}}\xi_{\frac{1}{2}}$ from Equation \eqref{defops} now annihilates all the holomorphic terms, while the images of the logarithmic term and those with $I_{0}$ amount, via Equations \eqref{Ikjder} and \eqref{derofIkj}, to $-\frac{\theta(z)}{4}$. This modular form thus differs from $2\pi\hat{\mathbf{Z}}_{+}(z)$ by a weight $\frac{1}{2}$ harmonic Maass form without any singularity and supported on square-indices, which must be a constant multiple of $\theta(z)$. Recalling that $\operatorname{Tr}_{0}(1)$ from Equation \eqref{Trd0} equals $\gamma$, our constant term is $H_{k}+\frac{\gamma-\log(16\pi)}{2}$, and comparing it with the vanishing constant term of $2\pi\hat{\mathbf{Z}}_{+}(z)$ (see Equation (1.19) of \cite{[DIT]}) yields the desired result (note that the traces defined in Equation (1.8) of that reference involve division by $2\pi$ while ours do not, so that this statement is consistent with comparing the actual formulae). This proves Theorem \ref{E2klift}.

For Theorem \ref{liftnoc0k} we apply Theorem \ref{Shintot} and Corollary \ref{nhIpk} using the same lattice, recalling the vanishing of the main terms of $\mathcal{I}^{\mathrm{cc}}_{k,L}$ implied by the vanishing of $c(0,k)$. As Equations \eqref{Qnu0} and \eqref{Xidef} evaluate the remaining term from Equation \eqref{Trcc0h} to be the asserted one, and the first term in Equation \eqref{Trcc0h} produces the one mentioned in Remark \ref{constk0}, this proves the (extended) theorem.
\end{proof}
For Remark \ref{E2kodd} one combines the proof of Theorem \ref{E2klift} with the twist from the proof of Theorem \ref{cycjE2}.

\begin{rmk} \label{Qmod}
Using \cite{[Ze3]} and \cite{[Ze7]}, the sum $\sum_{a=0}^{\lfloor p/2 \rfloor}\frac{1}{v^{a}a!}L_{\tau}^{a}\mathcal{I}_{k,L}(\tau,f)$ is a (vector-valued) quasi-modular form of weight $k+\frac{1}{2}$ and depth $\big\lfloor\frac{p}{2}\big\rfloor$, and one checks that the contribution of $\mathcal{I}^{\mathrm{nh}}_{k,L,h}(\tau,f)$ is just $\sum_{0 \leq m\in\mathbb{Z}+Q(h)}\operatorname{Tr}_{m,h}(f)q_{\tau}^{m}$ for every $h \in D_{L}$. Moreover, applying this combination to $\mathcal{I}^{\mathrm{neg}}_{k,L,h}(\tau,f)$, $\mathcal{I}^{\mathrm{c}}_{k,L,h}(\tau,f)$, and $\mathcal{I}^{\mathrm{cc}}_{k,L,h}(\tau,f)$ can be shown, using Lemmas \ref{diffhnu} and \ref{Ikdiff} and Equation \eqref{Jdiff} to replace $h_{l}$ from Equation \eqref{Inegdef}, $J_{l}$ in Equation \eqref{Trcdef}, and $I_{k}$ appearing in Equation \eqref{Trccdef} by $\sum_{a}h_{l-2a}{(-2)^{a}a!}$, $\sum_{a}J_{l-2a}{2^{a}a!}$, and $\sum_{a}I_{k-2a}{2^{a}a!}$ respectively. Moreover, after substituting Equations \eqref{hndef}, \eqref{Jnudef}, and \eqref{expInu}, the respective coefficients of $h_{0}$, $J_{0}$, and $I_{0}$ in these combinations is just the corresponding denominator $\eta^{l}$ or $\eta^{k}$. It would be interesting to investigate these functions further. We also note that as long as $2a \leq k$, the term $L_{\tau}^{a}\mathcal{I}_{k,L}(\tau,f)$ is easily verified to be $\frac{1}{(-4\pi)^{a}}$ times the Shintani lift $\mathcal{I}_{k-2a,L}(\tau,L_{z}^{2a}f)$.
\end{rmk}

\subsection{Orbital Integrals}

In order to prove Theorem \ref{Shintot}, we need to evaluate Equation \eqref{limTIkLh}, which we can do for every $m\in\mathbb{Z}+Q(h)$ separately. Moreover, the integral over $Y_{T}$ can be replaced by an integral over $\mathcal{F}_{T}$ from Equation \eqref{funddomT}, and integrals over (nice) regions in $\mathcal{H}$ can be expressed using the following lemma.
\begin{lem} \label{preim}
Let $\lambda \in V_{\mathbb{R}}$ and $f\in\widetilde{M}^{!}_{2k}(\Gamma)$ of some depth $p$ be given, and take a connected domain $\mathcal{R}\subseteq\mathcal{H}$ with a piecewise smooth positively oriented boundary $\partial\mathcal{R}$. Then we have \[\int_{\mathcal{R}}f(z)\varphi_{k,-1}\big(\sqrt{v}\lambda,z\big)d\mu(z)=\sum_{\nu=0}^{p}\oint_{\partial\mathcal{R}}(L_{z}^{\nu}f)(z)\varphi_{k-\nu-1,\nu}\big(\sqrt{v}\lambda,z\big)dz.\]
\end{lem}

\begin{proof}
We apply Lemma \ref{Stokes} repeatedly, where in the $\nu$\textsuperscript{th} step $f$ is replaced by $L_{z}^{\nu}f$ and $g(z)=\varphi_{k-\nu,\nu-1}\big(\sqrt{v}\lambda,z\big)$. Then Proposition \ref{Lzphi} allows us to take $G(z)=\varphi_{k-\nu-1,\nu}\big(\sqrt{v}\lambda,z\big)$, and the sum ends after $\nu=p$ since $L_{z}^{p+1}f=0$ by assumption. This proves the lemma.
\end{proof}

In view of Equation \eqref{Inhdef}, we will need the following lemma later.
\begin{lem} \label{sumSingb}
Let $\lambda \in V$ with $\iota(\lambda)=1$, $v>0$, and $f\in\widetilde{M}^{!}_{2k}(\Gamma)$ of some depth $p$ be given, and denote $\eta=2\sqrt{2\pi Q(\lambda)v}$. Then we have the equality \[\sum_{b=0}^{\lfloor p/2 \rfloor}\frac{\operatorname{Sing}_{\lambda}(L_{z}^{2b}f)}{(4\pi v)^{b}b!}=i^{k}\big(2\sqrt{Q(\lambda)}\big)^{k-1}\sum_{n\in\mathbb{Z}}\sum_{l=0}^{p}c_{\ell_{\lambda}}(n,l)\mathbf{e}\bigg(\frac{nr_{\lambda}}{\alpha_{\ell_{\lambda}}}\bigg) \phi_{n}\bigg(k-l,T;\frac{2\pi}{\alpha_{\ell_{\lambda}}}\bigg)\frac{\operatorname{He}_{l}(\eta)}{\eta^{l}}.\]
\end{lem}

\begin{proof}
As Corollary \ref{Hermite} presents the coefficient in front of $\operatorname{Sing}_{\lambda}(L_{z}^{\mu}f)$ on the left hand side of the first equality as $\frac{P_{\mu}(0)}{(2\pi v)^{\mu/2}}$, substituting Equation \eqref{FourLf}, with the summation index $l$ replacing $l-\mu$, into Equation \eqref{Singdef}, expresses the left hand side of the first equality as
\begin{align*}
&\sum_{\mu=0}^{p}\frac{P_{\mu}(0)}{(2\pi v)^{\mu/2}}i^{k-\mu}(2\sqrt{Q(\lambda)})^{k-\mu-1}\sum_{l=\mu}^{p}\sum_{n\in\mathbb{Z}}\frac{l!c_{\ell_{\lambda}}(n,l)}{(l-\mu)!}\mathbf{e}\bigg(\frac{nr_{\lambda}}{\alpha_{\ell_{\lambda}}}\bigg) \phi_{n}\bigg((k-\mu)-(l-\mu),T;\frac{2\pi}{\alpha_{\ell_{\lambda}}}\bigg) \\ &=i^{k}(2\sqrt{Q(\lambda)})^{k-1}\sum_{j=0}^{p}\sum_{n\in\mathbb{Z}}c_{\ell_{\lambda}}(n,l)\mathbf{e}\bigg(\frac{nr_{\lambda}}{\alpha_{\ell_{\lambda}}}\bigg)\phi_{n}\bigg(k-l,T;\frac{2\pi}{\alpha_{\ell_{\lambda}}}\bigg) \sum_{\mu=0}^{l}\frac{l!P_{\mu}(0)}{(l-\mu)!(2i\sqrt{2\pi Q(\lambda)v})^{\mu}}.
\end{align*}
Equation \eqref{Appell} and Corollary \ref{Hermite} now express the sum over $\mu$ as $\frac{l!P_{l}(i\eta)}{(i\eta)^{l}}=\frac{\operatorname{He}_{l}(\eta)}{\eta^{l}}$, as desired. This proves the lemma.
\end{proof}

Now, the coefficient of $q^{m}$ in Equation \eqref{limTIkLh} is evaluated for $m<0$, $m>0$ with $\iota(m)=0$, $m>0$ with $\iota(m)=1$, and $m=0$ respectively in the following four propositions.
\begin{prop} \label{intnegm}
For every $h \in D_{L}$ and $0>m\in\mathbb{Z}+Q(h)$ we have the equality \[\lim_{T\to\infty}\!v^{\frac{1-k}{2}}\!\int_{Y_{T}}\!\!f(z)\!\!\sum_{\lambda \in L_{m,h}}\!\!\varphi_{k,-1}\big(\sqrt{v}\lambda,z\big)d\mu(z)\!=\! \sum_{l=k}^{p}\!\frac{4^{k}\sqrt{\pi}|m|^{\frac{k-1}{2}}h_{l}(2\sqrt{2\pi|m|v})\operatorname{Tr}_{m,h}^{(k)}(R_{2k-2l}^{l-k}L_{z}^{l}f)}{\sqrt{2}(4\sqrt{2\pi|m|v})^{l}(l-k)!}.\]
\end{prop}

\begin{proof}
The proof is similar to that of Proposition 3.9 in \cite{[BFIL]}. Remark \ref{decay} gives us the strong decay of $\varphi_{\kappa,\nu}\big(\sqrt{v}\lambda,z\big)$, so that we can take the integral over $z \in Y$. Unfolding, cutting out a small neighborhood $B_{\epsilon}(z_{\lambda})$ for each $\lambda \in \Gamma \backslash L_{m,h}$, and applying Lemma \ref{preim} yields \[\int_{Y}f(z)\sum_{\lambda \in L_{m,h}}\varphi_{k,-1}\big(\sqrt{v}\lambda,z\big)d\mu(z)=\lim_{\epsilon\to0}\sum_{\lambda\in \Gamma \backslash L_{m,h}}\!\frac{-1}{|\Gamma_{\lambda}|}\sum_{l=0}^{p}\oint_{\partial B_{\epsilon}(z_{\lambda})}(L_{z}^{l}f)(z)\varphi_{k-1-l,l}\big(\sqrt{v}\lambda,z\big)dz.\] Substituting Equations \eqref{phidef} and \eqref{Zzw} (with $\zeta=\sigma\sqrt{|m|}$ for $\sigma:=-\operatorname{sgn}\big(\lambda,Z^{\perp}(z_{\lambda})\big)$), and multiplying by $v^{\frac{1-k}{2}}$, shows that the desired left hand side is the sum over $\lambda$ and $l$ of \[\sum_{\lambda\in \Gamma \backslash L_{m,h}}\sum_{l=0}^{p}\frac{-\big(4y_{\lambda}\sigma\sqrt{|m|}\big)^{k-1-l}}{(2\pi)^{(l+1)/2}v^{l/2}|\Gamma_{\lambda}|}\lim_{\epsilon\to0}\oint_{\partial B_{\epsilon}(z_{\lambda})}\frac{(L_{z}^{l}f)(z)A_{z_{\lambda}}(z)^{k-1-l}dz}{\big(1-A_{z_{\lambda}}(z)\big)^{2k-2l-2}}h_{l}\big(2\sigma\sqrt{2\pi mv}\tfrac{1+\epsilon^{2}}{1-\epsilon^{2}}\big).\] Evaluating this limit via Corollary \ref{intofexp}, using the parity from Proposition \ref{distderh}, and applying Equations \eqref{TrQneg} and \eqref{Trmh} yields the desired right hand side. This proves the proposition.
\end{proof}

\begin{prop} \label{intiota0}
If $h \in D_{L}$ and $0<m\in\mathbb{Z}+Q(h)$ with $\iota(m)=0$ then we have \[v^{\frac{1-k}{2}}\lim_{T\to\infty}\int_{Y_{T}}f(z)\sum_{\lambda \in L_{m,h}}\varphi_{k,-1}\big(\sqrt{v}\lambda,z\big)d\mu(z)=\sum_{b=0}^{\lfloor p/2 \rfloor}\frac{\operatorname{Tr}_{m,h}(L_{z}^{2b}f)}{(4\pi v)^{b}b!}.\]
\end{prop}

\begin{proof}
An element $\lambda \in L_{m,h}$ is not perpendicular to $\ell$ for any $\ell\in\operatorname{Iso}(V)$, and its stabilizer $\Gamma_{\lambda}$ is infinite cyclic by Lemma \ref{splithyper}. Therefore the functions $\varphi_{\kappa,\nu}\big(\sqrt{v}\lambda,z\big)$ again decay strongly, via Remark \ref{decay}, towards any cusp $\ell$. We combine this with the usual unfolding argument to express our left hand side as \[v^{\frac{1-k}{2}}\lim_{T\to\infty}\sum_{\lambda\in\Gamma \backslash L_{m,h}}\int_{\mathcal{F}_{\lambda,T}}f(z)\varphi_{k,-1}\big(\sqrt{v}\lambda,z\big)d\mu(z)=v^{\frac{1-k}{2}}\sum_{\lambda\in\Gamma \backslash L_{m,h}}\int_{\mathcal{F}_{\lambda}}f(z)\varphi_{k,-1}\big(\sqrt{v}\lambda,z\big)d\mu(z),\] where $\mathcal{F}_{\lambda}$ is a fundamental domain for the action of $\Gamma_{\lambda}$ on $\mathcal{H}$ and $\mathcal{F}_{\lambda,T}:=\mathcal{F}_{\lambda}\cap\mathcal{H}_{T}$ (this is well-defined modulo $\Gamma$ by Equation \eqref{phimod} and the modularity of $f$). We now remove an $\epsilon$-neighborhood of the geodesic $c_{\lambda}$ from $\mathcal{F}_{\lambda}$ for applying Lemma \ref{preim}, substitute the value of each $\varphi_{k-\nu-1,\nu}$ from Equation \eqref{phidef}, and gather powers of $v$ to write our expression as \[-\sum_{\lambda\in\Gamma \backslash L_{m,h}}\sum_{\nu=0}^{p}\frac{\lim_{\epsilon\to0^{+}}\big(h_{\nu}(\epsilon)-h_{\nu}(-\epsilon)\big)}{(2\pi)^{(\nu+1)/2}v^{\nu/2}}\int_{c_{\lambda}\cap\mathcal{F}_{\lambda}}(L_{z}^{\nu}f)(z)(\lambda,Z(z))^{k-\nu-1}dz.\] But the integral is $\operatorname{Tr}_{\lambda}(L_{z}^{\nu}f)$ from Equation \eqref{Triota0}, summing over $\lambda$ replaces it by $\operatorname{Tr}_{m,h}(L_{z}^{\nu}f)$ from Equation \eqref{Trmh}, and we have $\lim_{\epsilon\to0^{+}}\big(h_{\nu}(\epsilon)-h_{\nu}(-\epsilon)\big)=-\sqrt{2\pi}P_{\nu}(0)$ by Proposition \ref{distderh}. The desired formula now follows from Corollary \ref{Hermite}. This proves the proposition.
\end{proof}

When $\iota(m)=1$ the coefficient $a(\Theta_{k,\ell},m,h,v)$ from \eqref{expwithiota}, as well as the traces from Equations \eqref{Trcdef} and \eqref{Trccdef}, may be non-zero. Recalling the functions $\phi_{n}$ from Equation \eqref{Singdef}, the limit of the corresponding coefficient from Equation \eqref{limTIkLh} is evaluated as follows.
\begin{prop} \label{intiota1}
Let $h \in D_{L}$ and $0<m\in\mathbb{Z}+Q(h)$ with $\iota(m)=1$ be given. Then for large $T>1$ we have
\[\begin{split} v^{\frac{1-k}{2}}&\int_{Y_{T}}f(z)\sum_{\lambda \in L_{m,h}}\varphi_{k,-1}\big(\sqrt{v}\lambda,z\big)d\mu(z)=\sum_{b=0}^{\lfloor p/2 \rfloor}\frac{\operatorname{Tr}_{m,h}(L_{z}^{2b}f)}{(4\pi v)^{b}b!}+\operatorname{Tr}^{\mathrm{c}}_{m,h}(f,v)+ \\ &+\operatorname{Tr}^{\mathrm{cc}}_{m,h}(f,v)-\sum_{\ell\in\Gamma\backslash\operatorname{Iso}(V)}\frac{\varepsilon_{\ell}i^{k}}{\sqrt{N}}a(\Theta_{k,\ell},m,h,v)\sum_{l=0}^{p}c_{\ell}(0,l)\phi_{0}(k-l,T)+O\big(\tfrac{1}{T}\big). \end{split}\]
\end{prop}

\begin{proof}
For $\lambda \in L_{m,h}$ we set $\mathcal{H}_{\lambda,T}:=\mathcal{H}\setminus\big(B_{\epsilon}(\ell_{\lambda})\cup B_{\epsilon}(\ell_{-\lambda})\big)$ for $\epsilon=e^{-2\pi T}$. The usual unfolding argument, the fact that the stabilizer of $\lambda \in L_{m,h}$ is trivial by Lemma \ref{splithyper}, and the decay from Remark \ref{decay} allows us to replace, as in Lemma 5.2 of \cite{[BF2]}, the left hand side by \[v^{\frac{1-k}{2}}\!\sum_{\lambda\in\Gamma \backslash L_{m,h}}\!\int_{H_{T}}f(z)\varphi_{k,-1}\big(\sqrt{v}\lambda,z\big)d\mu(z)=v^{\frac{1-k}{2}}\!\!\sum_{\lambda\in\Gamma \backslash L_{m,h}}\!\int_{H_{\lambda,T}}\!f(z)\varphi_{k,-1}\big(\sqrt{v}\lambda,z\big)d\mu(z)+O\big(\tfrac{1}{T}\big)\] (the error term here is, in fact, much better). The argument from the proof of Proposition \ref{intiota0} expresses the term associated with $\lambda\in\Gamma \backslash L_{m,h}$ as
\begin{equation} \label{intwithbd}
\sum_{b=0}^{\lfloor p/2 \rfloor}\frac{1}{(4\pi v)^{b}b!}\int_{c(\lambda) \cap Y_{T}}\!(L_{z}^{2b}f)(z)(\lambda,Z(z))^{k-1-2b}dz+\sum_{\nu=0}^{p}v^{\frac{1-k}{2}}\int_{\partial \mathcal{H}_{\lambda,T}}\!(L_{z}^{\nu}f)(z)\varphi_{k-\nu-1,\nu}\big(\sqrt{v}\lambda,z\big)dz.
\end{equation}
Next, $\partial\mathcal{H}_{\lambda,T}=\partial B_{\epsilon}(\ell_{\lambda})\cup\partial B_{\epsilon}(\ell_{-\lambda})$ for this $\epsilon$, both with the opposite orientation. In the integral along the first part, as $\sigma_{\ell}^{-1}$ takes $\partial B_{\epsilon}(\ell_{\lambda})$ to $\mathbb{R}+iT$, Equation \eqref{rlambda} expresses the corresponding term in Equation \eqref{intwithbd} as \[-\sum_{\nu=0}^{p}\int_{-\infty}^{\infty}v^{\frac{1-k}{2}}(L_{z}^{\nu}f\mid_{2k-2\nu}\sigma_{\ell_{\lambda}})(x_{\ell_{\lambda}}+iT)\varphi_{k-\nu-1,\nu}\bigg(\sqrt{\tfrac{mv}{N}}\big(\begin{smallmatrix} 1 & -2r_{\lambda} \\ 0 &-1\end{smallmatrix}\big),x_{\ell_{\lambda}}+iT\bigg)dx_{\ell_{\lambda}}.\] Equation \eqref{phipar} thus implies that the integral along $\partial B_{\epsilon}(\ell_{-\lambda})$ is evaluated in the same manner, with $\lambda$ replaced by $-\lambda$, and multiplied by $(-1)^{k}$.

Applying Equation \eqref{phidef}, noting that the pairings with $Z(x_{\ell_{\lambda}}+iT)$ and $Z^{\perp}(x_{\ell_{\lambda}}+iT)$ are $2\sqrt{mv}(x_{\ell_{\lambda}}+iT-r_{\lambda})$ and $\frac{2\sqrt{mv}}{T}(x_{\ell_{\lambda}}-r_{\lambda})$ respectively, we substitute $\eta:=2\sqrt{2\pi mv}$ and $\xi:=\frac{\eta}{T}(x_{\ell_{\lambda}}-r_{\lambda})$ and use Equations \eqref{FourLf} and \eqref{gkldef} and the definition of the Fourier transform to present the latter expression as
\begin{align*}
-&v^{\frac{1-k}{2}}\sum_{\nu=0}^{p}\int_{-\infty}^{\infty}\sum_{n\in\mathbb{Z}}\sum_{l=\nu}^{p}\frac{(-1)^{\nu}l!c_{\ell_{\lambda}}(n,l)}{(l-\nu)!T^{l-\nu}}e^{-2\pi nT/\alpha_{\ell_{\lambda}}}\mathbf{e}\Big(\tfrac{nx_{\ell_{\lambda}}}{\alpha_{\ell_{\lambda}}}\Big)\frac{T^{k-\nu-1}}{(2\pi)^{k/2}}(\xi+i\eta)^{k-\nu-1}h_{\nu}(\xi)dx_{\ell_{\lambda}} \\  =&-v^{\frac{1-k}{2}}\sum_{n\in\mathbb{Z}}\sum_{l=0}^{p}\frac{l!c_{\ell_{\lambda}}(n,l)T^{k-l-1}}{(2\pi)^{k/2}}e^{-2\pi nT/\alpha_{\ell_{\lambda}}}\int_{-\infty}^{\infty}\mathbf{e}\Big(\tfrac{nx_{\ell_{\lambda}}}{\alpha_{\ell_{\lambda}}}\Big)\mathbf{g}_{k-l,l}(\xi;\eta)dx_{\ell_{\lambda}} \\  =&-v^{\frac{1-k}{2}}\sum_{n\in\mathbb{Z}}\sum_{l=0}^{p}\frac{l!c_{\ell_{\lambda}}(n,l)T^{k-l}}{(2\pi)^{k/2}\eta}e^{-2\pi nT/\alpha_{\ell_{\lambda}}}\mathbf{e}\Big(\tfrac{nr_{\lambda}}{\alpha_{\ell_{\lambda}}}\Big)\widehat{\mathbf{g}_{k-l,l}}\Big(-\tfrac{nT}{\alpha_{\ell_{\lambda}}\eta};\eta\Big).
\end{align*}
Propositions \ref{FTt0} and \ref{asympgkl} show, via Equation \eqref{Singdef}, that the contribution of fixed $n$ and $l$ to the latter expression is $\frac{i^{k}\eta^{k-1}}{(2\pi v)^{(k-1)/2}}c_{\ell_{\lambda}}(n,l)\mathbf{e}\big(\frac{nr_{\lambda}}{\alpha_{\ell_{\lambda}}}\big)=i^{k}\big(2\sqrt{m}\big)^{k-1}c_{\ell_{\lambda}}(n,l)\mathbf{e}\big(\frac{nr_{\lambda}}{\alpha_{\ell_{\lambda}}}\big)$ times
\[\begin{cases} \frac{\operatorname{He}_{l}(\eta)}{\eta^{l}}\phi_{n}\Big(k-l,T;\frac{2\pi}{\alpha_{\ell_{\lambda}}}\Big), & n>0\text{ or }n<0\text{ and }l>k, \\ \frac{\operatorname{He}_{l}(\eta)}{\eta^{l}}\phi_{n}\Big(k-l,T;\frac{2\pi}{\alpha_{\ell_{\lambda}}}\Big)-(-1)^{k}\sqrt{2\pi}\big(\frac{2\pi n}{\alpha_{\ell_{\lambda}}}\big)^{l-k}\frac{l!}{(l-k)!}\frac{J_{l}(\eta)}{\eta^{l}}, & n<0\text{ and }l \geq k, \\ \big(\frac{\operatorname{He}_{l}(\eta)}{\eta^{l}}-\frac{\operatorname{He}_{k}(\eta)}{\eta^{k}}\big)\phi_{0}(k-l,T), & n=0\text{ and }l \neq k, \\ \big(\frac{\operatorname{He}_{l}(\eta)}{\eta^{l}}-\frac{\operatorname{He}_{k}(\eta)}{\eta^{k}}\big)\phi_{0}(k-l,T)-\frac{(-1)^{k}k!}{\eta^{k}}\big(I_{k}(\eta)-\tilde{\Omega}_{k}(\eta)\big), & n=0\text{ and }l=k, \\ \end{cases}\] where the error terms from the former proposition go into $O\big(\frac{1}{T}\big)$ and the first term with $l=k$ trivially vanishes.

Now, since $Q(\lambda)=m$ for our $\lambda$, Lemma \ref{sumSingb} shows that the sum over $n$ and $l$ of the first terms gives $\sum_{b=0}^{\lfloor p/2 \rfloor}\frac{\operatorname{Sing}_{\lambda}(L_{z}^{2b}f)}{(4\pi v)^{b}b!}$. Hence these terms, the corresponding ones from the integral along $\partial B_{\epsilon}(\ell_{-\lambda})$, and the first term from Equation \eqref{intwithbd} combine, via Equation \eqref{Triota1}, to $\sum_{b=0}^{\lfloor p/2 \rfloor}\frac{\operatorname{Tr}_{\lambda}(L_{z}^{2b}f)}{(4\pi v)^{b}b!}$, and after summing over all $\lambda\in\Gamma \backslash L_{m,h}$ we get the first asserted term by Equation \eqref{Trmh}. Now fix $\ell\in\Gamma\backslash\operatorname{Iso}(V)$, and Equation \eqref{iotaellmh} implies that in the sum of the remaining terms over $\lambda\in\Gamma \backslash L_{m,h}$, we only get contributions to the integral along $\partial B_{\epsilon}(\ell_{\lambda})$ when $\iota_{\ell}(m,h)=1$, and to the one over $\partial B_{\epsilon}(\ell_{-\lambda})$ if $\iota_{\ell}(m,-h)=1$, with the sign $(-1)^{k}$. Lemma \ref{parLmhiota1} and Remark \ref{rporbs} now imply that for such $\ell$ the sum of $\mathbf{e}\big(\frac{nr_{\lambda}}{\alpha_{\ell_{\lambda}}}\big)$ over $\lambda$ with $\ell_{\lambda}=\ell$ gives $2\varepsilon_{\ell}\sqrt{\frac{m}{N}}\mathbf{e}\big(\frac{nr_{\lambda}}{\alpha_{\ell_{\lambda}}}\big)$ in case $2\varepsilon_{\ell}\sqrt{\frac{m}{N}}$ divides $n$ and 0 otherwise. The resulting sums thus produce the remaining required terms by Equations \eqref{Trcdef}, \eqref{Trccdef}, and \eqref{expwithiota} and the value of $\eta$. This completes the proof of the proposition.
\end{proof}


We can now state and prove the analog of Propositions \ref{intnegm}, \ref{intiota0}, and \ref{intiota1} for $m=0$, again using the expressions from Equations \eqref{expwithiota} and \eqref{Singdef}, and with the trace from Equation \eqref{Trcc0h}.
\begin{prop} \label{intm0}
For an isotropic element $h \in D_{L}$ and large $T>1$ we have \[\begin{split} v^{\frac{1-k}{2}}\int_{Y_{T}}&f(z)\sum_{\lambda \in L_{0,h}}\varphi_{k,-1}\big(\sqrt{v}\lambda,z\big)d\mu(z)=\sum_{b=0}^{\lfloor p/2 \rfloor}\frac{\operatorname{Tr}_{0,h}(L_{z}^{2b}f)}{(4\pi v)^{b}b!}+\operatorname{Tr}^{\mathrm{cc}}_{0,h}(f,v) \\ &-\sum_{\ell\in\Gamma\backslash\operatorname{Iso}(V)}\frac{i^{k}\varepsilon_{\ell}}{\sqrt{N}}a(\Theta_{k,\ell},0,h,v)\sum_{l=0}^{p}c_{\ell}(0,l)\phi_{0}(k-l,T)+O\big(\tfrac{1}{T}\big). \end{split}\]
\end{prop}

\begin{proof}
It is easy to see that $L_{0,h}\setminus\{0\}=\bigcup_{\{\ell\in\operatorname{Iso}(V)\,|\,\iota_{\ell}(0,h)=1\}}[(L+h)\cap\ell]$, a union that respects the $\Gamma$-action. The same unfolding argument from the proofs of Propositions \ref{intiota0} and \ref{intiota1} allows us to write this part of the left hand side as \[\sum_{\{\ell\in\Gamma\backslash\operatorname{Iso}(V)\,|\,\iota_{\ell}(0,h)=1\}}v^{\frac{1-k}{2}}\int_{\Gamma_{\ell}\backslash\mathcal{H}_{T}}f(z)\sum_{\lambda \in (L+h)\cap\ell}\varphi_{k,-1}\big(\sqrt{v}\lambda,z\big)d\mu(z),\] and Remark \ref{decay} implies that replacing $\Gamma_{\ell}\backslash\mathcal{H}_{T}$ by $\Gamma_{\ell}\backslash\mathcal{H}_{\ell,T}$, for $\mathcal{H}_{\ell,T}:=\mathcal{H}_{T} \setminus B_{\epsilon}(\ell)$, produces an error term that is much smaller that $O\big(\frac{1}{T}\big)$\footnote{Also here we actually work with representatives for $\Gamma\backslash\operatorname{Iso}(V)$, but we again allow this abuse of notation.}. As $\ell$ is perpendicular to neither $Z(z)$ nor $Z^{\perp}(z)$ for any $z\in\mathcal{H}$, when we invoke Lemma \ref{preim} as the proof of Proposition \ref{intiota1}, only (negatively oriented) integral along the boundary $\Gamma_{\ell}\backslash\partial B_{\epsilon}(\ell)$ remains. Now, since Equation \eqref{widthdef} shows that the latter maps under $\sigma_{\ell}^{-1}$ onto $(\mathbb{R}/\alpha_{\ell}\mathbb{Z})+iT$, Equation \eqref{komegaellh} expresses the summand associated with $\ell$ as \[-\sum_{\nu=0}^{p}v^{\frac{1-k}{2}}\int_{\mathbb{R}/\alpha_{\ell}\mathbb{Z}}(L_{z}^{\nu}f\mid_{2k-2\nu}\sigma_{\ell})(x_{\ell}+iT)\sum_{0\neq\xi\in\mathbb{Z}+\omega_{\ell,h}} \varphi_{k-1-\nu,\nu}\Big(\sqrt{v}\big(\begin{smallmatrix} 0 & \beta_{\ell}\xi \\ 0 & 0\end{smallmatrix}\big),x_{\ell}+iT\Big)dx_{\ell}+O\big(\tfrac{1}{T}\big).\]

Here after applying Equation \eqref{phidef}, the pairings with $Z(x_{\ell_{\lambda}}+iT)$ and $Z^{\perp}(x_{\ell_{\lambda}}+iT)$ are just $\sqrt{Nv}\beta_{\ell}\xi$ and $\frac{\sqrt{Nv}\beta_{\ell}\xi}{T}$ respectively. Setting $\upsilon:=\frac{\sqrt{2\pi Nv}\beta_{\ell}}{T}$, Equations \eqref{FourLf}, \eqref{gkldef}, \eqref{Gkldef}, and \eqref{Gkleta0} and simple Fourier integration evaluate the main term as
\begin{align*}
&-\sum_{\nu=0}^{p}v^{\frac{1-k}{2}}\int_{\mathbb{R}/\alpha_{\ell}\mathbb{Z}}\sum_{n\in\mathbb{Z}}\sum_{l=\nu}^{p}\frac{(-1)^{\nu}l!c_{\ell}(n,l)}{(l-\nu)!T^{l-\nu}}e^{-2\pi nT/\alpha_{\ell}}\mathbf{e}\bigg(\frac{nx_{\ell}}{\alpha_{\ell}}\bigg)\sum_{0\neq\xi\in\mathbb{Z}+\omega_{\ell,h}}\frac{(T\upsilon\xi)^{k-1-\nu}}{(2\pi)^{k/2}}h_{\nu}(\upsilon\xi)dx_{\ell} \\ &=\!-\sqrt{v}\sum_{l=0}^{p}\frac{l!c_{\ell}(0,l)}{(2\pi v)^{k/2}}\alpha_{\ell}T^{k-1-l}\!\!\!\sum_{0\neq\xi\in\mathbb{Z}+\omega_{\ell,h}}\!\!\!\mathbf{g}_{k-l,l}(\upsilon\xi;0)\!=\!-\sqrt{v}\sum_{l=0}^{p}\frac{l!c_{\ell}(0,l)}{(2\pi v)^{k/2}}\alpha_{\ell}T^{k-1-l}\mathbf{G}_{k-l,l}(\omega_{\ell,h};\upsilon).
\end{align*}
We apply Proposition \ref{Gkleval} and substitute the value of $\upsilon$, Equation \eqref{betaelldef}, and Corollary \ref{Hermite} for $P_{k}(0)$, and after summing over $\ell$ our part of the left hand side takes the form
\begin{align*}
&\sum_{\substack{\ell\in\Gamma\backslash\operatorname{Iso}(V) \\ \iota_{\ell}(0,h)=1}}\sum_{l=0}^{p}\frac{l!c_{\ell}(0,l)}{(2\pi v)^{(k-1)/2}}\beta_{\ell}\varepsilon_{\ell}\big(\sqrt{2\pi Nv}\beta_{\ell}\big)^{k-1-l}\bigg[P_{l}(0)\Phi_{k-l}(\omega_{\ell,h})+Q_{l}(0)\Xi_{k-l}(\omega_{\ell,h})\bigg] \\ &+\frac{i^{k}\operatorname{He}_{k}(0)}{(2\pi v)^{(k-1)/2}}\sum_{\substack{\ell\in\Gamma\backslash\operatorname{Iso}(V) \\ \iota_{\ell}(0,h)=1}}\frac{\beta_{\ell}\varepsilon_{\ell}}{\sqrt{2\pi Nv}\beta_{\ell}}\left(\sum_{\substack{0 \leq l \leq p \\ l \neq k}}{c_{\ell}(0,l)}\tfrac{T^{k-l}}{k-l}-c_{\ell}(0,k)\big(\log\tfrac{\sqrt{2\pi Nv}\beta_{\ell}}{T}+C_{k}\big)\right)\!\!+\!O\big(\tfrac{1}{T}\big),
\end{align*}
where the latter terms with $l>k$ can be artificially added from the error term.

After canceling, writing $\iota_{\ell}(0,h)$ as a multiplier, and evaluating $P_{l}(0)$ using Corollary \ref{Hermite}, Equations \eqref{FourLf} shows that the first terms combine to the desired combination of the traces from Equation \eqref{Tr0hg}. Equation \eqref{Singdef} the expresses the remaining expression as the main term from Equation \eqref{Trcc0h} plus \[-\frac{i^{k}\operatorname{He}_{k}(0)}{(2\pi v)^{k/2}}\sum_{\ell\in\operatorname{Iso}(V)}\iota_{\ell}(0,h)\frac{\varepsilon_{\ell}}{\sqrt{N}}\sum_{l=0}^{p}c_{\ell}(0,l)\phi_{n}(k-l,T),\] which is the last asserted term by Equation \eqref{expwithiota}. This gives the desired right hand side when $h\neq0$, and for $h=0$ there is also the integral involving $\varphi_{k,-1}\big(0,z\big)$ from $0 \in L_{0,0}$. But Equation \eqref{phidef} evaluates it as $0^{k}h_{-1}(0)=\delta_{k,0}$, and since in weight 0 we have, by Equation \eqref{Fourg}, the bound \[\int_{Y}^{\reg}f(z)d\mu(z)-\int_{Y_{T}}f(z)d\mu(z)=\sum_{\ell\in\Gamma\backslash\operatorname{Iso}(V)}\int_{T}^{\infty}\alpha_{\ell}\sum_{l=0}^{p}\frac{c_{\ell}(0,l)}{y_{\ell}^{l+2}}=O\big(\tfrac{1}{T}\big),\] we indeed get the remaining term from Equation \eqref{Trcc0h}. This proves the proposition.
\end{proof}

\noindent\textsc{Fachbereich Mathematik, AG 5, Technische Universit\"{a}t
Darmstadt, Schlossgartenstrasse 7, D-64289, Darmstadt, Germany}

\noindent E-mail address: li@mathematik.tu-darmstadt.de

\medskip

\noindent\textsc{Einstein Institute of Mathematics, the Hebrew University of Jerusalem, Edmund Safra Campus, Jerusalem 91904, Israel}

\noindent E-mail address: zemels@math.huji.ac.il


\begin{thebibliography}{}{}

\bibitem[ANS]{[ANS]} Alfes-Neumann, C., Schwagenscheidt, M., \textsc{Shintani Theta Lifts of Harmonic Maass Forms}, to appear in Trans. Amer. Math. Soc., https://doi.org/10.1090/tran/8265.
\bibitem[AE]{[AE]} Alfes, C., Ehlen, S., \textsc{Twisted Traces of CM values of Harmonic Weak Maass Forms}, J. Number Theory, vol 133 issue 6, 1827--1845 (2013).
\bibitem[Bo]{[Bo]} Borcherds, R. E., \textsc{Automorphic Forms with Singularities on Grassmannians}, Invent. Math., vol. 132, 491--562 (1998).
\bibitem[BGK]{[BGK]} Bringmann, K., Guerzhoy, P., Kane, B., \textsc{Shintani Lifts and Fractional Derivatives for Harmonic Weak Maass Forms}, Adv. Math., vol 255, 641–-671, (2014).
\bibitem[BF1]{[BF1]} Bruinier, J. H., Funke, J., \textsc{On Two Geometric Theta Lifts}, Duke Math J., vol 125 no. 1, 45--90 (2004).
\bibitem[BF2]{[BF2]} Bruinier, J. H., Funke, J., \textsc{Traces of CM Values of Modular Functions}, J. Reine Angew. Math., vol 594, 1--33 (2006).
\bibitem[BFI]{[BFI]} Bruinier, J. H., Funke, J., Imamo\u{g}lu, \"{O}, \textsc{Regularized Theta Liftings and Periods of Modular Functions}, J. reine angew. Math., vol 703, 43--93 (2015).
\bibitem[BFIL]{[BFIL]} Bruinier, J. H., Funke, J., Imamo\u{g}lu, \"{O}, Li, Y., \textsc{Modularity of Generating Series of Winding Numbers}, Res. Math. Sci., vol 5 paper 8 (2018).
  \bibitem[BGHZ]{[BGHZ]} Bruinier, J., van der Geer, G., Harder, G., Zagier, D., \textsc{The 1-2-3 of Modular Forms}, Universitext (2008).
\bibitem[BO]{[BO]} Bruinier, J. H., Ono, K., \textsc{Heegner Divisors, $L$-Functions, and Harmonic Weak Maass Forms}, Ann. of Math., vol 172, 2135--2181 (2010).
\bibitem[DI]{[DI]} Duke, W., Imamo\u{g}lu, \"{O}, \textsc{A Converse Theorem and the Saito--Kurokawa Lift}, Int. Math. Res. Not., vol 1996 issue 7, 357--355 (1996).
\bibitem[DIT]{[DIT]} Duke, W., Imamo\u{g}lu, \"{O}, T\'{o}th, \'{A}. \textsc{Cycle Integrals of the $j$-Function and Mock Modular Forms}, Ann. of Math., vol 173 issue 2, 947--981 (2011).
\bibitem[EMOT]{[EMOT]} Erd\'{e}lyi, A., Magnus, W., Oberhettinger, F., Tricomi, F., \textsc{Higher Transcendetal Functions}, McGraw--Hill (1953).
\bibitem[FH]{[FH]} Funke, J., Hofmann, E.,\textsc{The Construction of Green Currents and Singular Theta Lifts for Unitary Groups}, to appear in Trans. Amer. Math. Soc., https://doi.org/10.1090/tran/8289 (2020).
\bibitem[GKZ]{[GKZ]} Gross, B., Kohnen, W., Zagier, D., \textsc{Heegner Points and Derivatives of $L$-Series, II}, Math. Ann., vol 278 no. 1--4, 497--562 (1987).
\bibitem[HZ]{[HZ]} Hirzebruch, F., Zagier, D., \textsc{Intersection Numbers of Curves on Hilbert Modular Surfaces and Modular Forms of Nebentypus}, Invent. Math., vol. 36, 57--114 (1976).
\bibitem[JKK1]{[JKK1]} Jeon, D., Kang, S.-Y., Kim, C. H., \textsc{Weak Maass–Poincaré Series and Weight 3/2 Mock Modular Forms}, J. Number Theory, vol 133 issue 8, 2567--2587 (2013).
\bibitem[JKK2]{[JKK2]} Jeon, D., Kang, S.-Y., Kim, C. H., \textsc{Cycle Integrals of a Sesqui-Harmonic Maass Form of Weight Zero}, J. Number Theory, vol 141, 92--108 (2013).
\bibitem[K]{[K]} Kohnen, W., \textsc{Modular Forms of Half-Integral Weight on $\Gamma_{0}(4)$}, Math. Ann., vol 248, 249–-266 (1980).
\bibitem[L]{[L]} Li, Y., \textsc{Average CM-values of Higher Green's Function and Factorization}, pre-print, https://arxiv.org/abs/1812.08523 (2018).
\bibitem[LZ]{[LZ]} Li, Y., Zemel, S., \textsc{Shimura Lift of Weakly Holomorphic Modular Forms}, Math. Z., vol 290, 37--61 (2018).
\bibitem[MR]{[MR]} Martin, F., Royer, E., \textsc{Formes Modulaires et P\'{e}riodes}, Formes Modulaires et Transcendance, S\'{e}min. Congr., vol. 12, Soc. Math. France, Paris, 1–-117 (2005).
\bibitem[Sch]{[Sch]} Scheithauer, N. R., \textsc{The Weil Representation of $SL_{2}(\mathbb{Z})$ and Some Applications}, Int. Math. Res. Not., no. 8, 1488--1545 (2009).
\bibitem[Sh]{[Sh]} Shimura, G., \textsc{On Modular Forms of Half Integral Weight}, Ann. of Math., vol. 97, 440--481 (1973).
\bibitem[Sn]{[Sn]} Shintani, T., \textsc{On the Construction of Holomorphic Cusp Forms of Half-Integral Weight}, Nagoya Math. J., vol 58, 83--126 (1975).
\bibitem[Str]{[Str]} Str\"{o}mberg, F., \textsc{Weil Representations Associated to Finite Quad-ratic Modules}, Math. Z., vol 275 issue 1, 509--527 (2013).
\bibitem[Za]{[Za]} Zagier, D., \textsc{Traces of Singular Moduli}, In \emph{Motives, Polylogarithms and Hodge theory, Part I} (Irvine, CA, 1998), volume 3 of Int. Press Lect. Ser., pages 211-244. Int. Press, Somervillle, MA (2002).
\bibitem[Ze1]{[Ze1]} Zemel, S., \textsc{A $p$-adic Approach to the Weil Representation of Discriminant Forms Arising from Even Lattices}, Math. Ann. Qu\'{e}bec, vol 39 issue 1, 61--89 (2015).
\bibitem[Ze2]{[Ze2]} Zemel, S., \textsc{A Gross--Kohnen--Zagier Type Theorem for Higher-Codimensional Heegner Cycles}, Res. Number Theory, vol 1 paper 25, 1--44 (2015).
\bibitem[Ze3]{[Ze3]} Zemel, S. \textsc{On Quasi-Modular Forms, Almost Holomorphic Modular Forms, and the Vector-Valued Modular Forms of Shimura}, Ramanujan J., vol 37 issue 1, 165--180 (2015).
\bibitem[Ze4]{[Ze4]} Zemel, S., \textsc{Regularized Pairings of Meromorphic Modular Forms and Theta Lifts}, J. Number Theory, vol 162, 275--311 (2016).
\bibitem[Ze5]{[Ze5]} Zemel, S., \textsc{Generalized Riordan Groups and Operators on Polynomials}, Linear Algebra Appl., vol 494, 286--308 (2016).
\bibitem[Ze6]{[Ze6]} Zemel, S., \textsc{Normalizers of Congruence Groups in $SL_{2}(\mathbb{R})$ and Automorphisms of Lattices}, Int. J. Number Theory, vol 13 issue 5, 1275--1300 (2017).
\bibitem[Ze7]{[Ze7]} Zemel, S., \textsc{Shimura's Vector-Valued Modular Forms, Weight Changing Operators, and Laplacians}, Res. Number Theory, vol 4 issue 1, paper 8 (2018).

\end{thebibliography}
\end{document}